\documentclass[11pt,a4paper]{amsart}
\usepackage{enumerate}
\usepackage{pb-diagram}
\usepackage{microtype}
\usepackage{amsmath}
\usepackage{appendix}
\usepackage{url}
\usepackage{amssymb, pb-diagram}
\usepackage[all]{xy}
\usepackage{graphicx} 
\usepackage[dvipsnames]{xcolor}
\usepackage{chapterbib}
\usepackage{epsfig}
\usepackage{amsfonts}
\usepackage{amssymb}
\usepackage{amsmath}   
\usepackage{amsthm}
\usepackage{color}
\usepackage{dsfont}
\usepackage{latexsym}
\usepackage{hyperref}
\usepackage{mathabx}
\usepackage{graphicx}
\usepackage{tikz-cd}
\usepackage{mathabx}
\usepackage{enumitem}
\usepackage{lipsum}
\usepackage{adjustbox}

\makeatletter
\let\@wraptoccontribs\wraptoccontribs
\makeatother

\input xy
\xyoption{all}
\usepackage{pgf, tikz}
\usetikzlibrary{arrows,shapes}

\renewcommand{\epsilon}{\varepsilon}
\renewcommand{\setminus}{\smallsetminus}
\renewcommand{\emptyset}{\varnothing}

\numberwithin{equation}{section}
\newtheorem{theorem}[equation]{Theorem}
\newtheorem{observation}[equation]{Observation}
\newtheorem{proposition}[equation]{Proposition}
\newtheorem{corollary}[equation]{Corollary}
\newtheorem{lemma}[equation]{Lemma}

\theoremstyle{definition}
\newtheorem{example}[equation]{Example}

\newtheorem{definition}[equation]{Definition}

\newtheorem*{question}{Question}

\newtheorem{remark}[equation]{Remark}

\newcommand{\mH}{\mathcal {H}}

\newcommand{\mS}{\mathcal S}

\newcommand{\Z}{\mathbb Z}

\newcommand{\F}{\mathbb F}
\newcommand{\R}{\mathbb R}
\newcommand{\C}{\mathbb C}

\newcommand{\GL}{\operatorname{GL}}

\newcommand{\Link}{\lk}
\newcommand{\DownLink}{\lk_{-}}
\newcommand{\UpLink}{\lk_{+}}
\newcommand{\DescLink}[1][]{\lk_{#1}^{\downarrow}}
\newcommand{\DescDownLink}{\DescLink[-]}
\newcommand{\DescUpLink}{\DescLink[+]}
\newcommand{\intersect}{\cap}
\newcommand{\union}{\cup}

\newcommand{\defect}{\mathfrak{d}}
\DeclareMathOperator{\join}{\ast}
\DeclareMathOperator{\bigjoin}{\Asterisk}
\newcommand{\acx}{\mathcal{A}}
\newcommand{\bacx}{\acx^\circ}

\newcommand{\cohom}[3]{H^{{\raise1pt\hbox{$\scriptstyle#1$}}}(#2\>\!,#3)}
\newcommand{\tatecohom}[3]%
  {\widehat H^{{\raise1pt\hbox{$\scriptstyle#1$}}}(#2\>\!,#3)}

\newcommand{\Cohom}[3]%
  {H^{{\raise1pt\hbox{$\scriptstyle#1$}}}\big(#2\>\!,#3\big)}
\newcommand{\Tatecohom}[3]%
  {\widehat H^{{\raise1pt\hbox{$\scriptstyle#1$}}}\big(#2\>\!,#3\big)}

\newcommand{\homol}[3]{H_{{\lower1pt\hbox{$\scriptstyle#1$}}}(#2\>\!,#3)}
\newcommand{\homolog}[2]{H_{{\lower1pt\hbox{$\scriptstyle#1$}}}(#2)}

\newcommand{\lra}{\longrightarrow}

\renewcommand{\implies}{\Rightarrow}
\newcommand{\da}{\downarrow}

\newcommand{\Hom}{\operatorname{Hom}}

\newcommand{\Map}{\operatorname{Map}}

\newcommand{\frakF}{\mathfrak{F}}
\newcommand{\frakG}{\mathfrak{G}}

\catcode`\@=11

\newcommand{\calO}{\mathcal O}
\newcommand{\B}{\mathcal B}

\newcommand{\ModOFG}{\mathop{{\operator@font
Mod\text{-}}\calO_{\frakF}G}}
\newcommand{\OFGMod}{\mathop{\calO_{\frakF}G\text{-}{\operator@font
Mod}}}
\newcommand{\ModOGG}{\mathop{{\operator@font
Mod\text{-}}\calO_{\frakG}G}}
\newcommand{\OGGMod}{\mathop{\calO_{\frakG}G\text{-}{\operator@font
Mod}}}

\DeclareMathOperator{\Aut}{Aut}

\DeclareMathOperator{\Homeo}{Homeo}

\DeclareMathOperator{\lk}{lk}

\newcommand{\Sym}{\operatorname{Sym}}

\newcommand{\Fall}{\frakF_{\operator@font all}}
\newcommand{\Ffin}{\frakF_{\operator@font fin}}
\newcommand{\Fvc}{\frakF_{\operator@font vc}}
\newcommand{\Fic}{\frakF_{\operator@font ic}}
\newcommand{\Ffg}{\frakF_{\operator@font fg}}
\newcommand{\Fpc}{\frakF_{\operator@font pc}}
\newcommand{\Fab}{\frakF_{\operator@font ab}}
\newcommand{\Fvpc}{\frakF_{\operator@font vpc}}
\newcommand{\Fvab}{\frakF_{\operator@font vab}}

\newcommand{\bS}{{\mathbb{S}}}

\newcommand{\chain}[1]{Ch({#1})}

\newcommand{\sph}[1]{\CS({#1})}
\newcommand{\nssph}[1]{\CS^{ns}({#1})}

\DeclareMathOperator{\htb}{HTB}
\DeclareMathOperator{\htbs}{HTB_e}
\DeclareMathOperator{\bhtbs}{HTB_e^\circ}
\DeclareMathOperator{\tha}{TH}
\DeclareMathOperator{\thas}{TH_e}

\DeclareMathOperator{\dball}{\mathbb B}
\catcode`\@=12

\DeclareMathOperator{\SO}{SO}

\DeclareMathOperator{\id}{id}

\renewcommand{\coprod}%
{\mathop{\rotatebox[origin=c]{180}{$\displaystyle\prod$}}\limits}

\let\isom\iso

   \newcommand{\CS}{{\mathcal{S}}}   
\newcommand{\mP}{\mathcal P}
\newcommand{\mB}{\mathcal B}

\newcommand{\mC}{\mathcal C}

\DeclareMathOperator{\Diff}{Diff}
\DeclareMathOperator{\Fr}{Fr}
\DeclareMathOperator{\Emb}{Emb}
\DeclareMathOperator{\Out}{Out}

\numberwithin{equation}{section}

\newcommand{\notion}[1]{\emph{#1}}

\begin{document}


\title[Asymptotic mapping class groups of Cantor manifolds]{Asymptotic mapping class groups of Cantor manifolds and their finiteness properties}


\author{J. Aramayona}
\address{Javier Aramayona: Instituto de Ciencias Matem\'aticas (ICMAT). Nicol\'as Cabrera, 13--15. 28049, Madrid, Spain}
\thanks{J.A. was supported by grant PGC2018-101179-B-I00, and acknowledges financial support from the Spanish Ministry of Science, Innovation, and Universities, through the ``Severo Ochoa Programme for Centres of Excellence in R\&D (CEX2019-000904-S and CEX2023-001347-S)''}

\author{K.-U. Bux}
\address{Kai-Uwe Bux, Jonas Flechsig: Fakult\"at f\"ur Mathematik, Universit\"at Bielefeld, D-33501 Bielefeld, Germany}
\thanks{K.-U. B. and J.F. were funded by the Deutsche Forschungsgemeinschaft (DFG, German Research Foundation) – 426561549. K.-U. B. also acknowledges support by the DFG via project 491392403 – TRR 358.}

\author{J. Flechsig}

\author{N. Petrosyan}
\address{Nansen Petrosyan: School of Mathematical Sciences, University of Southampton, University Road, Southampton SO17 1BJ,United Kingdom}
\author{X. Wu}
\address{Xiaolei Wu: Shanghai Center for Mathematical Sciences, Jiangwan Campus, Fudan University, No.2005 Songhu Road, Shanghai, 200438, P.R. China}

\contrib[With an appendix by]{O. Randal-Williams}
\address{Oscar Randal-Williams: Centre for Mathematical Sciences\\
Wilberforce Road\\
Cambridge CB3 0WB\\
UK}



\curraddr{}
\email{}
\curraddr{}
\email{}


\subjclass[2010]{}

\date{\today}


\begin{abstract} 
 We prove that the infinite family of asymptotic mapping class groups of surfaces defined by Funar--Kapoudjian \cite{FK04,FK09} and Aramayona--Funar \cite{AF17} are of type $F_\infty$, thus answering \cite[Problem 3]{FKS12} and \cite[Question 5.32]{AV20}. 

As it turns out, this result is a specific instance of a much more general theorem which allows to deduce that asymptotic mapping class groups of {\em Cantor manifolds}, also introduced in this paper, are of type $F_\infty$, provided that the underlying manifolds satisfy some general hypotheses.

As important examples, we will obtain $F_\infty$ asymptotic  mapping class groups that contain, respectively, the mapping class group of every compact surface with non-empty boundary, the automorphism group of every free group of finite rank, or infinite families of arithmetic groups. 

In addition, for certain types of manifolds, the homology of our asymptotic mapping class groups coincides  with the {\em stable homology} of the relevant mapping class groups, as studied by Harer \cite{Har85} and Hatcher--Wahl \cite{HW10}.
\end{abstract}

\maketitle

\section{Introduction} 
Asymptotic mapping class groups of surfaces were introduced by Funar--Kapoudjian \cite{FK04} with the original motivation of finding natural discrete analogues of the diffeomorphism group of the circle. Since then, asymptotic mapping class groups of surfaces have received considerable attention from multiple perspectives, as we now explain. 

A first piece of motivation for the study of asymptotic mapping class groups is that they are (countable) subgroups of mapping class groups of infinite-type surfaces, now commonly known as {\em big mapping class groups}. In fact, a striking result of Funar--Neretin \cite[Corollary 2]{FN18} asserts that the {\em smooth mapping class group} of a closed surface minus a Cantor set coincides with the group of {\em half-twists}, a particular instance of an asymptotic mapping class group defined by Funar--Nguyen \cite{FN16} and subsequently studied in \cite{AF17}. Moreover, for the same class of surfaces, the group of half-twists (as well as some other groups of the same flavour) is dense in the topological mapping class group, see \cite[Theorem 1.3]{AF17}, and also \cite[Theorem 3.19]{SW21b}.

A second source of interest stems from the classic theme of {\em stable homology}: indeed, a result of Funar--Kapoudjian \cite[Theorem 3.1]{FK09} proves that the (rational) homology of an ``infinite-genus'' asymptotic mapping class group, which contains the mapping class group of every compact surface with non-empty boundary, is precisely the (rational) stable homology of the mapping class group, as computed by Harer \cite{Har85}; see also \cite[Theorem 1.4]{AF17} for an analogous result in the case of finite-genus surfaces. 

Yet another example of the significance of asymptotic mapping class groups arises from the fact that they are instances of {\em Thompson-like} groups, being extensions of direct limits of mapping class groups of compact surfaces by Thompson groups (e.g. \cite{FK04,FK09,AF17,GLU20}) and related groups, including Houghton groups \cite{Funar07, GLU20, GLU21}. As such, asymptotic mapping class groups are intimately related to other well-known Thompson-like groups, with the {\em braided Thompson groups} of Brin \cite{Br07} and Dehornoy \cite{De05} as notable examples. 

A common feature of Thompson-like groups is that they often are of type $F_\infty$, see e.g.  \cite{Far05,BFM+16,Thu17,WZ18,GLU20,SW21a,AACSW22}; here, recall that a group is said to be of type $F_n$ it has a classifying space with finite $n$-skeleton, and is of type $F_\infty$ if it is of type $F_n$ for all $n$. 

In stark contrast, the situation with asymptotic mapping class groups of surfaces remained mysterious. Funar--Kapoudjian proved in \cite{FK04} that the genus-zero asymptotic mapping class group is finitely presented, and Aramayona--Funar proved the analog for surfaces of positive genus. Funar and Kapoudjian \cite{FK09,Funar07} also proved the finite presentability of related asymptotic mapping class groups of surfaces that are obtained by thickening  planar trees, which include the so-called {\em braided Houghton groups}; very recently Genevois--Lonjou--Urech \cite{GLU20} have determined the finite properties of these groups, which depend on the local branching of the underlying tree. In this direction, the following question is implicit in \cite{FK04,FK09}, and appears explicitly in \cite[Problem 3]{FKS12} and \cite[Question 5.32]{AV20}:  

\begin{question}[\cite{FK04,FKS12,AV20}]
Study the finiteness properties of asymptotic mapping class groups. Are they of type $F_\infty$? 
\label{qu:main}
\end{question}

Our first objective is to answer the above question for the infinite family of asymptotic mapping class groups of surfaces defined by Funar, Kapoudjian, and the first author in the series of papers \cite{FK04,FK09,AF17} .


\begin{theorem}
The asymptotic mapping class groups of surfaces defined in  \cite{FK04,FK09,AF17} are  of type $F_\infty$. 
\label{thm:surfacesintro}
\end{theorem}

However, the main aim of this paper is to develop a much more general framework that will allow us to define asymptotic mapping class groups of {\em Cantor manifolds} of arbitrary dimension, not just surfaces. By translating the classical {\em topological} condition of  \cite{FK04,FK08,FK09,AF17,GLU20,GLU21} for a mapping class to be asymptotically rigid into an {\em algebraic} one, we will be able to prove that the asymptotic mapping class groups so obtained are of type $F_\infty$ under general hypotheses on the underlying manifolds. As interesting special cases, we will obtain examples of asymptotic mapping class groups that are of type $F_\infty$ and contain, respectively, the mapping class group of every compact surface with non-empty boundary, the automorphism group of every free group of finite rank, or infinite families of airthmetic groups.

We now offer an abridged overview of the main definitions and results, postponing their precise versions until Section \ref{prelim}.

\subsection*{Cantor manifolds and their asymptotic mapping class groups} 


In what follows, we will work in the smooth category. Let $O$ and $Y$ be compact, connected, oriented manifolds of the same dimension $n\ge 2$, and assume $Y$ is  closed. Let $Y^d$ be the manifold obtained by removing $d+1\ge 3$ disjoint open balls  from $Y$.  Given this data and an integer $r\ge 1$, we construct the {\em Cantor manifold} $\mC_{d,r}(O,Y)$ by first removing a set of $r$ balls from $O$, then gluing a copy of $Y^d$ to each of the resulting boundary components, and finally inductively gluing, according to a fixed diffeomorphism, copies of $Y^d$ to the resulting manifold, in a tree-like manner; see Figure \ref{fig:constr_Cantor_mfd} for an explicit example.

\begin{figure}
\includegraphics[scale=0.55]{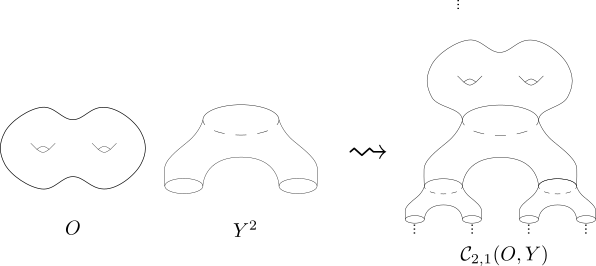}
\caption{Construction of the Cantor manifold $C_{2,1}(O,Y)$, for $O$ a closed surface of genus 2 and $Y$ a sphere (so that $Y^2$ is a pair of pants).} 
\label{fig:constr_Cantor_mfd}
\end{figure}

 Each of the glued copies of $Y^d$ is called a {\em piece}. A connected submanifold $M\subset \mC_{d,r}(O,Y)$ is {\em suited} if it is the union of $O$ and finitely many pieces. A boundary sphere of $M$ which does not come from $O$ is a {\em suited sphere}. 
 
\begin{remark} Throughout, we will restrict our attention to Cantor manifolds that satisfy three natural properties, namely the {\em cancellation, inclusion, and intersection} properties; see  Definitions \ref{def:cancel_prop}, \ref{def:inclusion property} and \ref{def:intersection property}, respectively. As we will see in Appendix \ref{sec:intersection}, these properties are satisfied in all our applications. 
\end{remark}

 For each piece of $\mC_{d,r}(O,Y)$, we fix a preferred diffeomorphism to a ``model'' copy of $Y^d$, and refer to the collection of these diffeomorphisms as the {\em preferred rigid structure} on  $\mC_{d,r}(O,Y)$.  
 The {\em asymptotic mapping class group} $\mB_{d,r}(O,Y)$ is the group of  (proper) isotopy classes  of self-diffeomorphisms of $\mC_{d,r}(O,Y)$ that {\em preserve} the preferred rigid structure outside some suited submanifold whose image is also suited -- see  Section \ref{prelim} for details. Generalizing the situation in two dimensions \cite{AF17,Br07,De05,FK04,FK09,SW21b}, in Proposition \ref{prop-rel-htg} we will see that there is a short exact sequence 
\begin{equation}
\label{eq:sesB}
    1\to \Map_c(\mC_{d,r}(O,Y)) \to \mB_{d,r}(O,Y) \to V_{d,r} \to 1,
\end{equation}
where $V_{d,r}$ denotes the standard Higman--Thompson group of degree $d$ and with $r$ roots, while $\Map_c(\mC_{d,r}(O,Y))$ stands for the {\em compactly-supported mapping class group}, which is a direct limit of mapping class groups of suited submanifolds of $\mC_{d,r}(O,Y)$. The homology of $\Map_c(\mC_{d,r}(O,Y))$ sometimes coincides with that of  the  so-called {\em stable mapping class group}; see Section~\ref{sec:homology}. 

\begin{remark}
    Let $D$ be the 2-dimensional disk.  Skipper--Wu proved that  $\mB_{d,r}(D,\mathbb S^2)$ is  isomorphic to the oriented ribbon Higman--Thompson groups $RV^{+}_{d,r}$ \cite[Theorem 3.24]{SW21b} and the braided Higman--Thompson groups $bV_{d,r}$ naturally sit inside it \cite[Introduction]{SW21a}. However, note that, since groups $bV_{d,r}$ are centerless but the Dehn twist along the boundary of $\mathcal{C}_{d,r}(D,\mathbb S^2)$ lies in the center of $\mB_{d,r}(D,\mathbb S^2)$, the two groups are actually not isomorphic.
\end{remark}

\subsection*{Statement of results}
In Section \ref{sec:Stein-complex} we will construct an infinite-dimensional contractible cube complex $ \mathfrak X_{d,r}(O,Y)$ on which $\mB_{d,r}(O,Y)$ acts cellularly, similar in spirit to the classical {\em Stein--Farley} complexes for Higman--Thompson groups \cite{Ste92,Far03}, and inspired by the complex constructed in \cite{GLU20}. In short, the vertices of $\mathfrak X_{d,r}(O,Y)$ are pairs $(M,f)$, where $M$ is a suited submanifold 
and $f\in \mB_{d,r}(O,Y)$, subject to  certain equivalence relation. The group $\mB_{d,r}(O,Y)$ acts on $\mathfrak X_{d,r}(O,Y)$ by multiplication in the second component. Our first result is as follows:

\begin{theorem} (Section \ref{sec:Stein-complex}) For any Cantor manifold  $\mC_{d,r}(O,Y)$, the complex $\mathfrak X_{d,r}(O,Y)$ is an infinite dimensional contractible cube complex on which $\mB_{d,r}(O,Y)$ acts cellularly, in such way that cube stabilizers are finite extensions of mapping class groups of suited submanifolds. 
\label{thm:maincomplex}
\end{theorem}

The complex $\mathfrak X_{d,r}(O,Y)$ is equipped with a {\em discrete Morse function} $h:\mathfrak X_{d,r}(O,Y) \to \mathbb R$ (which counts the number of pieces of the given suited submanifold), and gives rise  to a filtration of $\mathfrak X_{d,r}(O,Y)$ by cocompact sets.
 In Section \ref{sec:desc_lk_piece} we will show that the {\em descending link} (with respect to $h$) of every vertex is a {\em complete join} over a certain combinatorial {\em complex of pieces} $\mP_d(M,A)$ built from isotopy classes of embedded copies of $Y^d$ inside
a suited submanifold $M$ (which depends on the vertex) with suited spheres $A$; in particular, the two complexes have the same connectivity properties. The piece complex is strongly related (sometimes equal, even) to the complexes of {\em destabilizations} that appear in the homology stability results of Hatcher--Vogtmann \cite{HV17}, Hatcher--Wahl \cite{HW05} and Galatius--Randal-Williams \cite{GRW18}. Combining Theorem \ref{thm:maincomplex}, a celebrated result of Brown \cite{Br87}, plus classical arguments in discrete Morse theory (discussed in Appendix \ref{sec:connectivitytools}), we obtain:

\begin{theorem} (Section \ref{sec:Stein-complex})
Let $\mC_{d,r}(O,Y)$ be a Cantor manifold with the cancellation, inclusion and intersection properties. Assume that: 
\begin{enumerate}
    \item The mapping class group of every suited submanifold is of type $F_\infty$; 
    \item For every suited submanifold $M$ with $p$ pieces, $\mP_d(M,A)$ is $m(p)$-connected, where $m$ tends to infinity as $p$ does. 
\end{enumerate}
Then $\mB_{d,r}(O,Y)$ is of type $F_\infty$. Moreover, if $\mP_d(M,A)$ is flag for every $M$, then $\mathfrak X_{d,r}(O,Y)$ is ${\rm CAT}(0)$. 
\label{thm:metatheorem}
\end{theorem}

\subsection{Applications.} We now give some applications of the results above to specific classes of manifolds, all of which satisfy the cancellation, inclusion, and intersection properties; see Section \ref{prelim} and Appendix \ref{sec:intersection}. 

\subsubsection{Surfaces}
\label{subsec:surfaces}
We first treat the classical case of asymptotic mapping class groups of surfaces, discussed earlier on. 
 Suppose $O$ and $Y$ are surfaces, with $Y$ either a sphere or a torus. It is well-known that  mapping class groups of compact surfaces are of type $F_\infty$ \cite{Harv}. Also, building up on a result of Hatcher--Vogtmann \cite{HV17} we will  prove in Section \ref{sec:2dim} that $\mP_d(M,A)$ has the desired connectivity properties. Thus, we have:  

\begin{theorem} (Section \ref{sec:2dim})
Let $O$ be a compact surface, and $Y$ diffeomorphic to either $\bS^2$ or $\bS^1 \times \bS^1$. For every $d\ge 2$ and $r\ge 1$,  $\mB_{d,r}(O,Y)$ is of type $F_\infty$. 
\label{cor:surfaces} 
\end{theorem}



We stress that the groups $\mB_{d,r}(O,\mathbb S^2)$ coincide with the groups $\mB V_{d,r}(O)$ defined by Skipper--Wu in \cite[Section 3]{SW21b}, and hence our above result implies that their groups are of type $F_\infty$.

Second, as discussed earlier, our family of groups contains the groups $\mathcal B_g$ of Aramayona--Funar--Kapoudjian \cite{FK04,FK09,AF17}. More concretely, the group 
$\mB_0$ of  \cite{FK04} is $\mB_{2,1}(\mathbb S^2, \mathbb S^2)$; the group $\mB_g$ of \cite{AF17} is $\mB_{2,1}(S_g, \mathbb S^2)$, where $S_g$ denotes the  closed surface of genus $g$; and the group $\mB_\infty$ of \cite{FK09} is $\mB_{2,1}(\mathbb S^2, \mathbb S^1 \times \mathbb S^1)$; see Proposition \ref{prop:ARisotopy}. Therefore, we see Theorem \ref{thm:surfacesintro} above is a consequence of Theorem \ref{thm:metatheorem} above.




\subsubsection{3-manifolds}
\label{subsec:3-mfds} In Section \ref{sec:3dim}, we will see that our methods also apply when $O\cong \bS^3$ and $Y$ is diffeomorphic to $\bS^3$ or $\bS^2\times\bS^1$.  The combination of various well-known results implies that  the mapping class group of every suited submanifold of $C_{d,r}(O,Y)$ is of type $F_\infty$; see Section \ref{sec:mcgs}. Moreover, extending a result of Hatcher--Wahl \cite{HW05}, we will be able to deduce that the piece complex has the desired connectivity properties.  Thus, we will obtain: 

\begin{theorem} (Section \ref{sec:3dim})
Let $O\cong \mathbb \bS^3$, and $Y$ diffeomorphic to either $\bS^3$ or $\mathbb S^2 \times \mathbb S^1$. For every $d\ge 2$ and $r\ge 1$, $\mB_{d,r}(O,Y)$ is of type $F_\infty$.
\label{thm:3D}
\end{theorem}

The main motivation for the result above stems from a recent theorem of Brendle--Broaddus--Putman \cite{BBP20}, which yields the following, see the discussion around Proposition \ref{prop:splitting} below for more details; here, $\Aut(\mathbb F_k)$ stands for the automorphism group of the free group $\mathbb F_k$: 

\begin{proposition} 
Let $O\cong \mathbb \bS^3$, $Y\cong \mathbb S^2 \times \mathbb S^1$, For every $d\ge 2$ and $r\ge 1$, $\mB_{d,r}(O,Y)$ contains $\Aut(\mathbb F_k)$  for all $k$. 
\label{thm:BBP}
\end{proposition}

For arbitrary compact 3-manifolds, mapping class groups are only known to be finitely presented, see Conjecture F and the comments at the end of Section 4 of \cite{HM13}. However, the general form of Brown's criterion \cite{Br87} gives: 


\begin{theorem}  (Section \ref{sec:3dim})
Let $O$ and $Y$ be compact orientable 3-manifolds, such that the boundary of $O$ is a union of spheres and $Y$ is closed and prime. Let $d\ge 2$ and $r\ge 1$. Then  $\mB_{d,r}(O,Y)$ is finitely presented. \label{thm:3Dgen}
\end{theorem}



\subsubsection{Higher-dimensional manifolds} In Section \ref{sec:highdim} we will treat the case where $O\cong \bS^{2n}$ and $Y\cong \mathbb S^{2n}$ or $Y\cong \mathbb S^n\times \mathbb S^n$, with $d\ge 2$ and $r\ge 1$. 
%
%
A celebrated theorem of Sullivan \cite{Sullivan} implies that the mapping class group of every suited submanifold of $\mC_{d,r}(O,Y)$ is of type $F_\infty$; see Proposition \ref{prop:induction} below. Moreover, the complex of pieces has the correct connectivity properties, by a version of a theorem of Galatius--Randal-Williams \cite{GRW18} proved in Appendix~\ref{sec:piececomplex_highdim}. In particular, we will obtain: 

\begin{theorem}  (Section \ref{sec:highdim})
For $n\ge 3$, let $O\cong \bS^{2n}$ and $Y$ diffeomorphic to either $\bS^{2n}$ or $\mathbb S^n\times \mathbb S^n$. For every $d\ge 2$ and $r\ge 1$, $\mB_{d,r}(O,Y)$ is of type $F_\infty$. 
\label{thm:highdim}
\end{theorem}


As hinted above, when $Y \cong \mathbb S^n\times \mathbb S^n$, the group $\mB_{d,r}(O,Y)$ contains infinite families of groups that are, following the terminology of \cite{K-RW}, {\em commensurable (up to finite kernel)} to well-studied arithmetic groups. Here, commensurability up to finite kernel stands for the equivalence relation generated by passing to finite-index subgroups and taking quotients by finite normal subgroups.  

More concretely, write $W_g$ for the connected sum of $g$ copies of $\mathbb S^n\times \mathbb S^n$, and let $G_g$ be the group of automorphisms of the middle homology group $H_g(W_g, \mathbb Z)$ that preserve both the  intersection form and {\em Wall's quadratic form}; see \cite[Section 1.2]{Kr20}. It is known that $G_g$ is (up to finite index) a symplectic or orthogonal group, see \cite[Section 1.2]{Kr20} for details.  In light of the discussion at the end of Section \ref{subsec:highdim}, we have:

\begin{proposition} (Section \ref{sec:3dim})
For $n\ge 8$, let $O\cong \bS^{2n}$ and $Y\cong \mathbb S^n\times \mathbb S^n$. If $d\ge 2$ and $r\ge 1$, then $\mB_{d,r}(O,Y)$ contains a subgroup commensurable (up to finite kernel) to $G_g$ for all $g\ge 1$.
\label{prop:arithmetic}
\end{proposition}

\noindent{\bf A note on the hypotheses.} Before continuing, we explain the restrictions on the manifolds that appear in Theorems \ref{cor:surfaces}, \ref{thm:3D}, \ref{thm:3Dgen} and \ref{thm:highdim}. 

Firstly, as we mentioned above, the manifolds involved must satisfy the {\em cancellation}, {\em inclusion}, and {\em intersection} properties; however, as remarked after Proposition \ref{prop:properties} below, it is not true that all manifolds satisfy the three properties. 

Second, as discussed above, one needs to know the finiteness properties of mapping class groups of suited submanifolds of $\mathcal C_{d,r}(O,Y)$, which is the case for the classes of manifolds appearing in our main theorems. 

In addition, the main technical difficulty in our proofs is the analysis of the connectivity properties of the descending links of vertices in the complex $\mathfrak X_{d,r}(O,Y)$. This is carried out through the use of various intermediate complexes, starting with complexes that appear in the study of {\em homology stability} for precisely the different manifolds $O$ and $Y$ that appear in our applications, and whose connectivity properties are understood,  thanks to the work of Harer \cite{Har85}, Hatcher--Wahl \cite{HW05}, and Galatius--Randal-Williams \cite{GRW18}, respectively.

\subsubsection{A connection with Higman's Embedding Theorem}
By the celebrated Higman Embedding Theorem \cite{Higman}, there exists a finitely presented group that contains all finitely presented groups. The following immediate corollary of Theorems \ref{cor:surfaces},  \ref{thm:3D} and  \ref{thm:highdim} may be regarded as a strong form of Higman Embedding Theorem for the  family of mapping class groups of compact surfaces (resp. automorphism groups of free groups of finite rank, and the groups commensurable (up to finite kernel) to $G_g$ of Proposition \ref{prop:arithmetic})


\begin{corollary}
There exists a group of type $F_\infty$ that contains the mapping class group of every compact orientable surface with non-empty boundary (resp. the automorphism group of every free group of finite rank, or a group commensurable (up to finite kernel) to $G_g$ for all $g\ge 1$).
\end{corollary}

\subsection{Non-positive curvature} A natural question is whether the contractible cube complex $\mathfrak X_{d,r}(O,Y)$  is non-positively curved. In this direction,  we will prove that in all the cases considered in the previous subsection, the link of every vertex of $\mathfrak X$ is flag. Since $\mathfrak X_{d,r}(O,Y)$ contains no infinite-dimensional cubes (see Proposition \ref{prop:ascending}),  we will obtain: 

\begin{theorem}  (Section \ref{subsec:cat})
Suppose $O, Y, d$ and $r$ satisfy the hypotheses of Theorems \ref{cor:surfaces}, \ref{thm:3D}, \ref{thm:3Dgen} or \ref{thm:highdim}. Then the cube complex $\mathfrak{X}_{d,r}(O,Y)$ is complete and ${\rm CAT}(0)$. 
\end{theorem} 

An intriguing research avenue is to explore the geometric structure of $\mathfrak X_{d,r}(O,Y)$ as a non-positively curved space, both as an intrinsically interesting  object, and also with the ultimate aim  
of deducing algebraic properties of $\mB_{d,r}(O,Y)$ from its action on $\mathfrak X_{d,r}(O,Y)$. We propose the following problem in this direction: 

\bigskip

\noindent{\bf Problem.} 
Describe the ${\rm CAT}(0)$ boundary of $\mathfrak X_{d,r}(O,Y)$ in terms of topological data from $\mC_{d,r}(O,Y)$.

\subsection{Homology} As mentioned above, our results are intimately linked to the different {\em homological stability} phenomena, as we now explain. Write $W_g$ for the connected sum of $O$ and $g$ copies of $Y$, and let $W_g^1$ the manifold that results from removing an open ball from $W_g$. In the case when $O$ and $Y$ are surfaces, Harer \cite{Har85} proved that $H_i(\Map(W_g^1),\mathbb Z)$ does not depend on $g$, provided $g$ is {\em big enough} with respect to $i$ (see Section \ref{sec:homology}); for this reason, this homology group is called the $i$-th {\em stable} homology group of $\Map(W_g^1)$. An analogous statement for 3-manifolds was established by Hatcher--Wahl \cite{HW05}, and Galatius--Randal-Williams \cite{GRW18} proved a stability result for diffeomorphism groups of connected sums of $Y\cong \bS^n \times \bS^n$.

In Section \ref{sec:homology} we will adapt an argument of Funar--Kapoudjian \cite[Proposition 3.1]{FK09} to our setting to prove the following:

\begin{theorem} (Section \ref{sec:homology})
Suppose that either $O$ is any compact surface and $Y\cong \bS^1 \times \bS^1$, or
    $O \cong \bS^3$ and $Y\cong \bS^2 \times \bS^1$.
 Then $H_i(\mB_{2,1}(O,Y), \mathbb Z)$ is isomorphic to the $i$-th stable homology group of $\Map(W_g^1)$. 
\label{thm:homology}
\end{theorem}

\begin{remark}
The reason that the result above is stated for $d=2$ and $r=1$ is that, in this case, the corresponding Higman--Thompson group $V_{2,1}$ (which in fact is just the Thompson's group $V)$ is {\em simple} and {\em acyclic} \cite{SW19}, which is not the case for arbitrary $d$ and $r$. On the other hand, for any $r\geq1$, the groups $V_{2,r}$ are pairwise isomorphic, and thus acyclic and simple; therefore, our proof still applies to this case. 
\end{remark}

\vspace{0.5cm}

\noindent{\bf The topological vs the smooth categories.}
We have chosen to work in the differentiable category in order to present our results in a unified manner. However, if $M$ has dimension $\le 3$, then by \cite{FM,Ha78,Ha83}
\[\pi_0(\Homeo^+(M,\partial M)) \cong \pi_0(\Diff^+(M,\partial M)),\]
and thus the results above remain valid in the topological category also. In higher dimensions, however, the two categories are different.

\bigskip

\noindent{\bf Acknowledgements.} The authors are indebted to Oscar Randal-Williams for answering their many questions about mapping class groups of  higher-dimensional manifolds. We are also grateful to Andrew Putman for conversations about mapping class groups of 3-manifolds. 
Finally, we thank Louis Funar, Allen Hatcher, Ian Leary, Chris Leininger, Michah Sageev, Dennis Sullivan, Hongbin Sun, Nathalie Wahl and Matt Zaremsky. 
We are indebted to Chris Leininger for pointing out an error in Proposition 7.3 of a previous version, and to him and Rachel Skipper for suggesting an argument around it; see Proposition \ref{prop:join}. Finally, thanks to the referees for comments and suggestions that served to improve the paper.




\section{Mapping class groups} 
\label{sec:mcgs} 
 Let $M$ be a compact, connected, orientable  manifold of dimension $n\ge 2$, and $\Diff^+(M)$ the group of orientation-preserving self-diffeomorphisms of $M$. Denote by $\Diff^+(M,\partial M)$ the subgroup consisting of those diffeomorphisms that restrict to the identity on every connected component of $\partial M$. The {\em mapping class group} of $M$ is \[\Map(M):= \pi_0(\Diff^+(M,\partial M)).\] 
 We remark that every element of $\Map(M)$ has a representative that restricts to the identity on a collar neighborhood of every boundary component of $M$. 
  In what follows, we will be interested in manifolds that have some boundary components diffeomorphic to
 $\bS^{n-1}$, and  will need a version of the mapping class group that allows for these components to be permuted. To this end, for every sphere $S \subset \partial M$ we fix, once and for all, a diffeomorphism $h_S: S\to \mathbb{S}^{n-1}$, called a {\em parametrization} of $S$. We say that $f\in \Diff^+(M)$ {\em respects the boundary parametrization} if $h_{f(S)} \circ f_{|S} \equiv h_S$ for all $S$.

\begin{definition}\label{defn-sp-mcg}[Sphere-permuting mapping class group] Let $M$ as above. 
The {\em sphere-permuting} mapping class group $\Map_o(M)$ is the group of isotopy classes of orientation-preserving self-diffeomorphisms of $M$ that respect the boundary parametrization, modulo isotopies that fix $\partial M$ pointwise. 
\label{def:spherepermuting}
\end{definition}

Observe that $\Map_o(M)$ acts transitively on the set of boundary spheres of $M$. To see this, first note that this is obvious if $M$ is a ball; for arbitrary $M$, first choose a sphere in $M$ that cuts off a ball containing all of the boundary spheres of $M$, and apply the previous case. 

The following immediate observation will be important in what follows: 

\begin{lemma}
Suppose that $M$ has  finitely many spherical boundary components. Then $\Map(M)$ is a normal subgroup of finite index in $\Map_o(M)$
\end{lemma}

\subsection{Finiteness properties of mapping class groups} 
\label{subsec:finiteness}
 Recall that a group $G$ is of type $F_n$ if it admits a $K(G,1)$ with finite $n$-skeleton.  We say that $G$ has type $F_\infty$ if it has type $F_n$ for all $n$. We will  use of the following result; for a proof, see  e.g. \cite[Section 7]{Ge08}:

\begin{proposition} Let $G, K$ and $Q$ be groups. 
\begin{enumerate} 
    \item Let $1\to K \to G \to Q \to 1$ be a short exact sequence of groups. If $K$ and $Q$ (resp. $K$ and $G$) are of type $F_n$, so is $G$ (resp. $Q$). 

    \item If $K\le G$ has finite-index, then $G$ is of type $F_n$ if and only if $K$ is. 
\end{enumerate}
\label{prop:extensionF}
\end{proposition}

We now examine the finiteness properties of mapping class groups of specific families of manifolds.

\subsubsection{Surfaces} A classical fact about mapping class groups of compact surfaces is that they are of type $F_\infty$. Indeed, work of Harvey \cite{Harv} implies this for closed surfaces;  when the boundary is non-empty, one may combine the {\em capping homomorphism} \cite[Section 4.2.5]{FM} with part (i) of Proposition \ref{prop:extensionF} to deduce the following well-known fact:

\begin{lemma} \label{lem-fin-mcg-surf}
Let $S$ be a compact connected orientable surface. Then $\Map(S)$ (and thus $\Map_o(S)$ as well) is of type $F_\infty$. 
\end{lemma}

\subsubsection{3-manifolds} Finiteness properties of mapping class groups of arbitrary 3-manifolds are more mysterious. In this direction, the following  general statement appears in Hong--McCullough \cite{HM13}: 
\begin{theorem}[\cite{HM13}]
Let $M$ be a compact connected orientable 3-manifold whose boundary is a union of spheres.  Then $\Map(M)$ is finitely presented. 
\label{thm:3Dfinpres}
\end{theorem}

For certain families of 3-manifolds, one can drastically improve the result above. Indeed, the case of primary interest to us occurs for the connected sum $W_g$ of a finite number $g\ge 2$ of copies of $\mathbb S^2 \times \mathbb S^1$.  Since $\pi_1(W_g)$ is isomorphic to the free group $\F_g$, there is an obvious homomorphism \[\Map(W_g)\to\Out(\F_g),\] where $\Out(\F_g)$ denotes the outer automorphism group of the free group. By a theorem of Laudenbach \cite{Lau73}, this homomorphism is surjective and its kernel is a finite abelian group generated by twists along embedded spheres in $W_g$. Now, $\Out(\F_g)$ is of type $F_\infty$ by work of Culler--Vogtmann \cite{CV05}, and hence Proposition \ref{prop:extensionF} gives: 

\begin{theorem}\label{thm-fin-mcg-hbdy}
Let $W_g$ be the connected sum of $g\ge 2$ copies of $\bS^1 \times \bS^2$. Then $\Map(W_g)$ is of type $F_\infty$.
\end{theorem}

Consider now the manifold $W_g^b$ that results from $W_g$ by removing $b\ge 1$ open balls with pairwise disjoint closures. There is a short exact sequence 
\[1\to  K \to \Map(W^b_g) \to \Map(W_g^{b-1}) \to 1,\] where $K$ is isomorphic to either $\mathbb F_g$, if $b=1$; or to $\mathbb F_g\times \mathbb Z_2$ if $b\ge 2$, see \cite[Theorem 4.1]{Lan21}. In any case, again in light of Proposition \ref{prop:extensionF}, one has:

\begin{corollary} \label{cor-fin-mcg-hbdy}
For all $g\ge 1$ and $b\ge 0$, the group $\Map(W_g^b)$ is of type $F_\infty$. 
\end{corollary}

We end this subsection with a related discussion that  sheds light on the relation between  asymptotic mapping class groups of 3-manifolds and automorphisms of free groups (cf. Theorem \ref{thm:BBP}). Let $\Map(W_g,*)$ be the mapping class group of $W_g$ with one marked point, so that there is a homomorphism 
 \begin{equation}
    \Map(W_g, *) \to \Aut(\mathbb F_g)
    \label{eq:aut}
\end{equation} 
 which, again by Laudenbach's result \cite{Lau73}, is surjective and has finite kernel. A recent theorem of Brendle--Broaddus--Putman \cite{BBP20} asserts that the  homomorphism \ref{eq:aut} splits. Now, $\Map(W_g,*)$ and $\Map(W_g^1)$ are isomorphic (see, for instance,  \cite[Lemma 2.1]{Lan21}).  In particular, given $W_g^b$ with $g\ge 2$ and $b\ge 1$, by choosing a separating sphere that bounds a copy of $W_g^1$  we obtain: 

\begin{proposition}[\cite{BBP20}]
For $g \ge 2$ and $b \ge 1$,  $\Aut(\mathbb F_{g})< \Map(W_g^b)$. 
\label{prop:splitting}
\end{proposition}

\subsubsection{Higher dimensions}
\label{subsec:highdim} 
Mapping class groups of manifolds of dimension four and higher can be rather wild: for instance, $\Map(\bS^1 \times \bS^3)$  is not finitely generated \cite{BG}, and the same holds for the mapping class group of the $n$-dimensional torus, for $n\ge 6$ (see e.g. \cite[Theorem 2.5]{HS}). However, Sullivan proved the following striking result \cite{Sullivan}: 

\begin{theorem}[\cite{Sullivan}]
Let $M$ be a closed orientable simply-connected manifold of dimension at least five. Then $\Map(M)$ is commensurable (up to finite kernel) with an arithmetic group. 
\label{thm:sullivan}
\end{theorem}

As arithmetic groups are of type $F_\infty$, the combination of Proposition \ref{prop:extensionF} and Theorem \ref{thm:sullivan} yields: 

\begin{corollary}
If $M$ is a closed orientable simply-connected smooth manifold of dimension at least five, then $\Map(M)$ is of type $F_\infty$. 
\end{corollary}


As in the lower-dimensional cases, we will need the following well-known analog of the previous corollary for manifolds with spherical boundary:

\begin{proposition}\label{prop:fin-mcg-high}
Let $M$ be a compact orientable simply-connected   manifold of dimension at least five, and whose boundary is a union of spheres.  Then $\Map(M)$ is of type $F_\infty$. 
\label{prop:induction}
\end{proposition}

\begin{proof}
We prove it by induction on $k=\pi_0(\partial M)$. The case $k=0$ is given by Sullivan's Theorem \ref{thm:sullivan}, so suppose $k\ge 1$. Denote by $M'$ the manifold that results by capping a fixed sphere in $\partial M$ with a disk $B$.   We have $\Diff^+(M, \partial M) = \Diff^+(M', \partial M' \sqcup B)$, where the latter group stands for the group of orientation-preserving diffeomorphisms of $M'$ that restrict to the identity on $\partial M'$ and on $B$. There is a fiber bundle (see e.g.  \cite[Proof of Lemma 2.4]{T}): 
\begin{equation}
\Diff^+(M', \partial M' \sqcup \mathbb{B}^n) \to \Diff^+(M', \partial M' ) \to \Emb^+(\mathbb B^n, M'),
\label{eq1}
\end{equation}
where $\Emb^+(\mathbb{B}^n, M')$ is the space of orientation-preserving smooth embeddings of $\mathbb B^n$ into $M'$. The long exact sequence of homotopy groups associated to the bundle contains the following subsequence: 
\begin{equation}
\pi_1(\Emb^+(\mathbb{B}^n, M')) \to \Map(M) \to \Map(M') \to \pi_0(\Emb^+(\mathbb{B}^n, M'))
\label{eq2}
\end{equation}

Fix an orientation on $M'$ and consider the oriented frame bundle $\Fr^+(M')$, that is the bundle of positively oriented frames on $M'$. One has (see e.g. \cite[Proof of Lemma 2.4]{T})) that $\Emb^+(\mathbb{B}^n, M')$ is homotopy equivalent to $\Fr^+(M')$.
In particular, equation \eqref{eq2} reads:
\begin{equation}
\pi_1(\Fr^+(M')) \to \Map(M) \to \Map(M') \to \pi_0(\Fr^+(M'))
\label{eq2prime}
\end{equation}
Consider the obvious fibration: 
$\GL^+(n,\R) \to \Fr^+(M') \to M'$,
which yields: 
\[
\pi_1(\GL^+(n,\R)) \to \pi_1(\Fr^+(M')) \to \pi_1(M') \to \pi_0(\GL^+(n,\R)).
\]
Since $M$ is simply-connected, so is $M'$, and thus either $\Map(M') \cong \Map(M)$ or
 \eqref{eq2prime} reads: 
\begin{equation}
\mathbb Z_2 \to \Map(M) \to \Map(M') \to 1,
\label{eq2plus}
\end{equation}
in which case the result follows from Proposition \ref{prop:extensionF}. 
\end{proof} 

In what follows, we will be mainly interested in the
special case of manifolds connected sums of $\bS^n \times \bS^n$ for $n\ge 3$. More concretely, let $W_g$ denote the connected sum of $g$ copies of $\bS^n \times \bS^n$, and $W_g^b$ the manifold that results by removing $b$ open balls from $W_g^b$. The mapping class group $\Map(W_g^b)$ has a very explicit description that stems from the work of Kreck \cite{Kr79}; see Appendix \ref{sec:intersection}. In particular, the fact that $\Map(W_g^b)$ is of type $F_\infty$ follows also from Kreck's calculation \cite{Kr79}.

As mentioned in the introduction, one of the main reasons we are interested in mapping class groups of these manifolds is that they are commensurable (up to finite kernel) to well-studied arithmetic groups. Indeed, write $G_g$ for the subgroup of the automorphism group of the homology group $H_n(W_g, \mathbb Z)$ whose elements preserve the intersection form and {\em Wall's quadratic form},  see \cite[Section 1.2]{Kr20} for a precise definition.  It is known that $G_g$ is (a finite index subgroup of) a symplectic or orthogonal group, see \cite[Section 1.2]{Kr20} for details. Combining Kreck's calculation \cite{Kr79} (see also Appendix \ref{sec:intersection}),   and a recent result of Krannich \cite{Kr20}  (see Theorem A and Table 1 therein), we obtain: 

\begin{theorem}
For $n\ge 8$, $\Map(W_g)$ contains a subgroup commensurable (up to finite kernel) to $G_g$.
\end{theorem}

Moreover, a theorem of Kreck (see \cite[Lemma 1.1]{Kr20} for a proof) asserts that $\Map(W_g^1) \cong \Map(W_g)$. Thus, given $W_g^b$ with $b\ge 1$, by choosing a separating sphere that cuts off all of the boundary components of $W_g^b$, we obtain: 

\begin{corollary}
For $n\ge 8$ and $b\ge 1$, $\Map(W_g^b)$ contains a subgroup commensurable (up to finite kernel) to $G_g$.
\end{corollary}


\section{Cantor manifolds and asymptotic mapping class groups}\label{prelim}
In this section we detail the construction of Cantor manifolds and their asymptotic mapping class groups. It will be useful for the reader to keep Figure \ref{fig:constr_Cantor_mfd} in mind. Let $Y$ be  a closed, connected, orientable manifold of dimension $n\ge 2$. For $d\ge 2$, denote by $Y^d$ the manifold that results by removing $d+1$ disjoint open $n$-balls from $Y$.  We enumerate the boundary spheres of $Y^d$ as $S_0, S_1,\cdots, S_d$,  and fix a diffeomorphism  $\mu_i:\bS^{n-1}\to S_{i}$. We  refer to $S_0$ as the {\em top} boundary component of $Y^d$. In fact, we regard $Y^d$ as a cobordism with $\partial_0 Y^d =S_0$ at the top and $\partial_1 Y^d =  \{S_1,\cdots, S_d\}$ at the bottom; in particular, we assume that $\mu_0$ is orientation preserving and $\mu_i$ is orientation reversing for $i\geq 1$.

 Now let $O$ be a compact, connected, orientable manifold of the same dimension $n$, possibly with non-empty boundary, although this time not necessarily a collection of spheres. For a fixed $r\ge 1$, define a sequence $\{O_k\}_{k\ge 1}$ of compact, connected, orientable manifolds  as follows:
\begin{enumerate}
    
    \item $O_1$ is the manifold that results from $O$ by deleting the interior of $r$ copies of $\mathbb B^n$, denoted (and ordered as) $ B_1,\cdots,B_r$.  We fix orientation preserving 
    diffeomorphisms $\nu_i:\bS^{n-1}\to \partial {B}_i\subseteq O_1$.
    
    \item $O_2$ is the result of gluing $r$ copies of $Y^d$ to $O_1$, by identifying $\partial B_i$ with the top boundary sphere of the relevant copy of $Y^d$. Here we use the maps $\mu_0 \circ \nu_i^{-1}$ for the gluing.
    
    \item For each $k\ge 3$, the manifold $O_k$ is obtained from $O_{k-1}$ by gluing a copy of $Y^d$, along its top boundary component, to each boundary sphere in $O_k \setminus O_{k-1}$. Here we use the maps $\mu_0 \circ \mu_i^{-1}$ for the gluing.
\end{enumerate}
The spheres created in each step  above are called {\em suited spheres}; observe that the set of suited spheres of $O_k$ has an induced total order from that of $Y^d$ and that of the $B_i$.  Finally, we call the boundary components of $O_k$ that come from $O$ the {\em primary} boundary components of $O_k$.

\begin{definition}[Cantor manifold]
With respect to the above terminology, the  {\em Cantor manifold}  $\mC_{d,r}(O,Y)$ is the union of the manifolds $O_k$.  Each of the connected components of $O_k \setminus O_{k-1}$ (for $k\ge 2$) is called a {\em piece}.
\label{def:Cantormfd}
\end{definition}

By construction, each piece is diffeomorphic to $Y^d$, and each of its $d+1$ boundary components is a  suited sphere.
We remark that the name {\em Cantor manifold} is justified since the {\em space of ends} of $\mC_{d,r}(O,Y)$ is homeomorphic to a Cantor set. The following notion will be key in our definition of asymptotic mapping class groups: 

\begin{definition}[Suited submanifold]
Let $\mC_{d,r}(O,Y)$ be a Cantor manifold. A connected submanifold  of $\mC_{d,r}(O,Y)$ is {\em suited} if it is the union of $O_1$ and finitely many  pieces; in particular, suited submanifolds are  compact. 
\end{definition}

If $M$ is a suited submanifold of $\mC_{d,r}(O,Y)$, we will refer to the collection of boundary components coming from $O$ as {\em primary} boundary components, and denote the set of primary boundary components by $\partial_p M$.  Observe that every boundary component of $M$ which is not in $\partial_p M$ is a suited sphere; we will write $\partial_s M$ for the set of suited boundary components of $M$. As mentioned above, there is a natural total order on $\partial_s M$.


\subsection{The cancellation, inclusion, and intersection properties}
As mentioned in the introduction, we will restrict our attention to Cantor manifolds whose suited submanifolds have mapping class groups with some particular properties, which will play a central role in the construction of the cube complex in Section \ref{sec:Stein-complex}. The first property is the following: 

\begin{definition}[Inclusion property]\label{def:inclusion property}
We say that $\mC_{d,r}(O,Y)$ has the {\em inclusion property} if the following holds. Let $M$ be a suited submanifold, and $N\subset M$ a connected submanifold that is either suited or the complement of a suited submanifold with at least two boundary spheres. Then the homomorphism $\Map(N) \to \Map(M)$ induced by the inclusion map $N\hookrightarrow M$ is injective. 
\end{definition}


The next property allows us to ``trim off'' the trivial  parts of a mapping class supported on the intersection of two suited submanifolds.

\begin{definition}[Intersection property]\label{def:intersection property}
Let  $\mC_{d,r}(O,Y)$ be a Cantor manifold with the inclusion property. We say that $\mC_{d,r}(O,Y)$ satisfies the {\em intersection property} if, for every suited submanifold $M$, the following holds. 
Let $L_1$ and $L_2$ be disjoint submanifolds, each diffeomorphic to $Y^d$ (not necessarily a piece, see Figure \ref{fig:piece_complex}),  and such that all but one boundary components of $L_i$ are suited spheres of $M$.  Write $M_i$ for the complement of $L_i$ in $M$, and set $N=M_1\cap M_2$. Then $\Map(N) = \Map(M_1) \cap \Map(M_2)$.


\end{definition}

In the above definition, we are making implicit use of the inclusion property in order to view mapping class groups of submanifolds as subgroups of a given mapping class group.

Finally, we introduce the cancellation property alluded to in the introduction. Roughly speaking, it asserts that if we remove, from a suited submanifold, a submanifold diffeomorphic to a piece, then its complement is diffeomorphic to some other suited submanifold of smaller complexity.

\begin{definition}[Cancellation property]
\label{def:cancel_prop}
We say that $\mC_{d,r}(O,Y)$  has the {\em cancellation property} if the following holds. Let $M$ be a suited submanifold with at least one piece, and $S$ a separating sphere that  cuts off a submanifold diffeomorphic to $Y^d$ and with $d$ suited boundary spheres. Then 
the remaining component $N$ is  again diffeomorphic to some suited submanifold of $\mC_{d,r}(O,Y)$, where the diffeomorphism maps suited spheres to suited spheres and the image of $S$ is also suited.
\end{definition}

\medskip

In Appendix \ref{sec:intersection}, we will see that the above properties are satisfied by all manifolds that appear in the concrete applications mentioned in the introduction. 

\begin{proposition} \label{prop:properties}
Let $O$ and $Y$ be manifolds as in Theorems \ref{cor:surfaces}, \ref{thm:3D}, \ref{thm:3Dgen} or \ref{thm:highdim}. Then $\mC_{d,r}(O,Y)$ has the inclusion, intersection, and cancellation properties.  
\end{proposition}


Before we continue, it should be stressed that the cancellation property is not satisfied in general. For example $(\bS^2\times \bS^2)\#\overline{\C P^2}$ and $\C P^2\#\overline{\C P^2}\#\overline{\C P^2}$ are diffeomorphic but $\bS^2\times \bS^2$ and $\C P^2\#\overline{\C P^2}$ are not even homotopy equivalent. See for example \cite{BCF+21} for more information on this topic.

\subsection{Rigid structures on Cantor manifolds} We now abstract out a notion due to Funar--Kapoudjian \cite{FK04,FK09, AF17} in the case of surfaces, defined therein in terms of arcs and curves cutting the surface into finite-sided polygons. 
Before introducing it, observe that, for each piece $Z$ of a the Cantor manifold $\mC_{d,r}(O,Y)$, we have a canonical diffeomorphism $\iota_Z:Z \to Y^d$ which may be regarded as  the ``identity" map between $Y^d$ and its copy $Z$.

\begin{definition}[Preferred rigid structure]
The set $\{\iota_Z: Z \text{ is a piece}\}$  is called the  {\em preferred rigid structure} on $\mC_{d,r}(O,Y)$.
\end{definition}

\subsection{Asymptotic mapping class group} We are now in a position to introduce the central object of study of this work:

\begin{definition}[Asymptotically rigid diffeomorphism]
Let $\mC_{d,r}(O,Y)$ be a Cantor manifold equipped with the preferred rigid structure. A diffeomorphism $f:\mC_{d,r}(O,Y) \to \mC_{d,r}(O,Y)$ is {\em asymptotically rigid} if there exists a  suited submanifold $M$ such that

\begin{enumerate}
\item $f(M)$ is also suited;
\item $f$ is {\em rigid away from} $M$: for every piece $Z$ of $\mC_{d,r}(O,Y)\setminus M$, we have that $f(Z)$ is also a piece and   $ \iota_{f(Z)}\circ f_{|Z}\circ \iota_Z^{-1}= \id_{Y^d}$.\end{enumerate}
\label{def:asymphomeo}
\end{definition}

The suited submanifold $M$  is called a {\em support} of  $f$.  Note that the support of an asymptotically rigid diffeomorphism is by no means canonical:  any suited submanifold which contains $M$ is also a support for $f$.

We need one final definition before introducing asymptotic mapping class groups. Namely,  two asymptotically rigid self-diffeomorphisms $f,f'$ of $\mC_{d,r}(O,Y)$ are {\em properly isotopic} if there is a suited submanifold $M$ and an isotopy $g:\mC_{d,r}(O,Y) \times [0,1] \to \mC_{d,r}(O,Y)$ between $f$ and $f'$, such that $g_t(M)=M$ and $g_t$ is rigid away from $M$.

\begin{definition}[Asymptotic mapping class group]\label{def:asym}
The {\em asymptotic mapping class group} $\mB_{d,r}(O,Y)$ is the set  of asymptotically rigid self-diffeomorphisms of $\mC_{d,r}(O,Y)$ up to proper isotopy.
\end{definition}

In order to keep notation as simple as possible, we will often blur the difference between an element of $\mB_{d,r}(O,Y)$ and any of its representatives. Note that, even if supports are not unique, the proper isotopy class of $f$ is determined by the support $M$ and the restriction $f|_M$.

As the reader may suspect, $\mB_{d,r}(O,Y)$ turns into a group with respect to the obvious operation induced by the composition of asymptotically rigid  
self-diffeomorphisms of $\mC_{d,r}(O,Y)$ representing proper isotopy classes.


\begin{lemma}
With respect to the above operation, $\mB_{d,r}(O,Y)$ is a group. 
\end{lemma}
\begin{proof} Clearly, the composition of asymptotically rigid  
self-diffeomorphisms of $\mC_{d,r}(O,Y)$ is again asymptotically rigid. Therefore, it suffices to show that the composition preserves proper isotopy. 

For $i=1,2$, let $F_i:\mC_{d,r}(O,Y) \times [0,1] \to \mC_{d,r}(O,Y)$ be a proper isotopy supported on  $M_i$ (in the sense that it is the identity outside $M_i$) such that $F_i(x, 0)=f_i(x)$ and $F_i(x, 1)=f'_i(x)$ for all $x\in \mC_{d,r}(O,Y)$. Then $F(x,t):=F_2(F_1(x,t), t)$ is a proper isotopy supported on $M_1\cup M_2$ between $f_2\circ f_1$ and $f'_2\circ f'_1$. 
\end{proof}


\subsection{A remark on the definition.}
Our definition of the asymptotic mapping class group a priori differs from the one of Aramayona--Funar--Kapoudjian  \cite{FK04,FK09,AF17}. Indeed, in those papers an asymptotically rigid mapping class is the (unrestricted) isotopy class  of an asymptotically rigid diffeomorphism. However,  both definitions coincide:

\begin{proposition} Let $O$ and $Y$ be compact connected orientable surfaces. If two asymptotically rigid self-diffeomorphisms of $\mC_{d,r}(O,Y)$ are isotopic, then they are also properly isotopic.

\label{prop:ARisotopy}
\end{proposition}

\begin{proof}
Let $f$ and $f'$ be two asymptotically rigid diffeomorphisms which are  isotopic; without loss of generality, we may assume that $f'$ is the identity.
Therefore, there exists some suited subsurface $M\subset \mC_{d,r}(O,Y)$ such that $f(M) = M$ and $f$ is the identity when restricted to the complement of $M$. By the Alexander Method \cite[Section 2.3]{FM}, the restriction of $f$ to $M$ is isotopic, via an isotopy defined on $M$, to the identity map on $M$. We may extend this isotopy to the complement of $M$ by the identity map; this  gives the desired proper isotopy  between $f$ and the identity.
\end{proof}

As an immediate consequence, we get: 

\begin{corollary}
Let $O$ and $Y$ be compact connected orientable surfaces.
Then $\mB_{d,r}(O,Y)$ is a subgroup of  $\Map(\mC_{d,r}(O,Y))$. 
\end{corollary}

\begin{remark}
The analog of Proposition \ref{prop:ARisotopy} is not true in higher dimensions. In dimension higher than $2$, one may use the relation between sphere twists in a manifold with spherical boundary (see Theorem \ref{thm:3d-boundarytwist} and \ref{lemm-prod-dtw-trv} in Appendix \ref{sec:intersection}) in order to write the twist about a suited sphere in terms of twists about suited spheres that lie outside any given compact set. As a consequence, this sphere twist will be isotopic to the identity, but not properly so. 

Along similar lines, recall there are self-diffeomorphisms of the disc $\mathbb D^n$ that are not isotopic to the identity, as long as $n\ge 5$ \cite{Cerf}. Using this, one may find a self-diffeomorphism of $\mC_{d,r}(O,Y)$, supported on a small disc, that is not properly isotopic to the identity \cite[Theorem 2]{Kr79}; however, this diffeomorphism will be isotopic to the identity by ``pushing its support to infinity'' in $\mC_{d,r}(O,Y)$. We are grateful to Oscar Randal-Williams for explaining this to us. 

\end{remark}


\section{The relation with Higman--Thompson groups} \label{sec:rel-Hig-Thom}
In this section, we discuss the relation between asymptotic mapping class groups and the Higman--Thompson groups, establishing  the short exact sequence \eqref{eq:sesB} mentioned in the introduction.


\subsection{Higman--Thompson groups} Higman--Thompson groups were first introduced by Higman \cite{Hi74} as a generalization of Thompson's groups. We now briefly recall the definition of these groups.  

A \emph{finite rooted $d$-ary tree} is a finite tree such that every vertex has degree $d+1$, except the \emph{leaves} which have degree 1, and the \emph{root}, which has degree $d$.  As usual, we depict  trees with the root at the top and the nodes descending from
it, see Figure \ref{fig:reduction_V}. A vertex of the tree along with its $d$  descendants is a \emph{caret}. If the leaves of a caret are also leaves of the tree, the caret is \emph{elementary}. 

A disjoint union of $r$ many $d$-ary trees will be called an $\emph{$(d,r)$-forest}$; when $d$ is clear from the context, we will just refer to it as an $r$-forest.
A \emph{paired $(d,r)$-forest diagram} is a triple $(F_-,\rho,F_+)$
consisting of two $(d,r)$-forests $F_-$ and $F_+$ with the same number $l$ of leaves, and a permutation $\rho \in \Sym(l)$, the symmetric group on $l$ elements.  We label the leaves
of~$F_-$ with $1,\dots,l$ from left to right; the leaves of $F_+$ are labelled according to $\rho$, see  Figure \ref{fig:reduction_V}.

Suppose there is an
elementary caret in~$F_-$ with leaves labeled by $i,\cdots,i+d-1$ from left to right, and an elementary caret in~$F_+$ with  leaves labeled by ~$i, \ldots,i+d-1$ from left to right.  Then we can ``reduce'' the diagram
by removing those carets, renumbering the leaves and replacing
$\rho$ with the permutation~$\rho'\in \Sym(l-d+1)$ that sends the new leaf
of~$F_-$ to the new leaf of~$F_+$, and otherwise behaves like~$\rho$.
The resulting paired forest diagram~$(F'_-,\rho',F'_+)$ is then said to
be obtained by \emph{reducing}~$(F_-,\rho, F_+)$, and it is called a {\em reduction} of it; see Figure~\ref{fig:reduction_V}. A paired forest diagram is
called \emph{reduced} if it does not admit any reductions. 
The reverse
operation to reduction is called \emph{expansion}, so $(F_-,\rho,F_+)$
is an expansion of $(F'_-,\rho',F'_+)$.  

Define an equivalence relation on the set of paired $(d,r)$-forest diagrams by declaring two paired  forest diagrams to be equivalent if one can be obtained from the other by a finite sequence of reductions and expansions.
Thus an equivalence class of paired forest diagrams consists of all diagrams
having a common reduced representative; observe that such reduced representatives
are unique.  We will denote the equivalence class of $(F_-,\rho,F_+)$ by $[F_-,\rho,F_+]$.

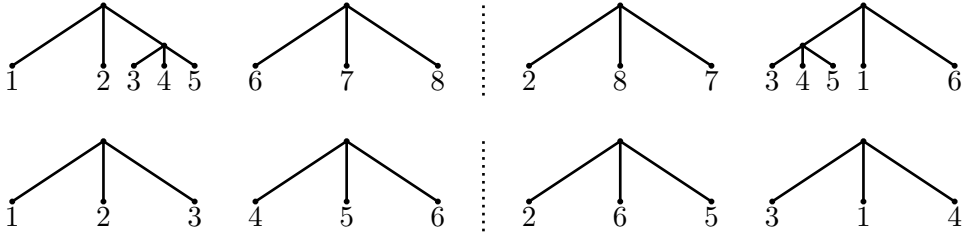
\begin{figure}\label{fig:ele-V}
\centering
\begin{tikzpicture}[line width=1pt, scale=0.4]
\begin{scope}[xshift= -4cm]
  \draw
   (-12,-2) -- (-9,0) -- (-6,-2)
   (-9,0) -- (-9,-2)
   (-7,-1.33) -- (-7,-2)
   (-7,-1.33) -- (-8,-2)
   (-4,-2) -- (-1,0) -- (2,-2)
   (-1,0)  -- (-1,-2);

\draw[dotted] (3.5,0)  -- (3.5,-3);

  \filldraw
  (-12,-2) circle (1.5pt)
  (-9,0) circle (1.5pt)
   (-6,-2) circle (1.5pt)
   (-9,-2) circle (1.5pt)
  (-7,-1.33)  circle (1.5pt)
   (-8,-2) circle (1.5pt)
   (-7,-2) circle (1.5pt)
    (-4,-2) circle (1.5pt)
   (-1,0) circle (1.5pt)
   (-1,-2) circle (1.5pt)
   (2,-2) circle (1.5pt);
  \node at  (-12,-2.5) {$1$};
  \node at (-8,-2.5) {$3$};
  \node at (-7,-2.5) {$4$};
  \node at (-9,-2.5) {$2$};
  \node at (-6,-2.5) {$5$};
  \node at (-4,-2.5) {$6$};
  \node at (-1,-2.5) {$7$};
  \node at (2,-2.5) {$8$};
\end{scope}

  \begin{scope}[ xshift=13cm]
   \draw
  (-12,-2) -- (-9,0) -- (-6,-2)
   (-9,0) -- (-9,-2)
   
  (-3,-2) -- (-3,-1.33) -- (-2,-2)
   
   (-4,-2) -- (-1,0) -- (2,-2)
   (-1,0)  -- (-1,-2);

  \filldraw
  (-12,-2) circle (1.5pt)
  (-9,0) circle (1.5pt)
   (-6,-2) circle (1.5pt)
   (-9,-2) circle (1.5pt)
  (-3,-1.33)  circle (1.5pt)
   (-3,-2) circle (1.5pt)
   (-2,-2) circle (1.5pt)
    (-4,-2) circle (1.5pt)
   (-1,0) circle (1.5pt)
   (2,-2) circle (1.5pt)
   (-1,-2) circle (1.5pt);
  \node at  (-12,-2.5) {$2$};
  \node at (-6,-2.5) {$7$};
  \node at (-2,-2.5) {$5$};
  \node at (-3,-2.5) {$4$};
  \node at (-9,-2.5) {$8$};
  \node at (-4,-2.5) {$3$};
  \node at (2,-2.5) {$6$};
  \node at (-1,-2.5) {$1$};
 \end{scope}

  \begin{scope}[ yshift=-4.5cm,xshift= -4cm]
    \draw
   (-12,-2) -- (-9,0) -- (-6,-2)
   (-9,0) -- (-9,-2)
   (-4,-2) -- (-1,0) -- (2,-2)
   (-1,0)  -- (-1,-2);

\draw[dotted] (3.5,0)  -- (3.5,-3);

  \filldraw
  (-12,-2) circle (1.5pt)
  (-9,0) circle (1.5pt)
   (-6,-2) circle (1.5pt)
   (-9,-2) circle (1.5pt)
    (-4,-2) circle (1.5pt)
   (-1,0) circle (1.5pt)
   (-1,-2) circle (1.5pt)
   (2,-2) circle (1.5pt);
  \node at  (-12,-2.5) {$1$};
  \node at (-9,-2.5) {$2$};
  \node at (-6,-2.5) {$3$};
  \node at (-4,-2.5) {$4$};
  \node at (-1,-2.5) {$5$};
  \node at (2,-2.5) {$6$};
  \end{scope}

  \begin{scope}[xshift=13cm, yshift=-4.5cm]
   \draw
   (-12,-2) -- (-9,0) -- (-6,-2)
   (-9,0) -- (-9,-2)
   
   
   (-4,-2) -- (-1,0) -- (2,-2)
   (-1,0)  -- (-1,-2);

  \filldraw
  (-12,-2) circle (1.5pt)
  (-9,0) circle (1.5pt)
   (-6,-2) circle (1.5pt)
   (-9,-2) circle (1.5pt)
    (-4,-2) circle (1.5pt)
   (-1,0) circle (1.5pt)
   (2,-2) circle (1.5pt)
   (-1,-2) circle (1.5pt);
  \node at  (-12,-2.5) {$2$};
  \node at (-6,-2.5) {$5$};
  \node at (-9,-2.5) {$6$};
  \node at (-4,-2.5) {$3$};
  \node at (2,-2.5) {$4$};
  \node at (-1,-2.5) {$1$};
  \end{scope}
\end{tikzpicture}

\caption{Reduction, of the top paired $(3,2)$-forest diagram to the bottom one.}
\label{fig:reduction_V}
\end{figure}

There is a natural binary operation $\ast$ on the set of equivalence classes
of paired $(d,r)$-forest diagrams.  Indeed, let $(F_-,\rho,F_+)$ and $(E_-,\xi,E_+)$
be reduced paired forest diagrams.  We can find representatives $(F'_-,\rho',F'_+)$ and
$(E'_-,\xi',E'_+)$, of $(F_-,\rho,F_+)$ and $(E_-,\xi,E_+)$, respectively, such that~$F'_+ = E'_-$.  Then we
set \[(F_-,\rho,F_+)\ast (E_-,\xi,E_+):= (F'_-,\rho'\xi',E'_+).\]  One readily checks that this operation is well defined, and gives a group operation on the set of equivalence classes.

\begin{definition}[Higman--Thompson group]
    \label{def:Higman-V-F-T}
   The {\em Higman--Thompson group} $V_{d,r}$ is the group of equivalence classes of paired
   $(d,r)$-forest diagrams, equipped with the multiplication~$\ast$.  
\end{definition}

 We remark  that $V_{2,1}$ is the classical Thompson's group $V$. Observe that $V_{d,r}$ is a naturally a subgroup of the group of homeomorphisms of the Cantor set.

\subsection{Compactly supported mapping class group} 
Let $\mC_{d,r}(O,Y)$ be a Cantor manifold. Every inclusion  $M\subseteq M'$ of suited submanifolds induces a  homomorphism $j_{M,M'} : \Map(M) \to \Map(M')$, and we set:


\begin{definition}
The \emph{compactly supported mapping class group} of $\mC_{d,r}(O,Y)$ is \[\Map_c(\mC_{d,r}(O,Y)):= \varinjlim \Map(M),\] where $M$ ranges over the set of  suited submanifolds of $\mC_{d,r}(O,Y)$. In particular, \[\Map_c(\mC_{d,r}(O,Y)):= \varinjlim \Map(O_k),\] where $O_k$ are the manifolds  in  Definition \ref{def:Cantormfd}.
\label{def:compactsupport} 
\end{definition}

Note that $\Map_c(\mC_{d,r}(O,Y))$ is  a subgroup of $\mB_{d,r}(O,Y)$; we denote the inclusion map by $\pi$. In fact, given any homeomorphism of a suited submanifold $M$, one extends it by the identity outside $M$ to a homeomorphism of $\mC_{d,r}(O,Y)$, so $\Map(M)$  is naturally a subgroup $\Map_c(\mC_{d,r}(O,Y))$ and hence $\varinjlim \Map(M)$ is also a subgroup.  It is immediate that $\Map_c(\mC_{d,r}(O,Y))$ coincides with the subgroup of $\mB_{d,r}(O,Y)$ whose elements have {\em compact support}, i.e. they are the identity outside some compact submanifold of $\mC_{d,r}(O,Y)$.

\subsection{Surjection to the Higman--Thompson groups} 

Recall from  Definition \ref{def:Cantormfd} that  $\mC_{d,r}(O,Y)$ is the union of the compact manifolds $O_k$. Choose an $n$-ball in $O_1$ which contains the $r$ suited boundary spheres of $O_1$, and let $O_0$ be the manifold obtained from $O_1$ by  removing this ball.

Let ${\mathcal{F}}_{d,r}$ be the infinite forest consisting of $r$ copies of an infinite rooted  $d$-ary tree, and $\mathcal{T}_{d,r}$  the infinite rooted tree obtained from ${\mathcal{F}}_{d,r}$ by adding an extra vertex to ${\mathcal{F}}_{d,r}$ and $r$ extra edges that connect this vertex to the root of the trees in ${\mathcal{F}}_{d,r}$.   There is a natural projection
\[q:\mC_{d,r}(O,Y) \to {\mathcal{T}_{d,r}}\] which respects the order on the suited spheres of $\mC_{d,r}(O,Y)$, and such that the pullback of the root is $O_0$ and the pull back of the midpoint of every edge is a suited sphere.
Let $f\in \mB_{d,r}(O,Y) $ and $M$ a support for $f$.
%
%
%
Let $F_-$ (resp. $F_+$) be the smallest subforest of $ \mathcal{F}_{d,r}$ which contains $q(\partial_s M) \cap \mathcal{F}_{d,r}$ (resp. $q(\partial_s f(M)) \cap \mathcal{F}_{d,r}$). Note that $F_-$ and $F_+$ have the same number of leaves, and that $f$ induces a bijection $\rho$ between the two sets of leaves. The triple $(F_-,\rho, F_+)$ defines an element of $V_{d,r}$; we  call this map $q_\ast$ since it is induced from $q$. One readily shows that $q_\ast$ is well defined. Similar to \cite[Proposition 2.4]{FK04}, \cite[Proposition 4.2, 4.6]{AF17} and \cite[Proposition 3.17]{SW21b}, we now have the following:

\begin{proposition}\label{prop-rel-htg}
There is a short exact sequence
$$ 1 \to \Map_c(\mC_{d,r}(O,Y) ) \xrightarrow{\pi} \mB_{d,r}(O,Y)  \xrightarrow{q_\ast} V_{d,r} \to 1.$$
\end{proposition}
\begin{proof}
 First we show $q_\ast$ is surjective. Given any element  $[F_-,\rho,F_+] \in V_{d,r}$, let $T_-$ (resp. $T_+$) be the tree obtained from $F_-$  (resp. $F_+$)  by adding a single root on the top and $r$ edges connecting to each root of the trees in $F_-$ (resp. $F_+$). Furthermore, let  $T'_-$ (resp. $T'_+$) be the result of removing from $T_-$  (resp. $T_+$)  the leaves and each of the open half edges from the leaves. Consider  $M_- =q^{-1}(T'_+)$ and $M_+ =q^{-1}(T'_+)$, which are both suited submanifolds of $\mC_{d,r}(O,Y) $, and a diffeomorphism $f: M_- \to M_+$ whose restriction to each component of $\partial_p M_-$ is the identity, and such that it maps the suited spheres of $M_-$ to the suited spheres of $M_+$ according to the pattern specified by $\rho$. Moreover, $f$ can be chosen so that $(\iota_{f(Z)}\circ f_{|Z}\circ \iota_Z^{-1})|_{S}= \id_{S}$ for each suited sphere $S$  of $M_-$, where $Z$ is the piece right below $S$. From here, we extend $f$ to (abusing notation)  $f\in \mB_{d,r}(O,Y)$, as desired. 

Next, suppose $q_\ast(f)$ is the identity for some  $f\in \mB_{d,r}(O,Y) $. Let $M$ be a support of $f$. We have that $f(M)=M$ and $f$ induces the identity permutation on the set of suited boundary spheres.  Thus $f \in \Map_c(\mC_{d,r}(O,Y))$.

Finally, given $g\in \Map_c(\mC_{d,r}(O,Y))$, it is immediate that $q_\ast \circ \pi(g) = 1$, and thus the result follows.
\end{proof}


\section{The complex}
\label{sec:Stein-complex}

In this section we will construct the cube complex $\mathfrak X_{d,r}(O,Y)$. We remark that our complex is similar in spirit to the one introduced by Genevois--Lonjou--Urech in \cite{GLU20} for establishing the finiteness properties of a related family of  groups, and is inspired by the {\em Stein--Farley complexes} associated to Higman--Thompson groups \cite{Ste92,Far03}.

Throughout this section, $\mC_{d,r}(O,Y)$ will be assumed to be a Cantor manifold satisfying the cancellation, inclusion  and intersection  properties, and endowed with the preferred rigid structure.  The particular manifolds $O$ and $Y$, as well as the integers $d$ and $r$, will play no particular role in the discussion, and so  we will simply write $\mC=\mC_{d,r}(O,Y)$ for the Cantor manifold, and $\mB=\mB_{d,r}(O,Y)$ for the associated asymptotic mapping class group.


\subsection{Stein complex} Consider all ordered pairs $(M, f)$, where $M$ is a suited submanifold of $\mC$ and $f\in\B$. We deem two such pairs $(M_1, f_1)$ and $(M_2, f_2)$ to be equivalent, and write $(M_1, f_1)\sim (M_2, f_2)$, if and only if there are representing diffeomorphisms (abusing notation) $f_1$ and $f_2$ such that $f_2^{-1}\circ f_1$ maps $M_1$ diffeomorphically onto $M_2$ and is rigid away from $M_1$ (cf. Definition \ref{def:asymphomeo}). Useful examples to keep in mind about this equivalence relation are: 

\begin{enumerate}
    \item If $f\in \B$ has support $M$, then $(M, g) \sim (f(M),  f \circ g)$, for all $g\in \B$.
    \item If $f\in \B$ is the identity outside $M$, then $(M,f) \sim (M, {\rm id})$.
\end{enumerate}

We will denote by $[M, f]$ the equivalence class of the pair $(M,f)$ with respect to this relation, and write $\mathcal P$ for the set of all equivalence classes. Observe that  $\B$ acts on $\mathcal P$ by left multiplication, namely $g\cdot [N,f]=[N, g\circ f]$. 

Consider a pair $(M,f)$. Since $M$ is a suited submanifold, it is uniquely expressed as the union of finitely many pieces and $O_1$. We define the \textit{complexity} $h((M,f))$ of $(M,f)$ as the number of pieces of $M$. Note that if $(M_1, f_1)\sim (M_2, f_2)$, then $h((M_1,f_1)) = h((M_2,f_2))$, and thus $h$ descends to a well-defined function (abusing notation) $h: \mathcal P \to \mathbb N$. Hence we set: 

\begin{definition}
The {\em complexity} of $[M,f] \in \mathcal P$ is defined as $h((M,f))$, for some, and hence any, representative of $[M,f]$.
\label{def:comp}
\end{definition}

We introduce a relation $\le$ on the elements of $\mathcal P$ by deeming  $x_1\leq x_2$ if and only if $x_1=[M_1, f]$ and $x_2=[M_2, f]$ for some suited submanifold $M_1\subseteq M_2$; if the inclusion is proper we will write $x_1 < x_2$. We are going to prove that the relation $\le$ turns $\mathcal P$ into a directed poset. Before doing so, we need the following immediate observation:

\begin{lemma}\label{lem:claim}  Suppose that $(M_1, f_1)\sim (M_2, f_2)$.
\begin{itemize}
\item[(1)] If $N_1$ is a union of pieces such that $M_1\cup N_1$ is  connected, then there exists a union of pieces $N_2$ such that $[M_1\cup N_1, f_1]=[M_2 \cup N_2, f_2]$. 

\item[(2)] If $[M_1\cup N_1, f_1]=[M_2 \cup N_2, f_2]$ and $N_1$ is a disjoint union of pieces, then so is $N_2$.
\end{itemize}
\end{lemma}
\begin{proof} The first part of the claim follows by taking $N_2= (f_2^{-1}\circ f_1)(N_1)$. For the second part, setting $N'_2 = (f_2^{-1}\circ f_1)(N_1)$, we have $[M_1\cup N_1, f_1]=[M_2 \cup N'_2, f_2]$. Then $M_2 \cup N'_2$ is isotopic and hence equal to $M_2 \cup N_2$ since they are both suited. It follows that $N_2=N'_2$.
\end{proof}

With the above result in hand, we prove: 

\begin{lemma}\label{lem:poset} $(\mathcal P, \leq)$ is a directed poset. Moreover,  if $x< y\in \mathcal P$, then $h(x)< h(y)$.
\end{lemma}
\begin{proof} First, we show that $\mathcal P $ is a poset. To see transitivity, let $x\leq y\leq z$. Then $x=[N, f]$, $y=[N\cup M, f]= [N', f']$, and $z=[N'\cup M', f']$, for some unions of pieces $M$ and $M'$. By the first part of Lemma \ref{lem:claim}, there is a union of pieces $M''$ such that $[N\cup M\cup M'', f]=z$ and hence $x\leq z$.  If $z=x$, then $[N\cup M\cup M'', f]=[N, f]$. This implies that $N\cup M\cup M''$ and $N$ are equal, since they are both suited. Thus $x=y$, and  $\mathcal P$ is indeed a poset. 

Now we show that $\mathcal P$ is directed. Consider two elements $x_1=[N_1, f_1]$ and $x_2=[N_2, f_2]$. By definition, $f_2^{-1}\circ f_1$ is asymptotically rigid, so we may choose a support $M_1$ that contains $N_1\cup (f_1^{-1}\circ f_2)(N_2)$. 
Then $M_2=(f_2^{-1}\circ f_1)(M_1)$ is a support of $f_1^{-1}\circ f_2$ containing $N_2$. By construction,  $[M_1, f_1]=[M_2, f_2]$ and $N_i\subseteq M_i$ for $i=1,2$. Thus it follows that $x_1, x_2\leq [M_1, f_1]$.

Finally, the last claim is an immediate consequence of the definition of the complexity function $h$. 
\end{proof}

Consider now the geometric realization $|\mathcal P|$, which is the simplicial complex with a $k$-simplex for every chain $x_0<\dots < x_k$ in $\mathcal P$, and where $x_k$ is called the {\it top} vertex and $x_0$ the {\it bottom} vertex of the simplex.  

\begin{corollary}
$|\mathcal P|$ is contractible. 
\label{cor:contractible}
\end{corollary}

\begin{proof} By Lemma \ref{lem:poset}, $\mathcal P$ is directed. It is  well known that the geometric realization of any directed set is contractible; for a proof, see \cite[Proposition 9.3.14]{Ge08}.
\end{proof}

We are now going to introduce a finer relation $\preceq$ on $\mathcal P$. Given two vertices $x_1, x_2 \in \mathcal P$, we will say that $x_1\preceq x_2$ if $x_1=[N, f]$, $x_2=[N\cup M, f]$ and $M$ is a disjoint union of pieces. If  $x_1\preceq x_2$ and $x_1\ne x_2$, we will write   $x_1\prec  x_2$. Note that, contrary to $\leq$, the finer relation $\preceq$ is not a partial ordering on $\mathcal P$ as it is not transitive. However, as a consequence of Lemma \ref{lem:claim}, it is true that  if $x_1\preceq x_3$ and $x_1\leq x_2\leq x_3$, then $x_1\preceq x_2\preceq x_3$.

We are now going to restrict our attention to a particular subfamily of simplices. To this end, we say that a simplex $x_0<\dots < x_k$   is {\it elementary } if $x_0  \preceq x_k$; see Figure \ref{fig:el_vs_non-el}. We have: 

\begin{figure}
\includegraphics[scale=0.85]{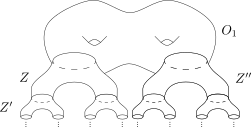}
\caption{The difference between elementary and non-elementary simplices in the poset for $C_{2,2}(O,Y)$, where $O$ is a closed surface of genus 2 and $Y$ is a sphere.  While $O_1\cup Z$ and $O_1\cup Z\cup Z''$ define an elementary $1$-simplex, the simplex defined by $O_1\cup Z$ and $O_1\cup Z\cup Z'$ is non-elementary. }
\label{fig:el_vs_non-el}
\end{figure}

\begin{definition}[Stein complex]
The {\em Stein complex} $\mathfrak X$ is the full subcomplex  of $|\mathcal P|$ consisting only of elementary simplices. 
\end{definition}

 Note that the action of $\B$ on $\mathcal P$ preserves the relation $\preceq$, and thus 
restricts to a simplicial action of $\B$  on $\mathfrak X$. 
Our next aim is to show that $\mathfrak X$ is also contractible.  For $x\leq y$, define the closed interval $[x,y]=\{z \;|\; x\leq z\leq y\}$. Similarly, define  $(x,y)$, $(x,y]$ and $[x,y)$. A similar reasoning to that of \cite[Section 4]{Bro92} yields:

\begin{lemma}\label{contr_int} For $x<y$ with $x\nprec y$, the geometric realization of the interval $|(x,y)|$ is contractible.
\end{lemma}
\begin{proof} By Lemma \ref{lem:claim}, for any  $[N,f]\leq v\leq [N\cup M, f]\in \mathcal P$, there exists a union of pieces $M'\subseteq M$ such that $v=[N\cup M', f]$. Thus, for any $z\in (x,y]$ there exists a largest element $z_0 \in (x,z]$ such that $x\prec z_0$. So, $z_0\in (x,y)$. Also, $z_0\leq y$ for any $z\in (x,y)$. The inequalities $z\geq z_0\leq y_0$ then
imply that $|(x,y)|$ is contractible using the ``conical'' contraction in  \cite[Section 1.5]{Qui78}.
\end{proof}

\begin{proposition}\label{prop:contr} The complex $\mathfrak X$ is contractible.
\end{proposition}

\begin{proof} First, $|\mathcal P|$ is contractible by Corollary \ref{cor:contractible}. It suffices to show that $|\mathcal P|$ can be built inductively from $\mathfrak X$ without changing the homotopy type at each stage. 
To this end, given a closed interval $[x, y]$ in $|\mathcal P|$, define \[r([x,y]):=h(y)-h(x),\] where recall that $h$ denotes the complexity of a vertex.  We attach the contractible subcomplexes $|[x,y]|$ for $x\npreceq y$ to $\mathfrak X$ in increasing order of $r$-value. The subcomplex $|[x,y]|$ is attached along $|(x,y]|\cup|[x,y)|$, which is the suspension of $|(x,y)|$ and hence is contractible by Lemma \ref{contr_int}. This shows that attaching $|[x,y]|$ to $\mathfrak X$ does not change the homotopy type of $\mathfrak X$. Since the result of gluing all such intervals is $|\mathcal P|$, the claim follows. 
\end{proof}

\subsection{Stein--Farley cube complex} We now explain how to obtain a cube complex from $\mathfrak X$. Recall that a {\em Boolean lattice} (also called a {\em Boolean algebra}) is a distributive lattice in which every element has a complement. We first have: 

\begin{lemma}\label{lem:unique} If $x\preceq y$, then $[x,y]$ is a finite Boolean lattice. 
\end{lemma}
\begin{proof}  Let $x=[N, f]$ and $y=[N\cup M, f]$ where $M$ is a disjoint union of pieces, and  $z \in [x,y]$. By Lemma \ref{lem:claim}, there exists a union of pieces $M'$, so that $z=[N\cup M', f]$. Again, by Lemma \ref{lem:claim}, $M'$ is a submanifold of $M$. It follows that $[x,y]$  is a Boolean lattice on the set of pieces of $M$. 
\end{proof}

As a consequence, if $x\preceq y$ then the simplices in the geometric realization of $[x,y]$ piece together into a cube of dimension $r([x,y])$. Before doing so, we need the following technical result, which rests upon the fact that the Cantor manifold has the intersection property:

\begin{lemma}\label{lem:inter}
Let $f_1,f_2$ be asymptotically rigid diffeomorphisms which represent the same element of $\mathcal B$. Let $M_i$ be a support for $f_i$ with $i=1,2$, and write $N_i=f_i(M_i)$. Then: 
\begin{enumerate}
    \item For $i=1,2$, one has $f_i(M_1\cup M_2)= N_1 \cup N_2$.
    \item There is an asymptotically rigid diffeomorphism $g$ with support $M_1\cap M_2$, representing the same element of $\mB$ as $f_i$. 
\end{enumerate}
\label{lem:unioninter}
\end{lemma}

\begin{proof} The manifolds $f_1(M_1\cup M_2)$ and $f_2(M_1\cup M_2)$ are suited and isotopic and therefore equal. In particular, $N_1\cup N_2\subseteq f_i(M_1\cup M_2)$, $i=1,2$. Similarly, $M_1\cup M_2 \subseteq f^{-1}_i(N_1\cup N_2)$, $i=1,2$. Thus,  $f_1(M_1\cup M_2)=f_2(M_1\cup M_2)=N_1\cup N_2$. 

To establish the second part, we proceed as follows. First, note that $M_1\cap M_2$ and $N_1\cap N_2$ are suited submanifolds that have the same number of pieces, since $M_1\cup M_2$ and $N_1\cup N_2$ have the same number of pieces, and $M_i$ and $N_i$, for $i=1,2$, also have the same number of pieces. Hence, there is a diffeomorphism $f$ mapping $N_1\cap N_2$ to $M_1\cap M_2$ which is rigid away from $N_1\cap N_2$. By applying a self-diffeomorphism of $M_1\cap M_2$ that permutes the boundary spheres, we can further assume that $f$ maps $N_i$ to $M_i$, $i=1,2$.  Then $f\circ f_i$ has support $M_i$ and is the identity outside $M_1\cup M_2$. By making iterated use of  the intersection property, we get that there exists an asymptotically rigid diffeomorphism $g$ supported on $M_1 \cap M_2$ such that $g$ is properly isotopic to $f\circ f_i$. Then $f^{-1}\circ g$ is properly isotopic to $f_i$ and is supported on $M_1 \cap M_2$, as claimed.
\end{proof}

We now prove the following lemma, which will be key in order to endow $\mathfrak X$ with the structure of a cube complex. 

\begin{lemma}\label{lem:bounds} Suppose $x\preceq y$ and $z\preceq w$. Denote $S=[x,y]\cap [z,w]$. For any $p,q\in S$, there are $s, t\in S$ such that $s\leq p,q\leq t$.
\end{lemma}
\begin{proof} Since the conclusion is trivially satisfied if $p=q$, we assume $p\ne q$. Let $x=[N_1, f_1]$, $y=[N_1\cup M_1, f_1]$ and $z=[N_2, f_2]$, $w=[N_2\cup M_2, f_2]$. There are unions of disjoint pieces $M_i', M_i''\subseteq M_i$  for $i=1,2$, such that $$p=[N_1\cup M_1', f_1]=[N_2\cup M_2', f_2] \;\mbox{ and }\; q=[N_1\cup M_1'', f_1]=[N_2\cup M_2'', f_2].$$ 
We then get diffeomorphisms 
$$f_2^{-1}\circ f_1: N_1\cup M_1'\to  N_2\cup M_2'$$ such that $f_2^{-1}\circ f_1$ is rigid away from $N_1\cup M_1'$ and 
$$g_2^{-1}\circ g_1: N_1\cup M_1''\to  N_2\cup M_2''$$
such that $g_2^{-1}\circ g_1$ is rigid away from $N_1\cup M_1''$  where $[f_i]=[g_i]\in \B$, $i=1,2$.
By Lemma \ref{lem:unioninter}, there exists an asymptotically rigid diffeomorphism $g$, with $[g]=[f_2^{-1}\circ f_1]=[g_2^{-1}\circ g_1]\in \B$, such that
$$g(N_1\cup M_1'\cup M_1'')=N_2\cup M_2'\cup M_2'',$$
$$g(N_1\cup (M_1'\cap M_1''))= N_2\cup (M_2'\cap M_2''),$$
and $g$ is rigid away from $N_1\cup (M_1'\cap M_1'')$.
By setting $s=[N_1\cup (M_1'\cap M_1''), f_1]$ and $t=[N_1\cup M_1'\cup M_1'', f_1]$, it follows that $s\leq p,q\leq t$. It remains to prove that $s,t\in S$. Clearly, $s,t\in [x,y]$. Also, by the above, 
$$[N_1\cup (M_1'\cap M_1''), f_1]=[N_2\cup (M_2'\cap M_2''), f_2],$$
$$[N_1\cup M_1'\cup M_1'', f_1]=[N_2\cup M_2'\cup M_2'', f_2],$$
which shows that $s,t\in [z,w]$.
\end{proof}

Finally, we show: 

\begin{proposition}\label{prop:sf} $\mathfrak X$ has the structure of a cube complex with each cube defined by an interval $[x,y]$, where $x\preceq y$.
\end{proposition} 

\begin{proof}
It suffices to show that the intersection of two cubes $[x,y]$ and $[z,w]$ is again a cube which is a face of both. To see this, denote $S=[x,y]\cap [z,w]$, and assume $S\ne \emptyset$. By Lemma \ref{lem:bounds}, there are  $s, t \in S$ such that $S\subseteq [s,t]$. Since  $[s,t]$ is a subinterval of both $[x,y]$ and $[z,w]$, it follows that $S=[s,t]$. As $s\preceq t$, $[s,t]$ is a cube and is a face of both $[x,y]$ and $[z,w]$.
\end{proof}

\begin{definition}[Stein--Farley complex]
We will refer to the complex $\mathfrak X$ equipped with the above cubical structure as the {\it Stein--Farley cube complex} associated to the Cantor manifold $\mC$.
\end{definition}
Note that the action of $\B$ on $\mathfrak X$ respects the cubical structure.  

\section{Discrete Morse Theory on the Stein-Farley cube complex}
\label{sec:DMT}
In order to deduce the desired finiteness properties for asymptotic mapping class groups, we will apply the following classical criterion of Brown \cite{Br87}: 

\begin{theorem}[Brown's criterion]\label{brown} Let $G$ be a group and $\mathcal K$ be a contractible $G$-CW-complex such that the stabilizer of every cell is of type $F_{\infty}$. Let $\{\mathcal K_n\}_{n\geq 1}$ be a filtration of $\mathcal K$, and such that each $\mathcal K_n$ is $G$-invariant and $\mathcal K_n/G$ is compact. Suppose the connectivity of the pair $(\mathcal K_{n+1}, \mathcal{K}_n)$ tends to $\infty$ as $n$ tends to $\infty$. Then $G$ is of type $F_{\infty}$.
\end{theorem}

In the theorem above, recall that a pair of spaces $(L,K)$ with $K\subseteq L$ is $k$-connected if the inclusion map $K \hookrightarrow L$ induces an isomorphism in $\pi_j$ for $j< k$ and an epimorphism in $\pi_k$.

Our aim is to prove that the action of $\mB$ on the contractible cube complex $\mathfrak X$ satisfies the hypotheses of Brown's criterion above. 
 First, note that the complexity function of Definition \ref{def:comp} yields a function $h:\mathfrak X^{(0)}\to \mathbb N$, which we may extend affinely on each cube to obtain a {\em height function}
$$h: \mathfrak X\to \R,$$ which is $\B$-equivariant  with respect to the trivial action of $\B$ on $\R$.
Each cube has a unique vertex at which $h$ is minimized (resp. maximized), which we call the {\em bottom} (resp. {\em top}) of the cube. 
Given an integer $k$, write $\mathfrak X^{\le k}$ for the subcomplex of $\mathfrak X$ spanned by those vertices with height  at most $k$. 

\begin{lemma}\label{lem:cocomp} $\B$ acts cocompactly on $\mathfrak X^{\leq k}$  for all $1\leq k<\infty$.
\end{lemma}
\begin{proof} Let $x=[N, f]$ be the bottom vertex of an $m$-cube, where $N$ has $l\le k$ pieces. The manifold $N$ has finitely many boundary components, and extending all of them by the adjacent pieces from the rigid structure defines the maximal cube with the given bottom vertex $x$. 
Therefore, it suffices to show that there are finitely many $\B$-orbits of bottom vertices of $m$-cubes. Since $f^{-1}x=[N, id]$, we can assume that $f=id$. But there are finitely many suited submanifolds containing $O_1$ with  complexity at most $k$, and hence the result follows.
\end{proof}

We now identify the cube stabilizers of the $\mB$-action.  Recall the definition of the sphere-permuting mapping class group (cf. Definition \ref{def:spherepermuting}): 

\begin{lemma}\label{lem:stab}  The cube stabilizers of the action of $\B$ on $\mathfrak X$ are isomorphic to a finite index subgroup of the  sphere-permuting mapping class groups of suited submanifolds. 
\end{lemma}
\begin{proof}  Let $C$ be a $k$-cube in $\mathfrak X$ and denote by $x=[N, g]$ its bottom vertex. Multiplying by $g^{-1}$, we can assume that $g=id$. Note that $C$ is spanned by vertices formed by attaching pieces to a set $A$ of exactly $k$ boundary components of $N$.  An element $f\in \B$ stabilizes $C$ if and only if  $f(N)=N$, $f(A)=A$ and $N$ is a support of $f$. It follows that the stabilizer of $C$ is a subgroup of $\Map_o(N)$. 
In addition, $\Map(N)$ preserves $C$, and thus the result follows. 
\end{proof}

\subsection{Descending links and piece complexes} 
\label{sec:desc_lk_piece}
 After Proposition \ref{prop:contr} and Lemma \ref{lem:cocomp}, and provided mapping class groups of suited submanifolds are of type $F_\infty$ (cf. Section \ref{sec:mcgs}), all that remains is to check the connectivity properties of the pair $(\mathfrak X_{n+1}, \mathfrak X_n)$, where $\mathfrak X_n:=\mathfrak X^{\le n}$. In turn this boils down, using a well-known argument in {\em discrete Morse theory}, to analyzing the connectivity of the {\em descending links}; we refer the reader to Appendix \ref{sec:discreteMorse} for a discussion on this. The advantage of working with descending links is that they may be identified with a simplicial complex built from topological information on (suited submanifolds of) $\mC$, called the {\em piece complex}, and which we introduce next. Before doing so, we need the following. 

\begin{definition}[Essential sphere]
Let $M$ be a connected, orientable smooth $n$-manifold, and $S$ a smoothly embedded $(n-1)$-dimensional sphere in $M$. We say that $S$ is {\em essential} if it does not bound an $n$-dimensional ball and is not isotopic to a boundary component of $M$.  
\end{definition}

We now define the piece  complex; see Figure \ref{fig:piece_complex} for an illustration:

\begin{definition}[Piece complex]
Let $M\subset \mC$ be a suited submanifold, and $A$  a subset of the set of suited spheres of $M$. The \notion{piece complex}  $\mP_d(M,A)$ is the simplicial complex whose vertices are isotopy classes of submanifolds of $M$ diffeomorphic to $Y^d$ (but not necessarily pieces) with one essential boundary sphere and all other boundary spheres in $A$. A $k$-simplex is given by $k+1$ vertices which can be realized in pairwise disjoint manner.
\end{definition}

As mentioned in the introduction, piece complexes (and some close relatives) have been used in the homological stability results  of Hatcher--Wahl \cite{HW05}, Hatcher--Vogtmann \cite{HV17} and Galatius--Randal-Williams \cite{GRW18}. 

\begin{figure}
\includegraphics[scale=0.55]{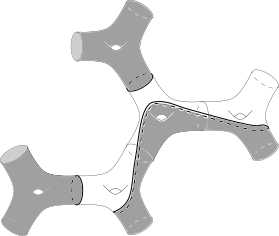}
\caption{A $2$-simplex in $\mP_d(M,A)$ with $M$ a surface of genus $5$ and $A=\partial M$ with $7$ spheres. }
\label{fig:piece_complex}
\end{figure}

We now explain the promised relation between descending links and piece complexes. We remark that the remainder of this section rests upon on arguments that were suggested by Chris Leininger and Rachel Skipper, to whom we are grateful. 

 Let $x$ be a vertex of $\mathfrak X$. By precomposing with an element of $\mathcal B$, we can write $x=[X,{\rm id}]$. We now define a map \[\Pi: \lk^\downarrow(x) \to \mathcal P(X, A),\] where $A$ is the set of all suited boundary spheres of $X$. Let $z$ be a $p$-simplex in $\lk^\downarrow (x)$; as such, it is determined by a cube, which in turn is determined by its top $[Z',g]$ and bottom  $[Z,g]$ vertices, where the manifold $Z'$ is obtained from $Z$ by adding pairwise disjoint pieces $Y_0, \dots, Y_p$ so that $g$ maps $Z'$ to $X$, and $Z'$ is a support for $g$. We then set \[\Pi(z)=\{g(Y_0), \dots, g(Y_p)\}.\]

First, we claim that $\Pi$ is well-defined; to this end, let $(W, h)$ be another representative of $z$. As before, this implies there exists a manifold $W'$ obtained from $W$ by adding pairwise disjoint pieces $Y_0', \dots, Y_p'$ so that $h$ maps $W'$ to $X$. By the definition of the equivalence relation, this means $h^{-1}\circ g$ maps  $Z$ to $W$ and in particular, $g(Z)=h(W)\subseteq X$. Therefore, $\{g(Y_0), \dots, g(Y_p)\}=\{h(Y_0'), \dots, h(Y_p')\}$, and we are done. 

As we now prove, the relevance of the map $\Pi$ is that it makes of descending links {\em complete join complexes}  (see Definition \ref{defn-join} in Appendix \ref{subsection:join})  over the relevant piece complexes:

\begin{proposition}\label{prop:join}
Suppose that $\mathcal C$ has the cancellation property. Then the map $\Pi$ is a complete join.
\end{proposition}

\begin{proof}

Write $x=[X, \rm id]$, and let $\{Z_0,\ldots, Z_p\}$ be a $p$-simplex in $\mathcal P(X, A)$. By the cancellation property, there is a suited manifold $X'$, a collection $\{Y'_0, \ldots, Y'_p\}$ of disjoint pieces in $X'$ and a diffeomorphism $g:\mathcal C\to \mathcal C$ that takes the suited manifold $W':=X'\setminus \cup_{i=0}^p Y'_i$ to $W:=X\setminus \cup_{i=0}^p Z_i$, $g(Y'_i)=Z_i$ and is rigid away from $X'$. Note that $(X',g)\sim (X, \rm id)$. Setting $W_i=X' \setminus Y'_i$, we have that 
\[\Pi(\{(W_0,g), \ldots, (W_p,g)\})=\{Z_0, \dots, Z_p\};\] in particular, $\Pi$ is surjective. The fact that $\Pi$ is simplex-wise injective is obvious from the definitions of the two complexes. 

It remains to show that the pre-image of a simplex is the join of the pre-images of its vertices. Clearly the pre-image of a simplex is contained in the join of the pre-images of its vertices, and so we only need to show the reverse inclusion. To this end, let $\sigma= \{Z_0,\ldots, Z_p\}$ be a $p$-simplex in $\mathcal P(X, A)$. Take vertices \[[W_0, g_0], \ldots, [W_p,g_p]\] of $\mathfrak X$  such that $\Pi([W_i,g_i])=Z_i$; we want to see that these vertices in fact span a $p$-simplex in $\lk^\downarrow(x)$.

Without a loss of generality, we can assume that, for all $i=0, \ldots, p$, we have $W_i\cup Y'_i=X'$ and $g_i=g$. To see this, let $Y_i=g_i^{-1}(Z_i)$, which by construction is a piece, and consider  $g^{-1}\circ g_i\in \mathcal B$ which takes $W_i\cup Y_i$ to $X'$, $Y_i$ to $Y'_i$ and is rigid away from $W_i$. Now we see that $(W_i, g_i)\sim (g^{-1}\circ g_i(W_i), g)$ and this new representative has the desired properties. 

By construction, $W'=\cap_{i=0}^p W_i$ and thus $[W',g]$ is the bottom vertex of the $(p+1)$-cube whose top vertex is $[X',g]$. In particular, the collection $\{[W_i, g]\}_{i=0}^{i=p}$ spans a simplex in $\lk^\downarrow(x)$. 
\end{proof} 

After Hatcher--Wahl \cite{HW05}, a simplicial complex  is {\em weakly Cohen--Macaulay} ($wCM$, for short) of dimension $n$ if it is $(n-1)$-connected and the link of each $p$-simplex is $(n-p-2)$-connected (see Appendix A). As a consequence of the result of Hatcher-Wahl stated as Proposition \ref{prop-cjoin-conn} in  Appendix \ref{subsection:join}, we have:  

\begin{corollary}
Let $x=[X,{\rm id}]$, and $A$ the set of suited boundary components of $X$. If $\mathcal{P}(X,A)$ is $wCM$ of dimension $n$, then so is $\lk^\downarrow(x)$. 
\end{corollary}

The remainder of the paper is devoted to proving the following general theorem, in the different settings required by Theorems \ref{cor:surfaces}, \ref{thm:3D}, \ref{thm:3Dgen} and \ref{thm:highdim}. Before stating the result, let $W_g$ be the connected sum of $O$ and $g$ copies of $Y$, and $W_g^b$ the result of removing from $W_g$ a collection of $b\ge 0$ open $n$-balls with pairwise disjoint closures.  Observe that every suited submanifold of $\mC_{d,r}(O,Y)$ is diffeomorphic to $W_g^b$, for some $g$ and $b$. 

\begin{theorem}
Let $W_g^b$ an $n$-dimensional manifold as in Theorems \ref{cor:surfaces}, \ref{thm:3D}, \ref{thm:3Dgen} and \ref{thm:highdim}, and $A$ a (not necessarily proper) subset of boundary spheres. Then there exist explicit increasing linear functions $\gamma,\delta: \mathbb N \to \mathbb N$ such that $\mP_d(W_g^b, A)$ is a flag $wCM$ complex of dimension $m$, provided $g \ge \gamma(m)$ and $|A|  \ge \delta(m)$. 
\label{thm:genpiececomplex} 
\end{theorem}

Assuming the validity of the above theorem, we can give a proof of our main Theorems  \ref{cor:surfaces}, \ref{thm:3D}, \ref{thm:3Dgen} and \ref{thm:highdim}: 

\begin{proof}[Proof of Theorems \ref{cor:surfaces}, \ref{thm:3D}, \ref{thm:3Dgen} and \ref{thm:highdim}]
We apply Brown's criterion \ref{brown} to the action of $\B$ on the cube complex $\mathfrak X$, equipped with the height function $h: \mathfrak X \to \mathbb R$ defined above. As explained at the beginning of Section \ref{sec:desc_lk_piece}, we do so through the language of discrete Morse theory.

By Proposition \ref{prop:contr}, the cube complex $\mathfrak X$ is contractible. Now, Lemma \ref{lem:cocomp} yields that cube stabilizers are sphere permutting mapping class groups of suited submanifolds, which contain the corresponding mapping class groups as groups of finite index, and therefore have the correct finiteness properties in light of Proposition \ref{prop:extensionF} and Lemma \ref{lem-fin-mcg-surf}, Theorems \ref{thm:3Dfinpres} and \ref{thm-fin-mcg-hbdy}, and Proposition \ref{prop:induction}. Finally, Theorem \ref{thm:genpiececomplex} yields that descending links have the desired connectivity properties for the application of Brown's criterion. 
\end{proof}

The remainder of the paper is devoted to proving Theorem \ref{thm:genpiececomplex}. As will become apparent, the spirit of the proof of Theorem \ref{thm:genpiececomplex} is the same in all cases, and proceeds to establish the desired connectivity bounds of the piece complex from known connectivity results of related complexes that appear in the literature. However, some particular arguments of the proofs become easier in higher dimensions, notably because arcs on a manifold of sufficiently high dimension have trivial combinatorics. For this reason, we will present the higher dimensional case first in all detail, and then highlight the similarities and differences with dimensions 2 and 3.

\subsection{The CAT(0) property} 
\label{subsec:cat}

We end this section by addressing the question of when the cube complex $\mathfrak X$ is a complete CAT(0) space. From here on we endow $X$ with the usual path metric in which all cubes are standard Euclidean unit cubes. Say that an infinite-dimensional cube complex has the {\em ascending cube property} (or is {\em locally finite-dimensional}) if every sequence $(x_i)_{i\in \mathbb N}$ of cubes, with $x_i$ a face of $ x_{i+1}$, is eventually constant. 
\begin{proposition}[{Leary \cite[Theorem A.6]{Leary}}]
Let $\mathfrak X$ be a cube complex. $\mathfrak X$ is complete if and only if $\mathfrak X$ satisfies the ascending cube property.
\end{proposition}

The following is a version of the celebrated Gromov's link condition which also applies to infinite-dimensional cube complexes; for a proof, see \cite[Theorem B.8]{Leary}: 

\begin{proposition}[Gromov]
Let $\mathfrak X$ be a simply-connected cube complex.  $\mathfrak X$ is CAT(0) if and only if the link of every vertex is flag. 
\end{proposition}

We now want to see that our (contractible) complex $\mathfrak X$ satisfies the hypotheses of the above propositions. First, since every suited submanifold has a finite number of boundary components, we immediately get: 

\begin{proposition}
The  complex $\mathfrak X$ has the ascending cube property.
\label{prop:ascending}
\end{proposition}

Next, we prove:

\begin{proposition}\label{prop:flag} Let $x$ be a vertex in $\mathfrak X$. Then $\lk(x)$  is flag if and only if  $\lk^\da x$ is flag.
\end{proposition}
\begin{proof} The forward direction is clear because $\lk^\da x$ is a complete subcomplex of $\lk(x)$. So, assume that $\lk^\da x$ is flag. Let $\sigma$ be an $n$-simplex in the flag completion of $\lk(x)$ whose boundary belongs to $\lk(x)$. We will show that $\sigma$ is in $\lk(x)$.

Suppose $\sigma\cap \lk^\da x$ is nonempty. Since $\lk^\da x$ is flag, this intersection is a $k$-simplex of $\lk(x)$ which we denote by $\tau$. Denote by $C_{\tau}$ the corresponding  $(k+1)$-cube in $\mathfrak X$ and its bottom vertex by $y=[N, f]$. Then there are disjoint pieces $\{Y_i\}_{i=0}^{k}$ such that   $x= [N\cup Y_0\cup \dots \cup Y_{k}, f]$. The vertices of $\sigma$ that are not in $\lk^\da x$ span an $(n-k-1)$-face $\mu$ of $\sigma$ and $\tau \join \mu=\sigma$. The corresponding  $(n-k)$-cube $C_{\mu}\subset\mathfrak X$ has $x$ as its bottom vertex and there are disjoint pieces $\{Z_i\}_{i=0}^{n-k-1}$ such that 
$$z=[N\cup Y_0\cup \dots \cup Y_{k}\cup Z_0\cup \dots \cup Z_{n-k-1}, f]$$ 
is the top vertex of $C_{\mu}$. Since $\sigma$ is a join of $\tau$ and $\mu$, the pieces $\{Y_i\}_{i=0}^{k}$ and $\{Z_i\}_{i=0}^{n-k-1}$ are disjoint. It follows that there is an $(n+1)$-cube $C$ such that $y$ and $z$ are its bottom and top vertices, respectively. Since  the cubes $C_{\tau}$, $C_{\mu}$ are faces of $C$ having $x$ as a common vertex, we have $\sigma\subset \lk(x)$ and  $C_{\sigma}=C$. 

Suppose now $\sigma\cap \lk^\da x = \emptyset$. In this case, we set $\tau$ to be the empty simplex and repeat the above argument with $x=y$ as the bottom vertex of~$C$.
\end{proof}

Applying Gromov's link condition and Proposition \ref{prop:flag}, we immediately obtain:

\begin{corollary}\label{cor:flag}If the descending links of all the vertices in $\mathfrak X$ are flag, then $\mathfrak X$ is a complete CAT(0) cube complex.
\end{corollary}

In the next sections, we will prove that piece complexes are flag, for the concrete families of manifolds listed in our applications. Once this is done, we will be able to deduce that descending links are also flag, and in particular deduce Corollary \ref{cor:flag}, from the following general observation: 

\begin{lemma}
Let $\pi:L \to K$ be a complete join of simplicial complexes. If $K$ is flag, then so is $L$.     
\end{lemma}

\begin{proof}
Let $y_0, ..., y_d$ be pairwise adjacent vertices in $L$; we have to
show that $\{ y_0, ..., y_d \}$ is a simplex. The images $\pi(y_i)$ are also pairwise adjacent, since $\pi$ is injective on simplices, and therefore span a simplex in $K$, as $K$ is flag. Since $\pi$ is a complete
join, the $y_i$ span a simplex in $L$, as claimed.
\end{proof}

\section{The piece complex in high dimensions} 
\label{sec:highdim}

Throughout this section, $n\ge 3$, $O\cong \mathbb S^{2n}$, and $Y\cong \bS^n \times \bS^n$ or $Y\cong \bS^{2n}$. Assume first $Y\cong \bS^n \times \bS^n$; the case $ Y\cong \mathbb S^{2n}$ follows from a simplified version of our arguments and will be dealt with at the end of the section. 
 
As above, $W_g$ denotes the connected sum of $O$ and $g$ copies of $Y$, and $W_g^b$ is the result of removing from $W_g$ a collection of $b\ge 0$ open $2n$-balls with pairwise disjoint closures.  Our aim is to prove: 

\begin{theorem}\label{thm:main-high-dim}
Let $W_g^b$ as above, and $A$ a (not necessarily proper) subset of boundary spheres of $W^b_g$. Then $\mP_d(W_g^b, A)$ is flag. Moreover, $\mP_d(W_g^b,A)$ is $wCM$ of dimension $m$, provided that $m \leq \frac{g-2}{2}$ and $m \leq \frac{|A|-d}{d+1}$.
\end{theorem}
We will prove Theorem~\ref{thm:main-high-dim} by relating $\mP_d(W_g^b,A)$ to several other complexes, which we proceed to introduce.

\subsection*{The ball complex}
First, consider the simplicial complex $\dball{(W_g^b,A)}$ whose vertices are isotopy classes $B$ of
 $2n$-balls with $d$ holes in $W_g^b$, all contained in the set $A$.  A $k$-simplex in $\dball(W_g^b,A)$ corresponds to a set of $k+1$ vertices with pairwise disjoint representatives. We denote by $\partial^0 B$ the $(2n-1)$-dimensional sphere which cuts off $B$ from $W_g^b$, observing that $\partial^0 B$ is essential.
 
 We stress that we will encounter the same construction in dimensions 3 and 2 as well; in all cases the complex $\dball{(W_g^b,A)}$ is closely related to the so called {\em $d$-hypergraph complex} $M_{|A|}(d)$, namely the simplicial complex whose vertices are  subsets of $\{1,\ldots,|A|\}$ with $d$ elements, and whose faces correspond to pairwise disjoint $d$-sets. There is a natural map $\alpha: \dball(W_g^b,A)\rightarrow M_{|A|}(d)$; in high dimensions, it is even an isomorphism:

\begin{lemma}\label{lem:conn-d-hole-cpx-high}
The map $\alpha: \dball(W_g^b,A)\rightarrow M_{|A|}(d)$ is an isomorphism.
\end{lemma}
\begin{proof}
Given $d$ distinct spheres $S_1,\cdots,S_d$ in $A$, one can find a ball containing them  by first  connecting $S_{i}$ to $S_{i+1}$ for $i\leq \{1,\cdots,d-1\}$ using $d-1$ disjointly embedded arcs, and then taking a small neighborhood of the union of these arcs together with the $d$ spheres. Since $W_g^b$ is simply connected, these arcs are unique up to isotopy, and can  be chosen to be pairwise disjoint. As a consequence, $\alpha$ is a surjective simplicial map. 
 
To see that $\alpha$ is injective, suppose there are two embedded balls $B, B'$ each containing $d$ spheres $S_1,\cdots,S_d$ in $A$. Roughly as above, we choose a set of $d-1$ arcs $a_i$ (resp. $a_i'$) in $B$ (resp. $B'$) connecting $S_{i}$ to $S_{i+1}$ for $i\leq \{1,\cdots,d-1\}$. We can now choose a small neighborhood of the union of the $d$ spheres and the arc $a_i$ (resp. $a_i'$) so that it lies in the interior of  $B$ (resp. $B'$). Since $W_g^b$ is simply-connected, we may isotope $a_i $ to $a_i'$. Applying a further isotopy that shrinks $B$, we can assume that $B$ lies in the interior of $B'$. By the $h$-cobordism theorem \cite{MilCobnor}, the manifold between the boundary of $B$ and $B'$ must be a cylinder, thus $B$ and $B'$ are in fact isotopic inside $W_g^b$.
\end{proof}

The  $d$-hypergraph complex $M_{|A|}(d)$ is a flag complex,  known to be weakly Cohen--Macaulay of dimension $\left\lfloor\frac{|A|-d}{d+1}\right\rfloor$, see \cite{At04}. Thus, we have:
\begin{corollary}\label{cor:ball-cx-conn-high-dim} 
$\dball(W_g^b,A)$ is flag, and $wCM$ of dimension $\left\lfloor\frac{|A|-d}{d+1}\right\rfloor$.
\end{corollary}

\subsection*{The handle complex} The second complex we will consider is the \emph{handle complex}, which we denote $\mH(W_g^b)$. The vertices of $\mH(W_g^b)$ are isotopy classes of smoothly embedded separating spheres which bound a {\em handle}, i.e. a copy of $\bS^n\times \bS^n$ minus one ball, and $k$-simplices correspond to sets of $k+1$ vertices which can be realized in a pairwise disjoint manner, see Figure \ref{fig:handle_complex}. (This can be interpreted in two ways: the spheres can be realized disjointly, or the handles which they bound can be realized disjointly. By Lemma \ref{lem:TwoDefsAreEqual} these are the same.) In what follows, we will blur the distinction between vertices in the handle complex, their representatives, and the handles cut off by them.


\begin{figure}
\includegraphics[scale=0.55]{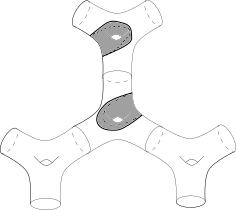}
\caption{A $1$-simplex in $\mH(M)$ with $M$ a surface of genus $4$ with $6$ boundary spheres.}
\label{fig:handle_complex}
\end{figure}

We shall make use of three main pieces of information about $\mH(W_g^b)$. The first is a cancellation result due to Kreck \cite{Kr99}; see also \cite[Corollary 6.4]{GRW18}:
\begin{theorem}\label{thm:cancellation}
Let $W_g^b$ as above, and let $S$ be a separating $(2n-1)$ sphere that cuts off a handle. Then, the other connected component of the complement of $S$ is diffeomorphic to $W_{g-1}^{b+1}$. 
\end{theorem}

\begin{remark}
An immediate consequence of Theorem \ref{thm:cancellation} is that $\mC_{d,r}(O,Y)$ has the cancellation property, where $Y \cong \bS^n \times \bS^n$ and $O$ is any closed, simply-connected orientable manifold of dimension $2n$. \end{remark}

The other two are provided to us by Randal-Williams (see Appendix~C).

\begin{lemma}\label{lemma:flag-in-high-dim}
$\mH(W_g^b)$ is a flag complex.
\end{lemma}

Together with Theorem~\ref{thm:cancellation}, this implies that the link of a $k$-simplex in $\mH(W_g^b)$ is isomorphic to $\mH(W_{g-k-1}^{b+k+1})$. Thus, we have: 

\begin{theorem}
\label{thm:RW}
 $\mH(W_g^b)$ is $\left\lfloor\frac{g-4}{2}\right\rfloor$-connected, and $wCM$ of dimension $\left\lfloor\frac{g-2}{2}\right\rfloor$.
\end{theorem}

In what follows, we will not distinguish between vertices of $\mH(W_g^b)$ or $\mP_d(W_g^b,A)$ and the manifolds representing them.

\subsection*{The handle-tether-ball complex}
Let $H$ be a vertex of $\mH(W_g^b)$ and $B$ be a vertex in $\dball(W_g^b,A)$. Following the terminology of Hatcher--Vogtmann \cite{HV17}, a {\em tether} connecting $H$ and $B$ is an arc  with one endpoint in $\partial H$ (an essential sphere cutting off a handle) and the other in $\partial^0 B$ (an essential sphere cutting off a ball with $d$ boundary components from $A$), and whose interior does not intersect neither $H$ nor $B$. We refer to the union of $H$, the arc, and $B$ as a {\em handle-tether-ball}. Let $\htb(W_g^b,A)$ be the handle-tether-ball complex,  whose $k$-simplices are  isotopy classes of pairwise disjoint systems of $k+1$ handle-tether-balls; in particular, we require the tethers to be disjoint. 
There is a natural projection map \[\pi:\htb(W_g^b,A)\rightarrow \mP_d(W_g^b,A)\]  given by mapping a handle-tether-ball to the copy of $\bS^n\times \bS^n$ minus $d+1$ balls obtained by taking a small tubular neighborhood of the handle-tether-ball.

Again, we will encounter this construction also in dimensions 3 and 2. Generally, \(\pi:\htb(W_g^b,A)\rightarrow \mP_d(W_g^b,A)\) is a complete join complex in the sense of Appendix \ref{sec:connectivitytools}. In high dimensions, we encounter an even simpler situation. Very much for the same reason that the ball complex is isomorphic to the hypergraph complex, the map $\pi$ is an isomorphism.

\begin{lemma}\label{lem-conn-totd-3holt}
The map $\pi:\htb(W_g^b,A)\rightarrow\mP_d(W_g^b,A)$ is a an isomorphism. In particular, if $\htb(W_g^b,A)$ is $wCM$ of dimension $m$, then so is $\mP_d(W_g^b,A)$.
\end{lemma}
\begin{proof}
First, we show that $\pi$ is a complete join. This part of the argument is worded so that it
applies in dimension 3 and 2, as well. It will be referenced there.

Let $\sigma=\langle P_0,\ldots,P_k\rangle$ be a $k$-simplex in $\mP_d(W_g^b,A)$. Choose an essential sphere $S_i\subset P_i$ which cuts a ball $B_i \subset P_i$ containing exactly the $d$ boundary spheres contained in $A$, a submanifold $P_i'\subset P_i$ diffeomorphic to $(\bS^n\times \bS^n) \setminus \mathbb B^{2n}$ and disjoint from $B_i$, and a tether $t_i$ connecting $B_i$ and $P_i'$. Then thickening the tether yields a submanifold of $W_g^b$ diffeomorphic to $P_i$. Since the $P_0,\ldots,P_k$ have pairwise disjoint representatives, the same holds for the handle-tether-balls $P_i'\sqcup t_i\sqcup B_i$ for $0\leq i\leq k$; in other words, these handle-tether-balls span a $k$-simplex in $\htb(W_g^b,A)$. Hence the map $\pi$ is surjective and in particular $\pi^{-1}(\sigma)=\ast_{i=0}^k\pi^{-1}(P_i)$. 

Further, if $v_0,\ldots,v_k$ form a $k$-simplex in $\htb(W_g^b,A)$, the $B_i$ contain pairwise distinct and disjoint boundary spheres, all contained in $A$. Thus $\pi(v_i)$ are pairwise distinct and $\pi$ is injective on simplices; in particular, $\pi$ is a complete join.

To see that $\pi$ is actually an isomorphism, observe that, within the submanifold $P_i$, the essential sphere $S_i$ is unique up to isotopy for the same reason as in Lemma~\ref{lem:conn-d-hole-cpx-high}. Also up to isotopy, there is a unique embedded arc connecting $S_i$ to the boundary $\partial^0 P_i$. Now, $P'_i$ is isotopic to the complement of a regular neighborhood (inside $P_i$) of the union of the arc, $S_i$, and $\partial^0 P_i$. Therefore, $P_i$ determines the ball and the handle. Finally, the tether is unique up to isotopy. Thus, the handle-tether-ball data describing $P_i$ is unique and $\pi$ is injective. 
\end{proof}

Again, we stress that, in the high dimensional case, the tether in a handle-tether-ball system is uniquely determined (up to isotopy) by the handle and the ball: the ambient manifold and all pieces involved have trivial fundamental groups. Even more is true: in dimension 4 and higher, arcs do not form obstructions for the movement of other arcs: all braids are trivial. Observe also that a ball $B \in \htb(W_g^b,A)$ may be realized as follows: first, we connect the $d$ boundary spheres of $W_g^b$ that $B$ cuts off by a collection of $d-1$ arcs so that the union of the arcs and the spheres is connected. Then, $B$ is isotopic to the boundary of a regular neighborhood of the union of the spheres and the arcs.  From this, the following description of the handle-tether-ball complex in the high-dimensional case is immediate.

\begin{observation}\label{obs:htbcx-in-high-dim}
Vertices in  $\htb(W_g^b,A)$ can be uniquely described as pairs $(H,B)$ of vertices $H\in\mH(W_g^b)$ and $B\in\dball(W_g^b,A)$. In this description, a $k$-simplex in  $\htb(W_g^b,A)$ is given by a set $\sigma=\{(H_0,B_0),\ldots,(H_k,B_k)\}$ of vertex pairs such that $\{H_0,\ldots,H_k\}$ is a $k$-simplex in the handle complex $\mH(W_g^b)$ and $\{B_0,\ldots,B_k\}$ is a $k$-simplex in the ball complex $\dball(W_g^b,A)$. 


In particular, there are simplicial projection maps 
\[
  \pi_h : \htb(W_g^b,A) \to \mH(W_g^b)
\]
onto the handle complex and
\[
  \pi_b: \htb(W_g^b,A) \to \dball(W_g^b,A)
\]
onto the ball complex, such that every $k$-simplex in $\dball(W_g^b,A)$ and every $k$-simplex in $ \mH(W_g^b)$ may be combined into a $k$-simplex in $\htb(W_g^b,A)$.
\end{observation}
As $\mH(W_g^b)$ and $\dball(W_g^b,A)$ are both flag complexes (see Lemma~\ref{lemma:flag-in-high-dim} and Corollary~\ref{cor:ball-cx-conn-high-dim}), we deduce from the description of $\htb(W_g^b,A)$ above the following:
\begin{corollary}\label{cor:htb-is-flag-high-dim}
The complex $\mP_d(W_g^b,A)\cong\htb(W_g^b,A)$ is a flag complex.
\end{corollary}

As another consequence of Observation~\ref{obs:htbcx-in-high-dim} we can determine links within the handle-tether-ball complex.
\begin{lemma}\label{lem:link-htb-high-dim}
Let $\tau=\{(H_0,B_0),\ldots,(H_k,B_k)\}$ be a $k$-simplex in the complex $\htb(W_g^b,A)$. Then $\Link(\tau)$ is isomorphic to $\htb(W_{g-k-1}^{b-(k+1)(d-1)}, A^*)$ where $A^*$ is obtained from $A$ by removing the $(k+1)d$ boundary spheres in $\tau$. 
\end{lemma}

In the above lemma, the number of boundary components to $W$ changes only by  $(k+1)(d-1)$ as we also get a new boundary sphere for each of the $k+1$ handles that we cut out.

\begin{proof}
A coface of $\tau$ has the form $\sigma=\tau\union\tau'$ with $\tau'=\{(H'_0,B'_0),\ldots,(H'_l,B'_l)\}$ where $\{H'_0,\ldots,H'_l\}$ is an $l$-simplex in the link of $\pi_h(\tau)$ and $\{B'_0,\ldots,B'_l\}$ is an $l$-simplex in the link of $\pi_b(\tau)$. The claim follows as the link of $\pi_h(\tau)$ in the handle complex is of the form $\mH(W_{g-k-1}^{b-(k+1)(d-1)})$ and the link of $\pi_b(\tau)$ in the $d$-hypergraph complex based on $A$ is the $d$-hypergraph complex based on~$A^*$.
\end{proof}

\subsection*{The extended handle-tether-ball complex}
In order to investigate the connectivity properties of the handle-tether-ball complex, we embed it in a larger complex $\htbs(W_g^b, A)$, which we call the \emph{extended handle-tether-ball complex} $\htbs(W_g^b, A)$,  whose connectivity is easier to analyze. More concretely, $\htbs(W_g^b, A)$ is the simplicial complex with the same vertex set as $\htb(W_g^b,A)$, and where we add more simplices by allowing multiple tethers connecting different handles to the same $d$-holed ball. In light of Observation~\ref{obs:htbcx-in-high-dim} we can give a purely combinatorial description of $\htbs(W_g^b,A)$ in this case.
\begin{observation}
A $k$-simplex in $\htbs(W_g^b,A)$ is given by a set $\sigma=\{(H_0,B_0),\ldots,(H_k,B_k)\}$ of vertices $(H_i,B_i)$ such that $\{H_0,\ldots,H_k\}$ is a $k$-simplex in the handle complex and $\{B_0,\ldots,B_k\}$ is a simplex (of arbitrary dimension!) in the ball complex.
\end{observation}

The extended complex $\htbs(W_g^b,A)$ is easier to analyze because the canonical projection (using a slight abuse of notation)
\[
  \pi_b : \htbs(W_g^b,A) \longrightarrow \dball(W_g^b,A)
\]
has well-behaved fibers. Combinatorially, $\pi_b$ is the projection on the second coordinate; geometrically, $\pi_b$ is the map that forgets the handle-tether part of a handle-tether-ball datum.
\begin{lemma}\label{lem:conn-htbs-high-dim}
If $m \leq \frac{g-4}{2}$ and $m \leq \frac{|A|-d}{d+1}-1$, then  $\htbs(W_g^b,A)$ is $m$-connected.
\end{lemma}
\begin{proof}
Let $\beta=\{B_0,\ldots,B_j\}$ be a $j$-simplex in $\dball(W_g^b,A)$. The fiber $\pi_b^{-1}(\beta)$ is spanned by the vertices $(H,B)$ with $B\in\beta$. Restricted to the fiber, the projection $\pi_h : \htbs(W_g^b,A) \rightarrow \mH(W_g^b)$ is a complete join with all fibers isomorphic to $\beta$. It follows from Lemma~\ref{thm:RW} that the fiber $\pi_b^{-1}(\beta)$ is $m$-connected.

The base space $\dball(W_g^b,A)$ of the projection $\pi_b$ is also $m$-connected. It follows from Quillen's Fiber Theorem~ (stated as \ref{app:quillen} in Appendix \ref{sec:connectivitytools}) that the total space $\htbs(W_g^b,A)$ is also $m$-connected.
\end{proof}

\subsection*{The connectivity of the handle-tether-ball-complex.}
Next, we use the {\em bad simplex} argument~\cite[Section~2.1]{HV17} in order to deduce the connectivity of the handle-tether-ball complex. We give an account of the argument via Morse theory in Proposition~\ref{app:gated-morse-theory} of the appendix. The particular flavor used here is explained in Example~\ref{app:colored-vertices}; and we shall use the projection $\pi_b : \htbs(W_g^b,A)\to\dball(W_g^b,A)$ as the coloring.

Observe that $\htb(W_g^b,A)$ is the subcomplex of $\htbs(W_g^b,A)$ consisting of those simplices on which $\pi_b$ is injective, i.e., the good simplices. Further, we can also describe the good links:
\begin{lemma}
Let $\sigma=\{(H_0,B_0),\ldots,(H_k,B_k)\}$ be a bad simplex. Its good link $G_\sigma$ is isomorphic to a handle-tether-ball complex $\htb(W_{g-k-1}^c,A^*)$ where $W_{g-k-1}^c$ is obtained from $W_g^b$ by excising the handles $H_i$ and the balls $B_i$. The set $A^*$ of admissible boundary spheres is induced by removing from $A$ those $s\leq kd$ boundary spheres of $W_g^b$ contained within the $B_i$.
\end{lemma}
\begin{proof}
The proof is very similar to the proof of Lemma~\ref{lem:link-htb-high-dim}. The only difference is that $\pi_b$ is not injective on $\sigma$. Thus the simplex $\pi_b(\sigma)$ has strictly smaller dimension than $\sigma$. Thus,  $s\leq kd$ elements of $A$ are used in $\sigma$.

Consider a simplex $\tau=\{(H'_0,B'_0),\ldots,(H'_l,B'_l)\}$ in the link of $\sigma$. As seen above, $\tau$ belongs to the good link of $\sigma$ if and only if  $\{B'_0,\ldots,B'_l\}$ is an $l$-simplex in the link of $\pi_B(\tau)$ in $\dball(W_g^b,A)$, which is the $d$-hypergraph complex on $A^*$.

For the handle part, the condition is that $\{H'_0,\ldots,H'_l\}$ lies in the link of $\{H_0,\ldots,H_k\}$. The link in the handle complex is isomorphic to $\mH(W_{g-k-1}^c)$.
\end{proof}

\begin{proposition}\label{prop:main-conn-high-dim}
If $m \leq \frac{g-4}{2}$ and $m \leq \frac{|A|-d}{d+1}-1$, then  $\htb(W_g^b,A)$ is $m$-connected. 
\end{proposition}
\begin{proof}
We induct on $m$. The extended complex $\htbs(W_g^b,A)$ is $m$-connected by Lemma~\ref{lem:conn-htbs-high-dim}. The complex $\htb(W_g^b,A)$ is the complex of good simplices in $\htbs(W_g^b,A)$. By Proposition~\ref{app:gated-morse-theory}, it suffices to show that good links $G_\sigma$ of bad simplices are $(m-\dim(\sigma))$-connected. Let $\sigma$ be a bad $k$-simplex. Then $k\geq 1$ and $G_\sigma\cong\htb(W_{g-k-1}^c,A^*)$ by the preceding lemma.

For the induction hypothesis to apply, we need to verify $m - k \leq \frac{(g-k-1)-4}{2}$ and $m - k \leq \frac{|A^*|-d}{d+1}-1$. For the first estimate, note that $m-k \leq \frac{g-2k-4}{2}$ follows from the assumption $m\leq \frac{g-4}{2}$. As $k\geq1$, we have $2k\geq k+1$, whence $m-k\leq\frac{g-(k+1)-4}{2}$ follows. Turning to the second inequality, note that $|A^*| \geq |A|-kd$. Thus, we find:
\[
  \frac{|A^*|-d}{d+1}-1 \geq \frac{|A|-kd-d}{d+1}-1 \geq \frac{|A|-d}{d+1}-1-k \geq m-k
\]
\end{proof}

\begin{corollary}\label{cor:main-wcm-high-dim}
Assume that $m \leq \frac{g-2}{2}$ and $m \leq \frac{|A|-d}{d+1}$. Then the handle-tether-ball complex $\htb(W_g^b,A)$ is $wCM$ of dimension $m$. Consequently,  the piece complex $\mP_d(W_g^b,A)$ is also $wCM$ of dimension $m$.
\end{corollary}
\begin{proof}
This follows from the description of links in $\htb(W_g^b,A)$ given in Lemma~\ref{lem:link-htb-high-dim} and the isomorphy of $\htb(W_g^b,A)$ and $\mP_d(W_g^b,A)$.
\end{proof}

We conclude this section with a diagram showing the main actors involved:
\begin{equation}\label{fig:complexes}
\begin{tikzcd}
 \htbs(W_g^b,A) \arrow[d,"\pi_b"] & \htb(W_g^b,A) \arrow[l,"\iota"] \arrow[r,"\pi"] & \mP_d(W_g^b,A)
\\
\dball(W_g^b,A) \arrow[r,"\alpha"] & M_{|A|}(d)  &  & 
\end{tikzcd}
\end{equation}

Finally, we prove Theorem \ref{thm:main-high-dim}: 

\begin{proof}[Proof of Theorem~\ref{thm:main-high-dim}]
  For $Y \cong \bS^n \times \bS^n$, the claim follows from Corollary~\ref{cor:main-wcm-high-dim} and Corollary~\ref{cor:htb-is-flag-high-dim}.
  The case of $Y\cong \bS^{2n}$ is a lot simpler: observe that in this case 
  \[
    \mP_d(W_g^b,A)\cong\dball(W_g^b,A)\cong M_{|A|}(d),
  \] 
  and we can apply Corollary~\ref{cor:ball-cx-conn-high-dim}.
\end{proof}


\newcommand{\TheHandle}{H}
\newcommand{\TheBall}{B}
\section{The piece complex for 3-manifolds}
\label{sec:3dim}
In this section we prove a version of Theorem \ref{thm:genpiececomplex} for 3-dimensional manifolds. In this setting, $W_g = O \# Y \# \cdots^{(g)} \# Y$,
where $O$ and $Y$ are compact, connected, oriented, smooth 3-manifolds, with $Y$ either irreducible or diffeomorphic to $\bS^2\times \bS^1$. The manifold $W_g^b$ is then obtained by removing a collection of $b\ge 0$ open balls (with pairwise disjoint closures) from $W_g$. With this notation, we will show:
\begin{theorem}\label{thm:main-3d}
Let $W_g^b$ as above, and $A$ a (not necessarily proper) subset of boundary spheres of $W^b_g$. Then  $\mP_d(W_g^b, A)$ is flag. Moreover, $\mP_d(W_g^b,A)$ is $wCM$ of dimension $m$ provided that $m \leq \frac{g}{4}$ and $m \leq \frac{|A|-d}{d+1}$.
\end{theorem}

The proof proceeds along the exact same lines as in the previous subsection, although certain steps become more intricate; for instance, the manifolds are no longer assumed to be simply-connected and therefore one has to be careful when dealing with systems of arcs up to isotopy. We now proceed to give a complete account of the ideas needed; however, we will omit the proofs of some results if these can be directly transplanted from the previous subsection without any modification. 

\subsection*{The ball complex}
As was the case in higher dimensions, we  have a map
\[  
  \alpha: \dball(W_g^b,A)\rightarrow M_{|A|}(d)
\]
induced by mapping any given $3$-ball with $d$ holes, all of them elements of $A$, to the corresponding $d$ points in the set $\{1,2,\ldots,|A|\}$. For this map, the same proof as in Lemma \ref{lem:conn-d-hole-cpx-high} gives:
\begin{lemma}\label{lem:conn-d-hole-cpx-3}
The map $\alpha: \dball(W_g^b,A)\rightarrow M_{|A|}(d)$ is a complete join. 
\end{lemma}
As mentioned in the previous section, we cannot expect this map to be an isomorphism as $W_g^b$ need not be simply connected; so even for $d=2$ there may be many ways to enclose two boundary spheres from $A$ in a ``sausage.''

Next, since $M_{|A|}(d)$ is flag and $wCM$ of dimension $\left\lfloor\frac{|A|-d}{d+1}\right\rfloor$, we  deduce:
\begin{corollary}\label{cor:ball-cx-conn-3d}
$\dball(W_g^b,A)$ is flag and $wCM$ of dimension $\left\lfloor\frac{|A|-d}{d+1}\right\rfloor$.
\end{corollary}
Note that this is the  same statement as Corollary~\ref{cor:ball-cx-conn-high-dim}, including the bounds.

\subsection*{The handle complex} 
For the sake of notational consistency with the previous section, we define a {\em handle} to be a submanifold of $W_g^b$ diffeomorphic to $Y\setminus \mathbb B^3$. The {\em handle complex} $\mH(W_g^b)$ is the simplicial complex whose $k$-simplices are sets of $k+1$ isotopy classes of handles in $W_g^b$ that can be represented in a pairwise disjoint manner. 

As in the previous section, our goal is to establish connectivity properties of the handle complex. We do that by relating the handle complex to the \notion{(non-separating) sphere complex}, whose connectivity is known after the work of Hatcher--Wahl \cite{HW05}.
Recall that an $l$-simplex in the sphere complex $\sph{W_g^b}$ is a set of $l+1$ isotopy classes of essential spheres in $W_g^b$ with pairwise disjoint representatives. The {\em non-separating sphere complex} $\nssph{W_g^b}$ is the full subcomplex of $\sph{W_g^b}$ spanned by those systems of spheres whose complement is still connected. We will need the following result of Hatcher--Wahl, see Proposition 3.2 of \cite{HW05}: 
\begin{theorem}[\cite{HW05}]
Let $W_g^b$ be as above. Then $\nssph{W_g^b}$ is $(g-2)$-connected. 
\label{thm:spherecomplex}
\end{theorem}

Armed with this, we have the following:
\begin{proposition}
With the notation above, $\mH(W_g^b)$ is:
\begin{enumerate}
    \item $(g-2)$-connected, if $Y \ne \bS^2 \times \bS^1$; 
    \item $\left\lfloor\frac{g-3}{2}\right\rfloor$-connected, if $Y=\bS^2\times \bS^1$. 
    \end{enumerate}
\label{prop:P1-3dim}
\end{proposition}

\begin{proof}
If $Y \ne \bS^2\times \bS^1$, the handle complex is the complex $X$ in \cite[Section 4.1]{HW05}, and the  connectivity bound is established in \cite[Proposition 4.1]{HW05}. Therefore, from now on we will assume that $Y=\bS^2 \times \bS^1$; in this case, our complex is an unbased version of the complex called $X^A$ in \cite[Section 4.2]{HW05}; in fact our connectivity bound coincides with that of \cite[Proposition 4.5]{HW05}. 

We argue by induction on $g$ again. First,  $\mH(W_g^b)$ is non-empty as long as $g\geq1$, so we assume $g\geq2$ and that the statement is true for any $g'<g$. 

Observe that there is a unique isotopy class of essential spheres contained in $\bS^2 \times \bS^1 \setminus \mathbb B^3$. Therefore, we have a forgetful map: 
\[
F: \mH(W_g^b) \rightarrow \CS^{ns}(W_g^b)
\]
which maps each vertex of  $\mH(W_g^b)$ to the corresponding essential sphere contained in it. At this point, we want to use Theorem \ref{thm:spherecomplex} to deduce a connectivity bound for $\mH(W_g^b)$ from that of $\CS^{ns}(W_g^b)$, showing that the map $F$ is a join complex, again in the sense of Appendix \ref{sec:connectivitytools}. To this end, we first need to check that $\CS^{ns}(W_g^b)$ is $wCM$. To see this, let $\sigma=\langle S_0,\ldots,S_p\rangle$ be a $p$-simplex in $\nssph{W_g^b}$, whose link $\lk_{\nssph{W_g^b}}(\sigma)$ is isomorphic to $\nssph{W_g^b\setminus\sigma}$. By the Prime Decomposition Theorem, the number of summands diffeomorphic to $\bS^2\times \bS^1$ in $W_g^b\setminus\sigma$ is $g-(p-1)$. Combining this with Theorem~\ref{thm:spherecomplex}, we obtain that  $\lk_{\nssph{W_g^b}}(\sigma)$ is $(g-p-3)$-connected. Thus $\nssph{W_ g^b}$ is $wCM$ of dimension $g-1$. 

We now prove that $F$ is a join complex. For every non-separating sphere $S_i$ in $\sigma$, we choose a point $x_i$ in $S_i$, which gives rise to two points when we cut $W_g^b$ along $\sigma$. Recall that, by the definition of the non-separating sphere complex, the manifold $W_g^b\setminus\sigma$ is connected, and thus we can connect the two copies of $x_i$ by an embedded path for each $0\leq i\leq k$. Moreover, up to modifying the paths by an isotopy, we may assume that no two of them intersect. Upon regluing, the  paths yield loops $\alpha_i$ that intersect transversely with $S_i$ exactly once. Taking a tubular neighborhood of $\alpha_i\cup S_i$, we get a manifold diffeomorphic to $\bS^2 \times \bS^1 \setminus \mathbb B^3$ which maps to $\sigma$ under $F$, and so $F$ is surjective. 

Further, $F$ is injective on individual simplices since if two handles are disjoint in $W_g^b$, the non-separating spheres contained in their interior are non-isotopic. In fact, given any handle $H$, we have an embedded loop  intersecting the non-separating sphere transversely with algebraic intersection number $1$ in $H$. If $F(H) = F(H')$ for another handle $H'$, then this loop of $H$ must also intersect the non-separating sphere of $H'$ with algebraic intersection number $1$. In particular $H$ and $H'$ cannot be disjoint.

Next, let $\sigma'=\langle H_0,\ldots,H_p\rangle$ be a $p$-simplex in $\mH(W_g^b)$ such that $F(H_i)=S_i$. Since each $H_i$ is a handle, we may assume that the essential spheres contained in the $H_i$ are pairwise disjoint. Using that $\sigma'$ is a $p$-simplex, we can isotope the $\bS^1$ parts of the $H_i$ such that they do not intersect with any other $H_j$. Moreover, since $W_g^b$ is three-dimensional, we may apply an isotopy to further assume that the $\bS^1$ parts of the $H_i$ are disjoint. Consequently, part (3) of the Definition \ref{defn-join} is also satisfied. 

This discussion proves that $F$ is a join complex over the $wCM$ complex $\nssph{W_g^b}$. Further we check that $F(\lk(\sigma'))=\lk(F(\sigma'))$. The handles in a simplex $\sigma''$ of $\lk(\sigma')$ may be isotoped such that they do not intersect $\sigma'$. In particular, the interior spheres in $\sigma''$ can be chosen to avoid intersecting the interior spheres in $\sigma'$,  hence \[F(\lk(\sigma'))\subseteq \lk(F(\sigma')).\] Using the same argument as for proving that $F$ is surjective, we deduce the reverse inclusion, and thus \[F(\lk(\sigma'))=\lk(F(\sigma'))\] and, since $\sigma'$ is a $p$-simplex, it follows that $F(\sigma')$ is also of dimension~$p$. Hence $F(\lk(\sigma'))$ is $wCM$ of dimension $g-p-2$. Now Theorem \ref{thm-conn-join} implies that $\mH(W_g^p)$ is $\lfloor\frac{g-1}{2}-1\rfloor$-connected if $Y=\bS^2\times \bS^1$. 
\end{proof}

Worsening the bounds allows us to formulate a unified statement that has the same connectivity estimate as Theorem~\ref{thm:RW}:
\begin{corollary}\label{cor:RW-3d}
For any $Y$ as above, $\mH(W_g^b)$ is $\lfloor\frac{g-4}{2}\rfloor$-connected. In fact, $\mH(W_g^b)$ is $wCM$ of dimension $\lfloor\frac{g-2}{2}\rfloor$.
\end{corollary}
\begin{proof}
The connectivity bound follows from Proposition~\ref{prop:P1-3dim}. The observation that the link of a $k$-simplex $\sigma$ in the handle complex $\mH(W_g^b)$ is again a handle complex establishes the connectivity of links.
\end{proof}

\subsection*{The handle-tether-ball complex}
As in the previous section, a \notion{handle-tether-ball} consists of a handle $\TheHandle$ representing a vertex in the handle complex $\mH(W_g^b)$, a $d$-holed ball $\TheBall$ disjoint from $\TheHandle$ and representing a vertex in $\dball(W_g^b,A)$, and a tether connecting $\TheHandle$ to $\TheBall$ disjoint from both. A vertex in $\htb(W_g^b,A)$ is the isotopy class of a handle-tether-ball, and $k+1$ vertices from a $k$-simplex if they can be represented  by pairwise disjoint handle-tether-balls. Passage to a regular neighborhood defines a projection
\[
  \pi:\htb(M_g^b,A)\rightarrow\mP_d(W_g^b,A).
\]

\begin{lemma}\label{lem-conn-totb-3holt-3D}
The map $\pi:\htb(M_g^b,A)\rightarrow\mP_d(W_g^b,A)$ is a complete join. In particular, if $\htb(W_g^b,A)$ is $m$-connected, so is $\mP_d(W_g^b,A)$. 
\end{lemma}
\begin{proof}
The claim that $\pi$ is a complete join follows as in the first step in the proof of Lemma~\ref{lem-conn-totd-3holt}. The connectivity then follows from Remark~\ref{rem-cjoin}.
\end{proof}

\begin{observation}\label{obs:isotopy-3d}
Let $M$ be a $3$-manifold with a submanifold $N$ cut off by an essential $2$-sphere $\partial^0 N$. Since any two simple arcs on the sphere $\partial^0 N$ are isotopic in $\partial^0 N$, we can push any isotopy of paths that passes through $N$ off $N$. More formally: let $P$ and $Q$ be two disjoint subsets of $M$ not intersecting $N$; if two simple arcs $\alpha$ and $\beta$ from $P$ to $Q$ that do not intersect $N$ are isotopic in $M$ rel.\ $P$ and $Q$, then they are also isotopic in $M$ rel.\ $P$, $Q$, and $N$. This also applies to a finite collection of submanifolds $N_0,\ldots,N_k$ whose boundaries $\partial N_i$ are $2$-spheres.
\end{observation}

Let $\sigma$ be a $k$-simplex in $\htb(W_g^b,A)$ and let $X_0,\ldots,X_k$ be pairwise disjoint handle-tether-balls that represent the vertices of $\sigma$. Let $X'_0,\ldots,X'_k$ be another such collection. A priori, it is not clear whether the two handle-tether-ball systems are isotopic or not; equivalently, the geometric realization of $\htb(W_g^b,A)$ might not be the same as the geometric realization of the poset of isotopy classes of handle-tether-ball systems with deletion of components as the face relation. The issue is that the isotopy from $X_i$ to $X'_i$ witnessing that they represent the same vertex may pass through some of the other $X_j$. 

As it is the tether part which is causing the trouble, we can apply Observation~\ref{obs:isotopy-3d} ($P$ and $Q$ being a handle and a ball connected by the tether we need to move). Not only does it follow that simplices in $\htb(W_g^b,A)$ can be viewed as isotopy classes of pairwise disjoint handle-tether-ball systems; it also implies that the link of a simplex $\sigma=\{X_0,\ldots,X_k\}$ represented by pairwise disjoint $X_i$ can be recognized as the handle-tether-ball complex of the manifold after excising regular neighborhoods of the $X_i$. Summarizing: 
\begin{lemma}\label{lem:htb-link-3d}
Let $\sigma=\{X_0,\ldots,X_k\}$ be a $k$-simplex in $\htb(W_g^b,A)$. The link of $\sigma$ in $\htb(W_g^b,A)$ is isomorphic to $\htb(W_{g-k-1}^{b-(k+1)(d-1)},A^*)$, where $A^*$ is obtained from $A$ by removing the $(k+1)d$ boundary spheres used in $\sigma$.
\end{lemma}
\begin{proof}
The manifold $W_{g-k-1}^{b-(k+1)(d-1)}$ is obtained from $W_g^b$ by excising regular neighborhoods of the handle-tether-balls $X_i$ each leaving a single new boundary sphere that previously was essential. The set $A^*$ contains those elements from $A$ that have not been lost due to this excision.
\end{proof}

\subsection*{The extended handle-tether-ball complex}
As in the previous section, the extended handle-tether-ball complex $\htbs(W_g^b,A)$ has the same vertex set as $\htb(W_g^b,A)$, but now $k+1$ vertices form a $k$-simplex if there are handle-tether-balls $X_0,\ldots,X_k$ representing the vertices whose handle-tether parts are pairwise disjoint and whose ball parts are pairwise disjoint-or-equal. Again, we have two projections, namely
\[
  \pi_h : \htbs(W_g^b,A) \rightarrow \mH(W_g^b)
\]
to the handle complex and
\[
  \pi_b : \htbs(W_g^b,A) \rightarrow \dball(W_g^b,A)
\]
to the ball complex, defined by forgetting the tether-ball parts or the handle-tether parts, respectively. Once more, we shall derive the connectivity of $\htbs(W_g^b,A)$ by understanding the connectivity of fibers for $\pi_b$ which,  in turn, are understood by restricting $\pi_h$. First, we describe the fibers of closed simplices in $\dball(W_g^b,A)$ under $\pi_b$. Before doing so, we need some definitions.  First, as in the previous section, given a vertex $B \in \dball(W_g^b,A)$, we write $\partial^0 B$ for the boundary sphere  that is essential  in $W_g^b$. Next, for a subset $Q$ of boundary spheres of $W_g^b$, define $\tha(W_g^b ,Q)$ to be the complex whose vertices are isotopy classes of handles tethered to elements of $Q$ and where $k+1$ such vertices form a $k$-simplex if they can be simultaneously realized in a way that the handle-tether parts are pairwise disjoint.

\begin{lemma}\label{lem:fiber-is-tha-3d}
Let $\sigma = \{B_0,\ldots,B_k\}$ be a $k$-simplex in $\dball(W_g^b,A)$.  Let $W_g^{b-(k+1)(d-1)}$ be the manifold obtained from $W_g^b$ by cutting out the $B_i$. Then the fiber $\pi_b^{-1}(\sigma)$ is isomorphic to  $\tha(W_g^c,Q)$, where $Q := \{ \partial^0 B_0, \ldots, \partial^0 B_k\}$.
\end{lemma}
\begin{proof}
Gluing the balls $B_i$ back in defines an obvious map
\[
  \gamma : \tha(W_g^c,Q) \longrightarrow \pi_b^{-1}(\sigma),
\]
which we claim is an isomorphism. For vertices this amounts to the observation that a handle-tether-ball $X$ in $W_g^b$ connecting to the ball $B_i$ can be isotoped off $B_0\union\cdots\union B_{i-1}\union B_{i+1}\union\cdots\union B_k$. One can do this by shrinking the handle and moving the tether. The case of a simplex does not pose additional difficulties: isotoping a single arc off $B_j$ is not different from moving a (possibly braiding) collection of arcs off $B_j$.
\end{proof}

\begin{lemma}\label{lem:tha-is-join-cx-3d}
The map
\(
  \pi_h : \tha(W_g^c,Q)\rightarrow \mH(W_g^c)
\)
forgetting the tethers is a join complex.
\end{lemma}
\begin{proof}
Given a simplex $\tau=\{X_0,\ldots,X_p\}$ in $\tha(W_g^c,Q)$, the handles $\pi_h(X_i)$ are pairwise distinct as they can be realized pairwise disjointly. Hence $\rho$ is injective on simplices.

Let $\sigma=\{H_0,\ldots,H_k\}$ be a $k$-simplex in $\mH(W_g^c)$ realized by pairwise disjoint handles $H_i$. We can tether the handles $H_i$ to boundary spheres from $Q$ by pairwise disjoint tethers that do not intersect the handles (away from the point of attachment). Thus, $\pi_h$ is surjective.

Before checking the final condition, it is helpful to clarify why $\pi_h$ is not a complete join. To this end, let $X_i$ represent vertices in $\tha(W_g^c,Q)$ with $H_i=\pi_h(X_i)$. In a complete join, these vertices $X_i$ would span a simplex in $\tha(W_g^c,Q)$, i.e., we would be able to choose pairwise disjoint representatives $X_i$. However, if the tether part of $X_0$ runs through the handle part of $X_1$ in an essential way (i.e., so that it cannot be pushed off the handle), we would not be able to not realize $X_0$ and $X_1$ disjointly.

The definition of a join complex~ (see Section \ref{subsection:join}) deals with this kind of obstruction: we only need to show that vertices $X_i$ above the $H_i$ span a simplex, provided each of the $X_i$ can be extended to a $k$-simplex $\tau_i$ mapping to $\sigma$ under $\pi_h$. The extending simplex $\tau_i$ shows that the tether part of $X_i$ can be pushed off all handles $H_j$. Now, it only remains to resolve intersections among tethers. As the ambient manifold has dimension~3, this can easily be arranged via a small perturbation within each isotopy class.
\end{proof}
\begin{lemma}\label{lem:link-onto-3d}
For any simplex $\sigma$ in $\tha(W_g^c,Q)$, we have $\pi_h(\Link(\sigma))=\Link(\pi_h(\sigma))$.
\end{lemma}
\begin{proof}
The inclusion $\pi_h(\Link(\sigma))\subseteq\Link(\pi_h(\sigma))$ holds because $\pi_h$ is injective on simplices. The reversed inclusion follows from the same reasoning as the surjectivity of $\pi_h$ in the previous lemma.
\end{proof}

\begin{corollary}\label{cor:conn-htbs-3d}
If $2m \leq \frac{g-4}{2}$ then $\pi_b^{-1}(\sigma)$ is $m$-connected for each simplex $\sigma$ in $\dball(W_g^b,A)$. If, in addition, $m \leq \frac{|A|-d}{d+1}-1$, then $\htbs(W_g^b,A)$ is also $m$-connected.
\end{corollary}
\begin{proof}
For any closed simplex $\sigma$ in $\dball(W_g^b,A)$, the fiber $\pi_b^{-1}(\sigma)$ can be described as complex $\tha(W_g^c,Q)$ by Lemma~\ref{lem:fiber-is-tha-3d}. By Lemma~\ref{lem:tha-is-join-cx-3d}, the fiber $\pi_b^{-1}(\sigma)\cong\tha(W_g^c,Q)$ is a join complex over the base $\mH(W_g^c)$, which is $wCM$ of dimension $2m+1$ by Corollary~\ref{cor:RW-3d}. By Lemma~\ref{lem:link-onto-3d}, the link condition of Theorem~\ref{thm-conn-join} is satisfied, and we conclude that the fiber $\pi_b^{-1}(\sigma)$ is $m$-connected.

By Corollary~\ref{cor:ball-cx-conn-3d}, the ball complex $\dball(W_g^b,A)$ is $m$-connected. We have just seen that the fiber above each closed simplex is also $m$-connected. It follows from Quillen's Fiber Theorem~\ref{app:quillen} that the total space $\htbs(W_g^b,A)$ is $m$-connected.
\end{proof}

\subsection*{The connectivity of the handle-tether-ball-complex.}
Again, we use the bad simplex argument in the setting of Example~\ref{app:colored-vertices}. We  regard the vertices in $\dball(W_g^b,A)$ as colors and the projection $\pi_b : \htbs(W_g^b,A) \to \dball(W_g^b,A)$ as a coloring. Then $\htb(W_g^b,A)$ is the subcomplex of good simplices. As in the previous section, induction works as we know good links of bad simplices are handle-tether-ball complexes of simpler manifolds.

\begin{lemma}\label{lem:conn-dul-3d}
Let $\sigma$ be a bad $k$-simplex of $\htbs(W_g^b,A)$. Then the good link $G_\sigma$ is isomorphic to a complex $\htb(W_{g-k-1}^c,A^*)$, for some (explicit) $c\leq b$, where $A^*$ is obtained from $A$ by removing those up to $kd$ boundary spheres used in $\pi_b(\sigma)$.
\end{lemma}
\begin{proof}
In $\htbs(W_g^b,A)$, we allow several handles to be tethered to the same ball. Nonetheless, the regular neighborhood of such a ``spider'' is a $2$-sphere.  Hence, Observation~\ref{obs:isotopy-3d} still applies because the components we add to $\sigma$ while staying in the good link contain just one tether each.
\end{proof}

\begin{proposition}\label{prop:main-conn-3d}
Assume $m\leq \frac{|A|-d}{d+1}-1$ and $2m \leq \frac{g-4}{2}$. Then the handle-tether-ball complex $\htb(W_g^b,A)$ is $m$-connected.
\end{proposition}
\begin{proof}
We recreate the induction argument given for Proposition~\ref{prop:main-conn-high-dim}. Everything carries over, only the bounds differ. Again, we consider a bad $k$-simplex and show that its good link is $(m-k)$-connected. In view of the preceding lemma, we need to verify
\(
  2(m-k) \leq \frac{(g-k-1)-4}{2}
\)
and
\(
  m-k \leq \frac{|A^*|-d}{d+1}-1
\). 
Using that $k\geq 1$, as vertices are not bad, we calculate:
\[
  2(m-k) \leq \frac{g-4}{2} - 2k = \frac{(g-k-1)-4}{2}+ \frac{-3k+1}{2} \leq \frac{(g-k-1)-4}{2}
\]
and:
\[
  m-k \leq \frac{|A|-d}{d+1}-1 -k 
  =\frac{|A|-kd-d}{d+1}-1 -\frac{k}{d+1}
  \leq \frac{|A^*|-d}{d+1}-1
\]
The claim follows.
\end{proof}

After Lemma \ref{lem:conn-dul-3d} and Proposition \ref{prop:main-conn-3d}, we have:
\begin{corollary}\label{cor:main-wcm-3d}
Assume $m\leq \frac{|A|-d}{d+1}$ and $m \leq \frac{g}{4}$.Then the handle-tether-ball complex $\htb(W_g^b,A)$ is $wCM$ of dimension $m$. Consequently,  $\mP_d(W_g^b,A)$ is also $wCM$ of dimension $m$.
\end{corollary}

\subsection*{The flag property} 
Finally, we prove that the piece complex is flag: 

\begin{theorem}\label{thm:flag-3d}
$\mP_d(W_g^b, A)$ is a flag complex.
\end{theorem}

\begin{proof}
We begin with the following consequence of the Prime Decomposition Theorem for 3-manifolds. Suppose $Z$ and $Z'$ are distinct vertices of $\mP_d(W_g^b,A)$, and let $S$ (resp. $S'$) be the unique boundary sphere of $Z$ (resp. $Z'$) that is essential in $W_g^b$. If $S$ and $S'$ are disjoint, then $Z$ and $Z'$ are also disjoint.

Now let $Z_0, \ldots, Z_k$ be vertices that define the 1-skeleton of a $k$-simplex in $\mP_d(W_g^b,A)$, and $S_i$ the unique essential boundary sphere of $Z_i$. 

Since the sphere complex of $W_g^b$ is flag (see \cite[Lemma 3]{AS11} which, although is stated for connected sums of $\bS^2\times \bS^1$, works in all generality), there is an isotopy $\phi_t: W_g^b\to W_g^b$
such that the spheres $S_i'= \phi_1(S_i)$ are all simultaneously disjoint. By the observation in the first paragraph, the submanifolds $Z_i'= \phi_1(Z_i)$ are simultaneously disjoint also, and hence define a $k$-simplex in  $\mP_d(W_g^b,A)$, as desired. 
\end{proof}

\begin{proof}[Proof of Theorem~\ref{thm:main-3d}]
Combine Theorem~\ref{thm:flag-3d} and Corollary~\ref{cor:main-wcm-3d}.
\end{proof}

\begin{remark}
The statement of Theorem~\ref{thm:main-3d} remains valid if $Y$ is the $3$-sphere $\bS^3$ because in that case, we have $\mP_d(W_g^b,A) \cong \dball(W_g^b,A)$, and Corollary~\ref{cor:ball-cx-conn-3d} applies.
\end{remark}

\bigskip
\section{The piece complexes in  dimension two}
\label{sec:2dim}
Finally, in this section we will prove a two-dimensional incarnation of Theorem \ref{thm:genpiececomplex}. In this case, $O$ is a compact connected orientable surface and~$Y$ is either a sphere or a torus. Here we will use slightly different notation than in the previous two sections, and will write $W_g^b$ for the compact surface of genus $g$ with $b$ boundary components. Observe, however, that $O\# Y \# \cdots ^{(k)} \# Y$ is diffeomorphic to $W_g$, with $g= \text{genus}(O) + k\cdot \text{genus}(Y)$. We will prove:



\begin{theorem}\label{thm:main-2d}
Let $W_g^b$ be the compact, connected surface of genus $g$ with $b$ boundary circles, and let $A$ a (not necessarily proper) subset of these boundary circles. Then  $\mP_d(W_g^b, A)$ is flag. Moreover, $\mP_d(W_g^b,A)$ is $wCM$ of dimension $m$ provided that $m \leq \frac{g-1}{2}$ and $m \leq \frac{|A|+2d}{2d-1}$.
\end{theorem}

\begin{remark}
As we will see below, the genus bound in the statement above is not needed when $Y$ is a sphere. 
\end{remark}

Yet again, we proceed along the same lines as in Section \ref{sec:highdim}. Once more, the topological analysis becomes a bit more involved: in two dimensions, an arc can be separating, and cutting a surface open along an arc affects the fundamental group.

\subsection*{The ball complex}
In the surface case, we will not analyze the connectivity properties of the ball complex $\dball(W_g^b,A)$ by means of its projection to the $d$-hypergraph complex. Instead we appeal to work by Skipper--Wu \cite[Theorem 3.12]{SW21a}, who have calculated the connectivity of the {\em $d$-holed disk complex} of a punctured surface. Capping each boundary component with a once-punctured disk yields an isomorphism between the ball complex (of the surface with boundary) and the $d$-holed disk complex (of the punctured surface). 
Thus, we conclude:
\begin{lemma}[\cite{SW21a}]
 $\dball(W_g^b,A)$ is $m$-connected for $m \leq \frac{|A|+1}{2d-1} -2$.
\end{lemma}
Since links in the ball complex are again ball complexes, we can also deduce the weak Cohen--Macaulay property.
\begin{corollary}\label{cor:ball-wcm-2d}
 $\dball(W_g^b,A)$ is $wCM$ of dimension $\lfloor \frac{|A|+1}{2d-1}-1 \rfloor$.
\end{corollary}

\begin{proof}[Proof of Theorem~\ref{thm:main-2d} for $Y\cong \bS^2$]
If $Y$ is a $2$-sphere, the geometric piece complex is isomorphic to the ball complex, which is a flag complex. Thus, Corollary~\ref{cor:ball-wcm-2d} establishes this case of Theorem~\ref{thm:main-2d}. 
\end{proof}

From now on, we assume that $Y\cong \bS^1\times \bS^1$. 

\subsection*{The handle complex}
As in the previous sections, the handle complex $\mH(W_g^b)$ has as vertices isotopy classes of handles (i.e. embedded copies of $\bS^1 \times \bS^1 \setminus \mathbb B^2$) and the adjacency relation is given by disjointness. To analyze the handle complex in the surface case, we will make use of the so-called \notion{complex of chains} introduced by Hatcher--Vogtmann \cite{HV17}.

A  \notion{chain} is an unordered pair of simple closed curves intersecting transversely at a single point; in particular, both curves are non-separating. The complex of chains $\chain{W_g^b}$ has simplices corresponding to isotopy classes of systems of pairwise disjoint chains in $W_ g^b$; observe that every such system is {\em coconnected}, i.e. its complement has exactly one connected component. The following was proved by Hatcher--Vogtmann \cite[Proposition 5.6]{HV17}: 
\begin{theorem}[\cite{HV17}]
The complex of chains $\chain{W_g^b}$ is $\lfloor\frac{g-3}{2}\rfloor$-connected.
\label{thm:HVchain}
\end{theorem}
Observe that every  chain has a tubular neighborhood  diffeomorphic to $\bS^1 \times \bS^1 \setminus \mathbb B^2$. As a consequence, the link of a $k$-simplex $\sigma$ in $\chain{W_g^b}$ is isomorphic to $\chain{W_{g-k-1}^{b+k+1}}$, and so we deduce the following:
\begin{corollary}
The complex of chains $\chain{W_g^b}$ is $wCM$ of dimension $\lfloor\frac{g-1}{2}\rfloor$.
\end{corollary}

 Similarly, given a $k$-simplex $\{ch_0,\cdots,ch_k\}$ in $\chain{W_g^b}$, we get a $k$-simplex $H_0, \ldots, H_k$ 
 by considering (sufficiently small) neighborhoods $H_i$. In other words, there is a simplicial map 
\[
  \alpha_{ch}:\chain{W_g^b}\to \mH(W_g^b).
\] 
We observe: 
\begin{lemma}\label{lem:onetorus_conn}
The map $\alpha_{ch}$ is a complete join. Therefore, $\mH(W_g^b)$ is $wCM$ of dimension $\lfloor\frac{g-1}{2}\rfloor$.
\end{lemma}
\begin{proof}
First, any two disjoint non-isotopic chains have non-isotopic neighborhoods, and so $\alpha_{ch}$ injective on individual simplices.
Next, given any simplex $\sigma=\{H_0,\ldots,H_k\}$ in $\mH(W_g^b)$, we can choose a  chain $ch_i$ in each $H_i$. Since these chains are pairwise disjoint, they form a simplex $\{ch_0,\ldots,ch_k\}$ in $\chain{W_g^b}$ that maps to $\sigma$ under $\alpha_{ch}$, and thus $\alpha_{ch}$ is surjective. 
Moreover, any element in $\alpha_{ch}^{-1}(H_i)$ can be realized as a chain in $H_i$. Hence $\alpha_{ch}^{-1}(\sigma)=\bigjoin_{i=1}^k\alpha_{ch}^{-1}(H_i)$, and thus $\alpha_{ch}$ is a complete join. The connectivity of $\mH(W_g^b)$ now follows from Remark \ref{rem-cjoin} and Theorem \ref{thm:HVchain}.
\end{proof}

\subsection*{The handle-tether-ball complex} 
The complex $\htb(W_g^b,A)$ is defined as in the previous sections. The same argument as in the proof of Lemma~\ref{lem-conn-totd-3holt} yields: 
\begin{lemma}\label{lem-conn-totd-3holt-2D}
The map $\pi:\htb(W_g^b,A)\rightarrow\mP_d(W_g^b,A)$ is a complete join. In particular, if $\htb(W_g^b,A)$ is $m$-connected, so is $\mP_d(W_g^b,A)$.
\end{lemma}

Also,  the proof of Lemma~\ref{lem:htb-link-3d} also applies in two dimensions:
\begin{lemma}\label{lem:htb-link-2d}
Let $\sigma$ be a $k$-simplex in $\htb(W_g^b,A)$. The link of $\sigma$ in $\htb(W_g^b,A)$ is isomorphic to a complex $\htb(W_{g-k-1}^{b-(k+1)(d-1)},A^*)$, where $A^*$ is obtained from $A$ by removing those $(k+1)d$ boundary circles used in $\sigma$.
\end{lemma}


\subsection*{The extended handle-tether-ball complex} 
As was the case in previous sections, allowing multiple handles to be tethered to the same ball leads to the complex $\htbs(W_g^b,A)$, which contains $\htb(W_g^b,A)$ as a subcomplex. Both complexes have the same vertex set, although $\htbs(W_g^b,A)$ may contain additional simplices.  The goal of this section is to prove the $m$-connectivity of $\htbs(W_g^b,A)$ if $m$ is bounded as an (explicit) linear function of $g$ and $\vert A\vert$. As in the higher dimensional cases we will use the projection 
\[
\pi_b:\htbs(W_g^b,A)\rightarrow\dball(W_g^b,A). 
\]

Recall that we know the connectivity of the base space $\dball(W_g^b,A)$ (see Corollary \ref{cor:ball-wcm-2d}). Thus it remains to determine the connectivity of the fibers $\pi_b^{-1}(\sigma)$ of simplices $\sigma$; for this we use tethered handle complexes $\tha(W_g^b,Q)$. 

Even if the outline of the proof is similar, the technical differences between the surface case and the cases of higher-dimensional manifolds are most pronounced in this step. Since there is no room to isotope two arbitrary arcs off each other, we cannot prove an analogue of Lemma \ref{lem:fiber-is-tha-3d}. Instead, we will consider an inductive argument, which requires to cut $W_g^b$ along tethered handles. This procedure  may cut $W_g^b$ into components, as well as cut boundary loops of $W_g^b$ into intervals. Therefore we start with a more general description of tethered handle complexes: we will drop the subscripts and use $W$ to denote a connected surface and $Q$ to denote a collection of pairwise disjoint circles and open intervals (with pairwise disjoint closures) contained in the boundary of $W$. A $k$-simplex of $\tha(W,Q)$ is given by $k+1$ isotopy classes of tethered-handles in $W$ with disjoint representatives. 

\begin{remark}
We emphasize that even if we allow the surface $W$ to be a subsurface of $W_g^b$ and the collection $Q$ to contain intervals, we are still interested in the following situation: let $\sigma=\{B_0,\ldots,B_k\}$ be a $k$-simplex in $\dball(W_g^b,A)$. We excise the balls $B_i$ from $W_g^b$ to obtain a new surface $W=W_g^c$, (where $c=b-(k+1)(d-1)$, although the exact value is of no importance) and let $Q$ be the set of those boundary circles in $W_g^c$ that come from the essential boundary spheres $\partial^0 B_i$. 
\end{remark}

\begin{figure}
\includegraphics[scale=0.75]{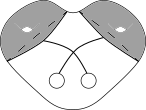}
\caption{Tethered handles whose tethers necessarily cross}
\label{fig:non_complete_join}
\end{figure}

Another departure from the higher-dimensional case is that the projection $\pi_h:\tha(W,Q)\rightarrow\mH(W)$ is not a join complex, in stark contrast to Lemma \ref{lem:tha-is-join-cx-3d}; see Figure \ref{fig:non_complete_join}. Instead, we use a bad simplex argument to understand the connectivity properties of $\tha(W,Q)$ as the subcomplex of $\thas(W,Q)$, which is a complex with the same vertex set as $\tha(W,Q)$, but whose simplices are tethered-handle systems where we allow multiple tethers at each handle and multiple tethers at each boundary loop. In other words: we only insist that tethers are pairwise disjoint; however, their {\em heads} and {\em tails} (handles or boundary loops, respectively) are either pairwise disjoint or equal.

\begin{lemma}\label{lem:conn-thas-2d}
The complex $\thas(W,Q)$ is $\lfloor\frac{g(W)-3}{2} \rfloor$-connected. In particular, $\thas(W_g^b,Q)$ is $\lfloor \frac{g-3}{2} \rfloor$-connected. 
\end{lemma}
\begin{proof}
We consider the projection $\pi_h : \thas(W,Q)\to\mH(W)$ forgetting tethers and balls. Let $\sigma=\{ H_0,\ldots,H_k\}$ be a $k$-simplex in $\mH(W)$. We claim that the preimage $\pi_h^{-1}(\sigma^\circ)$ above the barycenter $\sigma^\circ$ of $\sigma$ is contractible. To see this, we identify $\pi_h^{-1}(\sigma^\circ)$ with the complex $A^\circ_k(S,P,P,Q)$ from Proposition~\ref{app:two-sided-arc-complex}. Namely, we obtain the surface $S$ from $W$ by first collapsing the handles $T_i$ to punctures in $P$, and then collapsing the boundary circles and intervals in $Q$ to punctures.
Now, the complex $\bacx(S,P,P,Q)$ consists of those collections of tethers that connect $P$ to $Q$ and make use of all punctures in $P$. It is this last condition that identifies $\bacx(S,P,P,Q)$ as the barycentric fiber $\pi_h^{-1}(\sigma^\circ)$.

Now, Proposition \ref{app:two-sided-arc-complex} tells us that the barycentric fibers are all contractible, and Lemma~\ref{lem:fiber_conn} implies that $\thas(W,Q)$ inherits the connectivity properties of the base $\mH(W)$. The claim then follows from Lemma~\ref{lem:onetorus_conn}.
\end{proof}

\begin{lemma}\label{lem:conn_tha-2d}
 $\tha(W,Q)$ is $\left\lfloor\frac{g(W)-3}{2}\right\rfloor$-connected. 
 In particular, $\tha(W_g^b,Q)$ is $\left\lfloor\frac{g-3}{2}\right\rfloor$-connected. 
\end{lemma}
\begin{proof}
We induct on $g$, using a bad simplex argument on the inclusion $\tha(W, Q) \subseteq\thas(W,Q)$. If $g\geq1$ and $|Q|\geq1$ the complex $\tha(W,Q)$ is non-empty. By induction, we assume that the result is true for surfaces with $g'<g$ and $b'\ge 1$. 
The bad simplex argument follows Example~\ref{app:colored-vertices}. We use the projection to the handle complex as our coloring map; i.e., we call a simplex $\sigma$ in $\thas(W,Q)$ {\em bad} if each handle used in $\sigma$ hosts at least two tethers. In particular, vertices cannot be bad. A simplex $\sigma$ in $\thas(W,Q)$ is called good if at most one tether attaches to each handle. The good link $G_\sigma$ of a bad simplex $\sigma$ consists of those systems of tethered handles whose handles are pairwise distinct and not used in $\sigma$.

We can give an explicit description of the good links. Let $\sigma$ be a bad simplex and $W_1,\ldots,W_l$ be the connected components of $W\setminus\sigma$, obtained from $W$ by cutting out a regular neighborhood of $\sigma$. Further, $(W\setminus\sigma)\cap Q$ splits into a non-empty collection of circles and intervals $Q_i=W_i\cap Q$, see Figure \ref{fig:cutting_circles}. 
\begin{figure}
\includegraphics[scale=0.75]{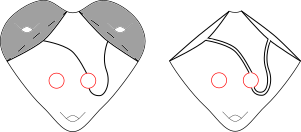}
\caption{Removing the two tethered handles on the left yields the surface on the right. Note how a  boundary circle from $Q$ splits into two intervals.}
\label{fig:cutting_circles}
\end{figure}
%
The good link $G_\sigma$ is isomorphic to the join $\bigjoin \tha(W_i,Q_i)$. If $m\leq k$ is the number of handles in $\sigma$, we have $g=m+\sum g(W_i)\leq k+\sum g(W_i)$. In particular, the components $W_i$ have genus strictly less than $g$ so we can apply the induction hypothesis to deduce that $G_\sigma$ is $\left(\sum_{i=1}^l \left\lfloor\frac{g(W_i)-3}{2}\right\rfloor+2l-2\right)$-connected. 
Recall that vertices cannot be bad, and hence $k\geq1$. A direct calculation now yields:
\begin{align*} 
\sum_{i=1}^l \left\lfloor\frac{g(W_i)-3}{2}\right\rfloor+2l-2 &\geq \left\lfloor  \sum_{i=1}^l \frac{g(W_i)-3}{2} -\frac{l}{2}\right\rfloor+2l-2 \\ 
&=  \left\lfloor \frac{\sum_{i=1}^l g(W_i)-3}{2} -\frac{3}{2}(l-1)-\frac{l}{2} +2l-2\right\rfloor \\
&\geq  \left\lfloor \frac{g-k-3}{2} -\frac{1}{2}\right\rfloor  \\
&=  \left\lfloor \frac{g-3}{2} -\frac{k+1}{2}\right\rfloor  (\text{since } k\geq1)\\
&\geq \left\lfloor \frac{ g-3}{2} -k\right\rfloor  
\end{align*}
Hence $G_\sigma$ is at least $\left\lfloor\frac{g-3}{2}-k\right\rfloor$-connected, at which point Theorem \ref{app:gated-morse-theory} implies that $\tha(W,Q)$ is $\left\lfloor\frac{g-3}{2} 
\right\rfloor$-connected. 
\end{proof}

\begin{lemma}\label{lem:conn-htbs-2d}
Assume that $m \leq \frac{g-3}{2}$ and $m \leq \frac{|A|+1}{2d-1}-2$. Then the extended handle-tether-ball complex $\htbs(W_g^b,A)$ is $m$-connected.
\end{lemma}
\begin{proof}
We consider the projection $\pi_b : \htbs(W_g^b,A)\to\dball(W_g^b,A)$ forgetting handles and tethers. As previously, we are interested in the connectivity of fibers above closed simplices. So let $\tau=\{B_0,\ldots,B_k\}$ be a $k$-simplex in $\dball(W_g^b,A)$. The final difference with the situation in  the previous two sections is that the fiber $\pi_b^{-1}(\tau)$ is not isomorphic to the tethered-handle complex $\tha(W_g^c,\{\partial^0 B_0,\ldots,\partial^0 B_k\})$. In other words, the statement of Lemma~\ref{lem:fiber-is-tha-3d} fails in the surface case. The reason is that a tether connecting a handle to, say, $B_0$ may run through $B_1$; and while we were able to push the tether off the $d$-holed ball in three dimensions, we may not be able to push it of the $d$-holed disk $B_1$ in two dimensions. The holes are obstacles to moving arcs.

Nonetheless, we can describe the fiber $\pi_b^{-1}(\tau)$ in terms of tethered-handle complexes. For an inclusion chain $\rho \subseteq \sigma\subseteq \tau$, denote by $S_\sigma$ the surface obtained from $W_g^b$ by excising the balls $B$ in $\sigma$, and by $Q_\rho$ the collection of $\partial^0 B$ for $B \in \rho$. Note that $Q_\rho$ can be interpreted in as a collection of boundary circles in $S_\sigma$. Also, observe that all $S_\sigma$ has genus $g$. The tethered-handle complex $\tha(S_\sigma,Q_\rho)$ consists of those tethered handles that are tethered to a balls in $\rho$ and can be pushed off all the balls $\sigma$.

With this notation, the fiber $\pi_b^{-1}(\tau)$ can be described as the following union:
\begin{equation}\label{eq:cover-2d}
  \pi_b^{-1}(\tau) =
  \bigcup_{\emptyset \neq \sigma \subseteq \tau} \tha(S_\sigma,Q_\sigma)
  .
\end{equation}

By Lemma~\ref{lem:conn_tha-2d}, all terms in the union are $m$-connected complexes. We want to apply a nerve-cover argument (see Proposition~\ref{app:B14}). We therefore have to consider the intersection of the terms, as well. Let $\sigma_0,\ldots,\sigma_j$ be a collection of faces of $\tau$. We observe:
\begin{equation}\label{eq:intersections-2d}
  \bigcap_{i=0}^{j} \tha(S_{\sigma_i},Q_{\sigma_i})
  =
  \tha(S_{\sigma_0 \union \cdots \union \sigma_j}, Q_{\sigma_0 \intersect \cdots \intersect \sigma_j})
  .
\end{equation}
This just states that a tethered handle belongs to the intersection if and only if it is tethered to a ball that lies in all the $\sigma_i$ and if it can be pushed off all the balls that lie in any of the $\sigma_i$.

From the description~\eqref{eq:intersections-2d}, we see that intersections are again $m$-connected unless they are empty. The latter happens if and only if the intersection $\sigma_0\intersect\cdots\intersect\sigma_j$ is empty. Therefore, the nerve of the cover~\eqref{eq:cover-2d} is isomorphic to the barycentric subdivision of a standard simplex and hence contractible. From Proposition~\ref{app:B14}, we deduce that the fiber $\pi_b^{-1}(\tau)$ is $m$-connected for each simplex $\tau$ in $\dball(W_g^b,A)$. Finally, by Quillen's Fiber Theorem~\ref{app:quillen}, the total space $\htbs(W_g^b,A)$ is $m$-connected.
\end{proof}

\subsection*{Connectivity of the handle-tether-ball complex}
Again, we make use of the bad simplex argument as stated in Proposition~\ref{app:gated-morse-theory} in the flavor presented in Example~\ref{app:colored-vertices}. We consider the projection $\pi_b : \htbs(W_g^b,A)\to\dball(W_g^b,A)$ as a coloring of the vertices in $\htbs(W_g^b,A)$. The good simplices form the subcomplex $\htb(W_g^b,A)$ and as before, the good link of a bad $k$-simplex is of the form $\htb(W_{g-k-1}^c,A^*)$. As in the previous sections, we have: 
\begin{lemma}\label{lem:conn-dul-2d}
Let $\sigma$ be a bad $k$-simplex of $\htbs(W_g^b,A)$. The good link $G_\sigma$ is isomorphic to a complex $\htb(W_{g-k-1}^c,A^*)$, for some (explicit, albeit unimportant) $c\leq b$, and where $A^*$ is obtained from $A$ by removing those up to $kd$ boundary spheres used in $\pi_b(\tau)$.
\end{lemma}
We use this to deduce the connectivity of the handle-tether-ball complex:
\begin{proposition}\label{prop:main-conn-2d}
Assume $m\leq \frac{|A|+1}{2d-1}-2$ and $m \leq \frac{g-4}{2}$. Then the handle-tether-ball complex $\htb(W_g^b,A)$ is $m$-connected.
\end{proposition}
\begin{proof}
The induction will apply once we show that the good link $G_\sigma$ of a bad $k$-simplex is $(m-k)$-connected. In view of the preceding lemma, we have to verify
\(
  m-k \leq \frac{(g-k-1)-3}{2}
\)
and
\(
  m-k \leq \frac{|A^*|+1}{2d-1}-2
\).
As vertices are good, we find $k\geq 1$. For the first inequality, we deduce from $m\leq\frac{g-3}{2}$ and $2k\geq k+1$ that
\(
  m-k \leq\frac{g-2k-3}{2} \leq \frac{(g-k-1)-3}{2}
\) 
holds. For the second estimate, we find:
\[
  \frac{|A^*|+1}{2d-1} \geq \frac{|A|-kd+1}{2d-1} \geq \frac{|A|+1}{2d-1} - \frac{kd}{2d-1} \geq m - k
  .
\]
Thus, the induction hypothesis applies and the claim follows.
\end{proof}

Using Lemma~\ref{lem:htb-link-2d}, we also obtain
the following:
\begin{corollary}\label{cor:main-wcm-2d}
Assume $m\leq \frac{|A|+1}{2d-1}-1$ and $m \leq \frac{g-1}{2}$. Then the handle-tether-ball complex $\htb(W_g^b,A)$ is $wCM$ of dimension $m$. Consequently,  $\mP_d(W_g^b,A)$ is also $wCM$ of dimension $m$.
\end{corollary}

\subsection*{The flag property} Finally, we prove that the piece complex is flag. 

\begin{proposition}
$\mP_d(W_g^b,A)$ is a flag complex. 
\label{prop:flag2d}
\end{proposition}

\begin{proof}

Suppose $Z_0, \ldots, Z_k$ are vertices of $\mP_d(W_g^b,A)$ that define the 1-skeleton of a $k$-simplex in $\mP_d(W_g^b,A)$, and $S_i$ the unique boundary component of $Z_i$ that is essential in $W_g^b$. At this point, there is an isotopy $\phi_t: W_g^b\to W_g^b$
such that the curves $S_i'= \phi_1(S_i)$ are all simultaneously disjoint; this follows immediately from the fact that $W_g^b$ supports a hyperbolic metric. But then the subsurfaces $Z_i'= \phi_1(Z_i)$ are simultaneously disjoint also, and hence define a $k$-simplex in  $\mP_d(W_g^b,A)$, as desired. 
\end{proof}

\begin{proof}[Proof of Theorem~\ref{thm:main-2d} for $Y=\bS^1\times\bS^1$.]
We have just seen that $\mP_d(W_g^b,A)$ is a flag complex. The claim about being $wCM$ follows from Corollary~\ref{cor:main-wcm-3d}.
\end{proof}


\section{Homology of asymptotic mapping class groups}
\label{sec:homology}
We study the homology of asymptotic mapping class groups $\mB_{d,r}(O,Y)$ following the ideas in the proof of  \cite[Theorem 3.1]{FK09}. We will focus solely on the case $d=2$ and $r=1$; for this reason we will simply write $\mC(O,Y)$ and $\mB(O,Y)$. Our proof uses crucially the fact that $V=V_{2,1}$ is simple \cite{Hi74} and acyclic \cite{SW19}; we remark that neither fact is true for all Higman--Thompson groups. Recall from Definition \ref{def:compactsupport} that \[\Map_c(\mC(O,Y)) = \bigcup_{k=1}^\infty \Map(O_k),\] where $O_k$ are the compact $n$-manifolds used in the construction of $\mC(O,Y)$. Recall that $O_{k+1}$ is constructed from $O_k$ by gluing $2^k$ copies of $Y^2$ along the suited boundaries of $O_k$, where $Y^d$ denotes the result of removing $d+1$ open balls from $Y$. It will be interesting for our discussion to further divide this gluing map as follows: 

\begin{enumerate}[label= \textbf{Step} \arabic*.]
    \item gluing $Y^1$ to a suited boundary of $O_k$;
    \item gluing $S^3$ to the remaining boundary of $Y^1$, where $S^3$ denotes   the manifold obtained from the standard sphere $\bS^n$ by deleting three disjoint open balls.
\end{enumerate}

In the presence of the inclusion property, the inclusion map $O_k \hookrightarrow O_{k+1}$ induces an injective homomorphism $j:\Map(O_k) \to \Map(O_{k+1})$. In turn, the injective homomorphism $j$ induces a homomorphism $j_\ast:H_i(\Map(O_k)) \to H_i(\Map(O_{k+1}))$, which is well-studied in the literature about {\em homological stability} of mapping class groups; in short, this asserts that the map $j_\ast$ is in fact an isomorphism for certain values of $i$ (known as the {\em stable range}) in terms of $k$. 
The following result summarizes the results on homological stability of mapping class groups that are relevant to our purposes: 

\begin{theorem}\label{thm-hmg-stab}

\begin{enumerate}
    \item  Let $O$ be a compact surface, and $Y\cong \bS^1\times \bS^1$. Then $j_\ast$ is an isomorphism when $i\leq \lfloor\frac{2^{k+1}-4}{3}\rfloor$.
    \item Let $O$ be any compact $3$-manifold, and $Y\cong \bS^2\times \bS^1$. Then $j_\ast$ is an isomorphism when $i\leq \lfloor\frac{2^{k}-3}{2}\rfloor$.
\end{enumerate}
\end{theorem}

\begin{proof}
Part (i) follows from \cite{Har85}; here, note that the genus of $O_k$ is at least $2^k-1$.
Part (ii) follows from results in \cite[Theorem 6.1]{HW10}  and \cite[Corollary 6.2]{HW10}. We remark that  the map in \cite[Theorem 6.1]{HW10} is defined in terms {\em boundary connected sum}, although it applies just the same to the gluing construction detailed above.  Again, it is important the number of  $\bS^2\times \bS^1$  summands  is at least $2^k-1$ in the  calculation of  the stable range.  
%
\end{proof}


Since the compactly supported mapping class group is the direct limit of the mapping class groups of the $O_k$, we immediate obtain:
\begin{corollary}
Let $O$ and $Y$ be as in Theorem \ref{thm-hmg-stab}. Then $H_i(\Map_c(\mC(O,Y)))$ is the $i$-th stable homology group of $\Map(W_g)$.
\label{cor:homstab}
\end{corollary}

Note also the following:

\begin{theorem}\label{thm-fin-hmg-Mapc}
 $H_i(\Map_c(\mC(O,Y)),\mathbb{Z})$ is finitely generated for all $i$, if 
 \begin{enumerate}
    \item  $O$ is a compact surface and $Y$ is a torus.
    \item $O\cong \bS^3$, and $Y\cong \bS^2\times \bS^1$. 
\end{enumerate}
\end{theorem}
\begin{proof}
By Theorem \ref{thm-hmg-stab}, it suffices to show that  $H_i(O_k,\mathbb{Z})$ is finitely generated for all $i$. At this point, (i)  follows from Lemma \ref{lem-fin-mcg-surf}, and (ii)   from Corollary \ref{cor-fin-mcg-hbdy}.
\end{proof}

\begin{remark}\label{rem-mapc-nfg}
$H_i (\Map_c(\mC(O,Y)),\Z )$ need not be necessarily finitely generated, for example when $O \cong \mathbb{S}^{2n}, Y\cong \bS^n\times \bS^n$ ($n\ge 3$). In fact by Lemma \ref{prop-cap-ker-high} in Appendix \ref{sec:intersection}, the number of $\Z_2$ summands generated by Dehn twists around boundaries in $\Map(O_k)$ will grow exponentially. A similar phenomenon happens in dimension $2$ when one takes $O$ to be a disk and $Y \cong \mathbb{S}^2$. 

\end{remark}

We finally prove our desired result about the relation between stable homology and the homology of asymptotic mapping class groups, using the same argument as in \cite[Proposition 4.3]{FK09}:

\begin{proposition}\label{prop:hmg-asymc}
Suppose $H_i (\Map_c(\mC(O,Y)),\Z )$ is finitely generated for all $i$. Then $H_i (\mB(O,Y)) \cong H_i (\Map_c(\mC(O,Y),\Z )$ for all $i$.
\end{proposition}

\begin{proof}
After Proposition \ref{prop-rel-htg}, we have
$$ 1 \to \Map_c(\mC(O,Y)) \rightarrow \mB(O,Y)  \rightarrow V \to 1.$$
Now apply the Lyndon-Hochschild-Serre spectral sequence, getting
$$ E^2_{p,q} = H_p(V, H_q( \Map_c(\mC(O,Y),\Z)) \implies H_i(\mB(O,Y), \Z)$$
For convenience, write $A_q = H_q( \Map_c(\mC(O,Y)),\Z)$.  By assumption,  we have $A_q$ is a finitely generated abelian group for all $q$. In particular $\mathrm{Aut(A_q)}$ is residually finite \cite{Ba63}. Since $V$ is simple, we have any homomorphism $V \to \mathrm{Aut}(A_q)$ must be trivial. Thus the action of $V$ on  $A_q$  is trivial. Furthermore, since V is acyclic \cite{SW19}, we have $H_p(V, A_q) \cong 0$ if $p>0$ and 
$H_0(V, A_q) \cong A_q$. Therefore the spectral sequence collapses, and we have $H_i (\mB(O,Y),\Z) \cong H_i (\Map_c(\mC(O,Y),\Z )$ for all $i$, as claimed.
\end{proof}
\begin{remark}
One might try to use the same argument to calculate the homology of the braided Thompson's group $V_{br}$ using  the short exact sequence $1\to PB_\infty\to V_{br} \to V\to 1$. However, $H_1(PB_\infty,\Z)$ is already not finitely generated which one can prove using winding numbers.
\end{remark}

After all this, we arrive at the proof of Theorem \ref{thm:homology}: 

\begin{proof}[Proof of Theorem \ref{thm:homology}] 
The result follows at once from the combination of Corollary \ref{cor:homstab} and Proposition \ref{prop:hmg-asymc}. 
\end{proof}

Finally, we record the following, which may be of independent interest:

\begin{corollary}
$\mB(\bS^3,\bS^2 \times \bS^1)$ is rationally acyclic. 
\end{corollary}
\begin{proof}
We only need to show that $\Map_c(\mC(O,Y))$ is rationally acyclic. In turn, by Theorem \ref{thm-hmg-stab}, it suffices to prove that the stable homology of $\Map(O_k)$ is rationally acyclic.  Using \cite[Theorem 1.1]{HW05} and the Lyndon-Hochschild-Serre spectral sequence, it is enough to show the groups $A_{n,0}^s$ of \cite{HW05}  are stably rationally acyclic; this is proved in \cite[Theorem A.2]{Sou18} based on the fact that the automorphism group of free group is rationally acyclic  \cite{Ga11}.
\end{proof}

\appendix
\section{Connectivity tools}
\label{sec:connectivitytools}
The purpose of this section is to give some details on some  connectivity tools that are essential for calculating the connectivity of our spaces. A good reference is \cite[Section 2]{HV17}, although not all the tools we use can be found there.

\subsection{Combinatorial Morse theory}
\label{sec:discreteMorse} 
As will have become apparent by now, we make extensive use of combinatorial Morse theory. For simplicial complexes, that is mostly a modern name for techniques to establish connectivity that have been employed for a long time. Let $X$ be a simplicial complex with a subcomplex $L$. Consider $X$ with $L$ coned off, i.e., consider the space
\(
  Y := X \union_L C(L)
\)
obtained by gluing the cone $C(L)$ to the space $X$ along the common subspace $L$. Using the long exact sequence of homotopy groups for the pair $(Y,X)$, one deduces the following:
\begin{observation}
If $L$ is $m$-connected, then the inclusion $X\hookrightarrow Y=X\union_L C(L)$ induces isomorphisms in homotopy groups $\pi_d$ for $d\leq m$ and an epimorphism in $\pi_{m+1}$.
\end{observation}
One may also cone off several subcomplexes $L_i$ at the same time:
\begin{observation}\label{obs:app-cone-off}
Let $X$ be a simplicial complex,  $\{L_i\}_{i\in I}$  a family of $m$-connected subcomplexes, and  $Y$  the space obtained by gluing each cone $C(L_i)$ to $X$ along $L_i$. Then the inclusion $X \hookrightarrow Y$ induces isomorphisms in homotopy groups $\pi_d$ for $d\leq m$ and an epimorphism in $\pi_{m+1}$.
\end{observation}
Now, consider a map $h : X^{(0)} \to T$ defined on the vertex set of a simplicial complex $X$ with values in a totally ordered set $(T,<)$. We say that $h$ is a \notion{combinatorial Morse function}, or a \notion{height function}, if adjacent vertices are mapped to different values (the slogan is: $h$ is non-constant on edges).
For any $t\in T$, denote by $X^{<t}$ the subcomplex spanned by vertices $v$ with $h(v) < t$ and denote by $X^{\leq t}$ the subcomplex of those $v$ with $h(v) \leq t$. The \notion{descending link} $\DescLink(v)$ of the vertex $v$ is the subcomplex of the link $\Link(v)$ spanned by those neighbors $w$ of $v$ \notion{below} $v$, i.e., satisfying $h(w) < h(v)$. From $h$ being non-constant on edges and Observation~\ref{obs:app-cone-off}, we deduce:
\begin{corollary}\label{cor:morse-theory-local}
$X^{\leq t}$ is obtained from $X^{< t}$ by coning off the descending links $\DescLink(v)$ of all vertices of height $h(v)= t$. If all these descending links are $m$-connected, the inclusion $X^{<t} \hookrightarrow X^{\leq t}$ induces isomorphisms in homotopy groups $\pi_d$ for $d \leq m$ and an epimorphism in $\pi_{m+1}$.
\end{corollary}
We call the height function $h$ \notion{discrete} if for any pair $s,t\in T$ the \notion{interval of values}
\(
  \{\, h(v) \,\mid\, v\in X \text{\ and\ } s \leq h(v) \leq t \,\}
\)
is finite. In that case, we can compare the homotopy type of a sublevel set to the whole complex.
\begin{proposition}[Morse Lemma for simplicial complexes]\label{app:morse-theory-global}
If $h$ is a discrete height function on $X$ and all descending links above level $s\in T$  are $m$-connected, then the inclusion of $X^{\leq s} \hookrightarrow X$ induces isomorphisms in homotopy groups $\pi_d$ for $d\leq m$ and an epimorphism in $\pi_{m+1}$.
\end{proposition}
\begin{proof}
For any inclusion $X^{\leq s} \hookrightarrow X^{\leq t}$, the claim follows from Corollary~\ref{cor:morse-theory-local} as the interval of values between $s$ and $t$ is finite. The whole complex $X$ is the direct limit of the $X_t$, and the homotopy group functor commutes with taking direct limits.
\end{proof}

Bestvina--Brady~\cite{BB97} applied this line of reasoning to piecewise Euclidean complexes in such a way that its parallels with Morse theory become apparent. Let $K$ be a piecewise Euclidean cell complex, and denote its vertex set by $K^{(0)}$. A function $h : K \to \R$ is called a \notion{discrete Morse function} or a \notion{height function} provided that the image $h(K^{(0)})$ is a discrete subset of $\R$ and that $h$ is affine on closed cells and non-constant on edges. 
Since all cells carry a Euclidean structure, we can think of links geometrically as spaces of \notion{directions}. A \notion{direction} based at $v$ pointing into the coface $c$ is an equivalence class of straight line segments starting at $v$ with the other endpoint in $c$ (the Euclidean structure of the cell $c$ tells us what such straight line segments are). Two such segments are equivalent if one is an initial segment of the other. The directions based at $v$ pointing into $c$ form a spherical polyhedron $\lk_c(v)$. For instance, if $c$ is a cube, directions based at a corner $v$ of $c$ pointing into $c$ form a spherical simplex. The cofaces of $v$ form a partially ordered set ordered by inclusion. This poset encodes the intersection pattern of the complex $K$ in the neighborhood of $v$. The link $\lk_K(v)$ of $v$ is the spherical CW comlpex obtained by gluing the pieces $\lk_c(v)$ for all cofaces $c$ of $v$ along their intersections.

The \notion{descending link} of $v$ consists of the contributions of those proper cofaces $c$ such that $h$ attains its maximum on $c$ at $v$. The \notion{ascending link} of $v$ comes from those proper cofaces such that $h$ attains its minimum on $c$ at $v$. We can give an interpretation of the geometric realization of the descending and ascending links of $v$ as subspaces of the realization $|K|$ of $K$ as follows. Choose $\varepsilon>0$ so that
\[
  \varepsilon \leq | h(w) - h(v) |
\]
for all vertices $w$ adjacent to $v$. The descending link of $v$ is the intersection of the level set $|K|^{=h(v)-\varepsilon} := \{ x \in |K| : h(x) = h(v) - \varepsilon \}$ with the descending star of $v$, i.e., the union of all cells containing $v$ as their top vertex. The homeomorphism is given as follows: let $\xi\in\lk_c(v)$ be a direction based at $v$ pointing into $c$ and assume that $v$ is the top vertex of $c$. Then the line segment starting at $v$ in the direction $\xi$ intersects the level set
\(
  |c|^{=h(v)-\varepsilon} := \{ x \in |c| : h(x) = h(v) - \varepsilon \}
\)
in a unique point. This defines a homeomorphism from the spherical polyhedron $\lk_c(v)$ to the euclidean polyhedron $|c|^{=h(v)-\varepsilon}$.

If $u$ and $v$ are vertices of the same height $t$, they are not connected by an edge, as $h$ would be constant on such an edge. It follows that there is no cell $c$ that has both $u$ and $v$ as vertices of maximum height. In particular, the descending links of $u$ and $v$ are disjoint subspaces of the sublevel set $|K|^{\leq s}:=\{ x \in |K| : h(x) \leq s\}$ where $s<t$ is chosen so that no vertex $w$ satisfies $h(w)\in(s,t)$.

\begin{proposition}[Morse Lemma for piecewise Euclidean complexes]\label{app:pe-morse-theory}
Let $K$ be a piecewise Euclidean complex, and  $h : K \to \R$ a discrete Morse function. Assume that the descending links of all vertices $v$ with $h(v)\in(s,t]$ are $m$-connected. Then the inclusion $|K|^{\leq s} \hookrightarrow |K|^{\leq t}$ induces isomorphisms in homotopy groups $\pi_d$ for $d\leq m$ and an epimorphism in $\pi_{m+1}$.
\end{proposition}

We remark that the inclusion of $|K^{\leq s}|$ into $|K|^{\leq s}$ is a homotopy equivalence. Therefore, under the hypotheses of the Morse Lemma, one obtains that also the inclusion $K^{\leq s} \hookrightarrow K^{\leq t}$ induces isomorphisms in homotopy groups $\pi_d$ fro $d \leq m$ and an epimorphism in $\pi_{m+1}$.

\subsection{Morse theory on the barycentric subdivision}\label{sub-bad-sim}
We will use the \emph{bad simplex argument} of ~\cite[Section 2.1]{HV17}. Here, we shall provide a treatment in terms of combinatorial Morse theory. Let $X$ be a simplicial complex. By~$X^\circ$ we denote its barycentric subdivision, and by $\sigma^\circ$ we denote the barycenter of the simplex $\sigma$. Note that $\sigma^\circ$ is a \emph{vertex} in $X^\circ$. Simplices in~$X^\circ$ correspond to chains of simplices in $X$. In particular if $\sigma^\circ$ and $\tau^\circ$ are adjacent in $X^\circ$, the dimensions of $\sigma$ and $\tau$ differ. Thus, on the barycentric subdivision $X^\circ$, we can use  dimension as a secondary function to break ties. Instead of making this generic, we shall illustrate the method by means of example. 

Assume that for each simplex $\sigma$, we are given a subset $\bar\sigma\subseteq\sigma$ of vertices such that the following two conditions hold:
\begin{enumerate}
  \item
    For $\sigma\subseteq\tau$, we have $\bar\sigma \subseteq \bar\tau$. (Monotonicity)
  \item
    For any simplex, $\bar{\bar\sigma}=\bar\sigma$. (Idempotence)
\end{enumerate}
The simplex $\sigma$ is \notion{good} if $\bar\sigma$ is empty and \notion{bad} if $\bar\sigma=\sigma$. The monotonicity condition implies that faces of good simplices are good. In particular, the good simplices form a subcomplex $X^\text{good}$ of $X$. The \notion{good link} $G_\sigma$ of a simplex $\sigma$ is induced by its proper cofaces $\tau$ with $\bar\sigma=\bar\tau$. It follows also from monotonicity that the good link is a subcomplex of the link, i.e., it is closed with respect to taking faces.

In \cite[Section 2.1]{HV17}, Hatcher--Vogtmann phrase the bad simplex argument in terms of the collection of bad simplices. The equivalence to our approach is seen as follows:
\begin{remark}
If two bad simplices $\sigma$ and $\tau$ span a simplex, then the span $\sigma\union\tau$ will also be bad by monotonicity. Conversely, consider a collection $\mS$ of simplices with the property that if $\sigma\union\tau$ is a simplex, $\sigma,\tau\in\mS$ implies $\sigma\union\tau\in\mS$. Then defining $\bar\sigma$ as the union of all faces of $\sigma$ that belong to $\mS$ satisfies monotonicity and idempotence. With this definition, $\mS$ coincides with the class of bad simplices, i.e., $\sigma\in\mS$ if and only if $\sigma=\bar\sigma$. 
\end{remark}
\begin{proposition}[Bad simplex argument]\label{app:gated-morse-theory}
Suppose there is $m$ such that for all bad simplices $\sigma$, the good link $G_\sigma$ is $(m-\dim(\sigma))$-connected. Then the inclusion $X^\text{good} \hookrightarrow X$ induces an isomorphism in homotopy groups $\pi_d$ for $d\leq m$ and an epimorphism in $\pi_{m+1}$.
\end{proposition}

\begin{proof}
We want to use the cardinality $\#\bar\sigma$ as a Morse function, noting that it is integer-valued on the vertices $\sigma^\circ$ of the barycentric subdivision $X^\circ$. However, we may have $\bar\sigma=\bar\tau$ for a simplex $\tau$ and a proper face $\sigma$ of $\tau$. Thus, we use the dimension to break ties. More precisely, on the barycentric subdivision $X^\circ$, we consider the following height function:
\begin{align*}
  h : X^\circ & \to \Z \times \Z \\
  \sigma^\circ & \mapsto \langle \#\bar\sigma, -\dim(\sigma) \rangle
\end{align*}
We use the lexicographic order on $\Z\times\Z$.

If $X$ has infinite dimension, $h$ may not be discrete. However, by passing to the $(m+1)$-skeleton, we may assume w.l.o.g. that $X$ is of finite dimension at most $m+1$.

A simplex $\sigma$ is good if and only if $\bar\sigma$ is empty, i.e., exactly when $\#\bar\sigma = 0$. Thus, $h(\sigma^\circ)<(1,-m-1)$ if $\sigma$ is bad. Conversely, assume $h(\sigma^\circ)<(1,-m-1)$. Then $\sigma$ is good or $\bar\sigma$ consists of a single vertex and $-\dim(\sigma)<-m-1$. The latter case does not occur as $X$ is assumed to be of dimension at most $m+1$. Thus, the barycentric subdivision of the subcomplex $X^\text{good}$ of good simplices is the $h$-sublevel complex $X^\circ_{<(1,-m-1)}$.

In the barycentric subdivision, the link of a vertex $\sigma^\circ$ is best described by recalling that $\sigma$ is a simplex in $X$. Then, adjacent vertices are faces and cofaces of $\sigma$; correspondingly, the link of $\sigma^\circ$ in $X$ is the join $\DownLink(\sigma)\join\UpLink(\sigma)$ of the \notion{down-link} $\DownLink(\sigma)$, spanned by the barycenters of proper faces of $\sigma$; and the \notion{up-link} $\UpLink(\sigma)$, spanned by the barycenters of proper cofaces of $\sigma$. The down-link is the barycentric subdivision of the boundary sphere $\partial \sigma$ whereas the up-link is the geometric realization of the poset of proper cofaces of $\sigma$, i.e., it can be thought of as barycentric subdivision of the ordinary link of the simplex $\sigma$ in the complex $X$.
The join decomposition $\Link(\sigma^\circ)=\DownLink(\sigma)\join\UpLink(\sigma)$ is inherited by the descending link $\DescLink(\sigma^\circ)=\DescDownLink(\sigma)\join\DescUpLink(\sigma)$ where $\lk^{\downarrow}_{\pm}(\sigma) := \lk_{\pm}(\sigma) \intersect \DescLink(\sigma^\circ)$ is the descending part of the up-link or down-link.

Let $\rho$ be a proper face of $\sigma$. The vertex $\rho^\circ$ lies in the down-link of $\sigma^\circ$ and it is descending if and only if $\bar\rho$ is strictly smaller than $\bar\sigma$ since the second coordinate of $h$ increases. Note that $\bar\rho\subseteq\bar\sigma$ follows from monotonicity. On the other hand, $\bar\sigma\subseteq\rho$ implies $\bar\sigma=\bar{\bar\sigma}\subseteq\bar\rho$ whence $\bar\sigma=\bar\rho$; and conversely $\bar\sigma=\bar\rho$ implies $\bar\sigma\subseteq\rho$. Thus, $\rho^\circ\not\in\DescDownLink(\sigma)$ if and only if $\bar\sigma\subseteq\rho$. This translates as:
\[
  \DescDownLink(\sigma) = \{\, \rho^\circ \in \DownLink(\sigma) \,\mid\, \bar\sigma \not \subseteq \rho \,\}
  .
\]
If $\sigma$ is good, $\DescDownLink(\sigma)=\emptyset$ since there is no proper face $\rho$ that does not contain the empty set $\bar\sigma$. If $\sigma$ is bad, $\DescDownLink(\sigma)$ is the entire boundary sphere $\partial\sigma$ since no proper face $\rho$ contains all of $\bar\sigma=\sigma$. If $\sigma$ is neither good nor bad, $\DescDownLink(\sigma)$ contractible as it is the boundary sphere $\partial\sigma$ with the star of $\bar\sigma$ removed.

The descending up-link $\DescUpLink(\sigma)$ is spanned by those $\tau^\circ\in\UpLink(\sigma)$, for which the first coordinate of $h$ does not increase (it cannot decrease by monotonicity). Thus $\DescUpLink(\sigma)$ is the barycentric subdivision of the good link $G_\sigma$. Hence $\DescUpLink(\sigma)$ is $(m-\dim(\sigma))$-connected by hypothesis, provided that $\sigma$ is bad.

Hence, for any vertex $\sigma^\circ$ outside the sublevel set, the descending link $\DescLink(\sigma^\circ)=\DescDownLink(\sigma)\join\DescUpLink(\sigma)$ is at least $m$-connected: if $\sigma$ is bad, the descending down-link is the $(\dim(\sigma)-2)$-connected sphere $\partial\sigma$ and the descending up-link in $(m-\dim(\sigma))$-connected; if $\sigma$ is not bad, the descending link is even contractible and the descending up-link does not matter. The claim now follows from the Morse Lemma.
\end{proof}
In the main body, we make use of the bad simplex argument in two particular incarnations.
\begin{example}\label{app:colored-vertices}
Assume that the vertices of $X$ are colored. We consider a simplex $\sigma$ bad if no vertex has a unique color among the vertices of $\sigma$. Then $\bar\sigma$ is the set of vertices in $\sigma$ whose color features in $\sigma$ at least twice. The conditions of monotonicity and idempotence are clearly satisfied. The subcomplex $X^\text{good}$ then consists precisely of those simplices wherein each vertex has a unique color. The good link $G_\sigma$ of $\sigma$ is induced by those cofaces $\tau$ such that $\tau\setminus\sigma$ is a good simplex and uses no color that is already used in~$\sigma$.

For concreteness, let $\pi : X \to Y$ be a simplicial map and consider the vertices of $Y$ as the colors. Then $X^\text{good}$ consists precisely of those simplices $\sigma$ such that the restriction $\pi|_{\sigma}$ is injective. The good link $G_\sigma$ of a bad simplex $\sigma$ consists of those simplices $\tau\in\Link(\sigma)$ such that $\pi|_\tau$ is injective and such that the images $\pi(\tau)$ and $\pi(\sigma)$ are disjoint.
\end{example}
\begin{example}\label{app:vertex-partition}
Partition the vertices of $X$ into good vertices and bad vertices. For a simplex $\sigma$ let $\bar\sigma$ denote the set of bad vertices in $\sigma$. Clearly, this definition satisfies monotonicity and idempotence. Equivalently, a simplex is bad if all its vertices are bad. Then $X^\text{good}$ is the full subcomplex spanned by the good vertices. The good link of a simplex is also easily seen to be $G_\sigma = \Link(\sigma) \intersect X^\text{good}$.
\end{example}

\subsection{A variation of the arc complex}
Let $S$ be a surface of negative Euler characteristic with two disjoint sets of marked points $P=\{p_0,\ldots,p_m\}$ and $Q=\{q_0,\ldots,q_n\}$. We regard interior marked points as punctures. An arc is a simple curve with one endpoint in $P$ the other endpoint in $Q$ and otherwise disjoint from $P\union Q$. We consider arcs up to isotopy relative to $P \union Q \union \partial S$, i.e., isotopic arcs will define the same vertex in the arc complex. We endow $S$ with a complete hyperbolic metric, so that interior marked points form the set of cusps. Then each arc can be represented by a unique geodesic. 
Let $\acx(S,P,Q)$ be a simplicial complex whose $k$-simplices are collections $\{ \alpha_0,\ldots,\alpha_k\}$ of (isotopy classes of) arcs that are pairwise non-isotopic and disjoint except possibly at the endpoints (thinking of arcs as geodesics connecting cusps, this amounts to the geodesics being pairwise disjoint). One might call $\acx(S,P,Q)$ the \notion{bipartite arc complex} to distinguish it from the arc complex of a surface.

\begin{proposition}\label{app:arc-cx}
The complex $\acx(S,P,Q)$ is contractible.
\end{proposition}
This is a variation on~\cite[Theorem~1.6]{Har85}. The difference is that our punctures $P$ are not located on the boundary. Instead of recycling Harer's proof, we give a proof using more modern technology, namely the surgery flow method introduced by Hatcher~\cite{Ha91}.

\begin{proof}
Hatcher--Vogtmann~\cite[Lemma~2.9]{HV17} have turned Hatcher's idea of surgery flow into a fairly generic method. In establishing the hypotheses, we shall already employ their notation.

Fix an arc $w$ from $P$ to $Q$. For any other arc $v$, we define the \notion{complexity} $c(v)$ as the geometric intersection number of $v$ and $w$. The geometric intersection number will be realized if the representatives $v$ and $w$ are chosen to be geodesic with respect to a hyperbolic metric. For a simplex $\sigma$ in the arc complex, we define $c(\sigma)$ as the sum of the complexities over all vertices of $\sigma$. Note that the star $Y$ of $w$ in $X:=\acx(S,P,Q)$ consists precisely of the zero-complexity simplices.

For an arc $v$ that intersects $w$, we consider the point of intersection closest along $w$ to the $Q$-end of $w$. We surger $v$ along the final segment of $w$ and obtain two new arcs: one, which we call $\Delta v$ from the $P$-end of $v$ to the $Q$-end of $w$; and another one, which we discard, connecting $Q$ to $Q$. Note that $\Delta v$ and $v$ do not intersect and therefore span an edge in the arc complex. Moreover, we find $c(\Delta v) < c(v)$.

For a simplex $\sigma$ not in $Y$, there is at least one vertex (arc) that intersects~$w$. Let $v_\sigma$ be that intersecting arc whose point of intersection with $w$ is closest to the $Q$-end. Note that $\sigma$ and $\Delta v_\sigma$ span a simplex. Also note that $v_\sigma=v_\tau$ for any face $\tau$ of $\sigma$ that contains $v_\sigma$.

In this situation~\cite[Lemma~2.9]{HV17} applies and we conclude that the (contractible) star $Y$ of $w$ is a deformation retract of $X=\acx(S,P,Q)$. Hence $\acx(S,P,Q)$ is contractible.
\end{proof}

Now, we fix a subset $P' \subseteq P$ and consider the subcomplex $\bacx(S,P,P',Q)$ of the barycentric subdivision $\bacx(S,P,Q)$ spanned by those vertices $\sigma^\circ$ where the $P$-endpoints of the arcs in $\sigma$ form a superset of $P'$. We make use of the following result in the analysis of the extended tethered-handle complex $\thas(W,Q)$ in the surface setting.
\begin{proposition}\label{app:two-sided-arc-complex}
The complex $\bacx(S,P,P',Q)$ is contractible.
\end{proposition}
\begin{proof}
For a simplex $\sigma$ of $\acx(S,P,Q)$, we denote by $P(\sigma)$ the set of $P$-points issuing arcs in $\sigma$. For non-empty $A\subseteq B \subseteq P$, let $X(A,B)$ denote the subcomplex of $\bacx(S,P,Q)$ spanned by $\{\ \sigma^\circ \,\mid\, A \subseteq P(\sigma) \subseteq B \,\}$. By $Y(A,B)$ we denote the complex spanned by $\{\ \sigma^\circ \,\mid\, P(\sigma) \subseteq B \text{\ and\ } A\intersect P(\sigma) \neq \emptyset\,\}$. Note that $X(A,B)$ is a subcomplex of $Y(A,B)$.

For a subset $A\subseteq P$, we denote by $\sigma\mid_A$ the \notion{restriction} of $\sigma$ to $A$, i.e., the collection of those arcs in $\sigma$ issuing from $A$. For $A \subseteq B \subseteq C$, restriction to $B$ defines simplicial retractions
\(
 \mid_B:  X(A,C) \to X(A,B)
\)
and
\(
  \mid_B: Y(B,C) \to Y(B,B)
  \text{.}
\)
By~\cite[1.3~Homotopy Property]{Qui78}, these restriction maps are deformation retractions, i.e., they are homotopic to the identity map on their respective domains of definition.
Note that $Y(B,B)$ can be identified with the barycentric subdivision of an arc complex $\acx(S',B,Q)$ where $S'$ is obtained by puncturing $S$ at $P \setminus B$. In particular, the complexes $Y(B,C)$ are all contractible.

We show by induction on the cardinality of $A$ that the complexes $X(A,B)$ are contractible. If $A$ contains just a single element, we find $X(A,B)=Y(A,B)$ and $Y(A,B)$ is contractible.

To complete the induction step, we need to consider a puncture $b\in B$ with $b\not\in A$ and put $A':=A\union\{b\}$. We want to understand the inclusion $X(A',B) \hookrightarrow X(A,B)$ by means of a bad simplex argument as in Example~\ref{app:vertex-partition}. In particular, we call a vertex $\sigma^\circ$ bad if $b\not\in P(\sigma)$. Then $X(A',B)$ is the good subcomplex of $X(A,B)$. Let $\{ \tau_0^\circ, \ldots, \tau_l^\circ\}$ be a bad $l$-simplex with $\tau_0 \subset \tau_1 \subset \cdots \subset \tau_l$. Note that for a good vertex $\sigma^\circ$, we never have $\sigma\subseteq\tau_l$ because $b$ is used in $\sigma$. Therefore, the good link $G$ of $\{ \tau_0^\circ, \ldots, \tau_l^\circ\}$ consists of good simplices $\{ \sigma_0^\circ, \ldots, \sigma_k^\circ\}$ with $\tau_l \subset \sigma_i$. Restriction to $\{b\}$ defines a deformation retraction of $G$ onto the subcomplex spanned by those $\sigma^\circ$ where $\sigma$ contains $\tau_l$ and $P(\sigma\setminus\tau_l)=\{b\}$. This complex can be identified with the contractible complex $\bacx(S',\{b\},Q)$ where $S'$ is obtained from $S$ by cutting along $\tau_l$.

The claim now follows as $\bacx(S,P,P',Q)=X(P',P)$.
\end{proof}

\subsection{Join complexes}\label{subsection:join} The concept of a {\em join complex}, introduced by Hatcher--Wahl in \cite[Section 3]{HW10}  will be another useful tool for us in order to analyze connectivity properties. We review the basics here, referring the reader to \cite{HW10} for more details. 

\begin{definition}[Join complex]\label{defn-join}
A {\em join complex} over a simplicial complex $L$ is a simplicial complex $K$ together with a simplicial map $\pi: K\to L$ satisfying the following properties:
\begin{enumerate} [label=(\arabic*)]
    \item $\pi$ is surjective.
    \item $\pi$ is injective on individual simplices. 
    \item For each $p$-simplex $\sigma = \langle x_0,\cdots,x_p\rangle$ of $L$, the subcomplex $K(\sigma)$ of $K$ consisting of all the $p$-simplices that project to $\sigma$ is the join $K_{x_0} (\sigma)\join \cdots \join K_{x_p}(\sigma)$ of the vertex sets of $K_{x_i}(\sigma) :=  K(\sigma) \cap \pi^{-1}(x_i)$.
\end{enumerate}
Note that $K(\sigma)$ need not be equal to $\pi^{-1}(\sigma)$. If all the inclusions $K_{x_i}(\sigma) \subseteq \pi^{-1}(x_i)$  are in fact equalities, then we will call $K$ a {\em complete join complex} over $L$.
A reformulation of (3) which sometimes is helpful is the following condition; see \cite{HW10} for a proof of their equivalence:

\begin{enumerate}{}{\setlength{\leftmargin}{22pt}\setlength{\labelwidth}{16pt}\setlength{\labelsep}{5pt}\setlength{\itemsep}{3pt}}

\item[(3$'$)]  A collection of vertices $(y_0,\cdots,y_p)$ of $K$ spans a $p$-simplex if and only if for each $y_i$ there
exists a $p$-simplex $\sigma_i$ of $K$ such that $y_i\in \sigma_i$ and $\pi(\sigma_i)=\langle \pi(y_0),\cdots,\pi(y_p)\rangle$.
\end{enumerate}
 \end{definition}

As mentioned above, it is possible to use the concept of join in order to analyze the connectivity of a given complex, if this belongs to the following class of complexes, introduced by Hatcher--Wahl \cite{HW10}: 

\begin{definition}
A simplicial complex  is {\em weakly Cohen--Macaulay} ($wCM$ for short) of dimension $n$ if it is $(n-1)$-connected and the link of each $p$-simplex is $(n-p-2)$-connected. 
\label{def:wCM}
\end{definition}

Armed with this definition, one has the following result in the case of a join complex:

\begin{proposition}\cite[Proposition 3.5]{HW10} \label{prop-cjoin-conn}
If $K$ is a complete join complex over a $wCM$ complex $L$ of dimension $n$, then $K$ is also $wCM$ of dimension~$n$.
\end{proposition}

\begin{remark}\label{rem-cjoin}
If $\pi: K\to L$ is a complete join, then $L$ is a retract of $K$. In fact, we can define a simplicial map $s:L\to K$ with $\pi\circ s = \id_L$, by sending a vertex $v\in L$ to any vertex in $\pi^{-1}(v)$ and then extending it over to all simplices, which can be done since $\pi$ is a complete join. In particular, if $K$ is $n$-connected, so is $L$.
\end{remark}

However, we will encounter situations when we will only have a join complex instead of a complete one. In these cases we will make use of the following theorem, also due to Hatcher--Wahl \cite{HW10}: 

\begin{theorem}\cite[Theorem 3.6]{HW10} \label{thm-conn-join}
Let $K$ be a join complex over a $wCM$ complex $L$ of dimension $n$. Suppose that, for each $p$-simplex $\sigma \subset K$, one has that $\pi(\lk_K(\sigma))$ is $wCM$ of dimension $n-p-2$. Then $K$ is $(\frac{n}{2}-1)$-connected.
\end{theorem}

\subsection{Nerve covers}

The \textit{nerve} of a family of sets $(K_i)_{i\in I}$ is the simplicial complex $\mathcal{N}(K_i)$ defined on the vertex set $I$ which contains a simplex $J\subseteq I$ if $\bigcap_{j\in J}K_j$ is non-empty. The following result allows to identify the homotopy groups of  certain simplicial complexes 
with those of a nerve complex, up to a fixed dimension.  

\begin{proposition}\cite[Theorem 6]{Bj03}
\label{app:B14}
Let $K$ be a connected simplicial complex 
and $(K_i)_{i\in I}$ a family of
subcomplexes that covers $K$. Suppose that every non-empty finite intersection
$K_{i_1}\cap K_{i_2}\cap\ldots\cap K_{i_l}$ is $(k-l+1)$-connected for $l\geq1$. Then there is a map $f:\Vert K\Vert\rightarrow\Vert\mathcal{N}(K_i)\Vert$ which induces isomorphisms of homotopy groups $f_j^\ast:\pi_j(K)\rightarrow\pi_j(\mathcal{N}(K_i))$ for all $j\leq k$. 
\end{proposition}

\subsection{Fiber connectivity}
Given a simplicial map $f : X \to Y$, there are several results relating the connectivity of $X$ and $Y$ provided that we have information about the connectivity of the fibers. A useful instance is due to Quillen~\cite[Proposition~7.6]{Qui78}, who stated his result in the more general terms of posets and their geometric realizations. Here, we state the proposition specialized for simplicial maps.
\begin{proposition}[Quillen's Fiber Theorem for simplicial complexes]\label{app:quillen}
Let $ f : X \to Y$ be a simplicial map and assume that the preimage $f^{-1}(\sigma)$ of each closed simplex $\sigma$ in $Y$ is an $n$-connected subcomplex of $X$. Then $X$ is $n$-connected if and only if $Y$ is $n$-connected.
\end{proposition}
One can also make use of fibers above open simplices. Technically, it is convenient to phrase the statement in terms of the barycentric subdivision.
\begin{lemma}\cite[Lemma 2.8]{HV17}
\label{lem:fiber_conn}
Let $f:L\rightarrow K$ be a simplicial map of simplicial complexes. Suppose that $K$ is $n$-connected and the fibers $f^{-1}(\sigma^\circ)$ over the barycenters $\sigma^\circ$ of all $k$-simplices in $K$ are $(n-k)$-connected. Then $L$ is $n$-connected. 
\end{lemma}

\section{The inclusion and intersection properties}
\label{sec:intersection}

In this section we prove the inclusion, intersection, and cancellation properties for suited submanifolds of the Cantor manifolds that appear in the statements of Theorems \ref{cor:surfaces}, \ref{thm:3D}, \ref{thm:3Dgen} and \ref{thm:highdim}. None of the material here is new, and the arguments will be well-known to the appropriate community of experts. 


\subsection{Dimension 2} The two-dimensional case is by far the easiest, so we deal with it first. First, the inclusion property appears, for example, in \cite[Section 3.6]{FM}. In turn, the intersection property is a direct consequence of the Alexander method, see \cite[Section 2.3]{FM}. Finally, the cancellation property boils down to the classification theorem of compact surfaces.

\medskip




%

We now move onto the cases of Cantor manifolds of dimension bigger than two. First, the cancellation property in dimension 3 is afforded by the Prime Decomposition Theorem for 3-manifolds, while the fact that the Cantor manifolds of Theorem \ref{thm:highdim} also have this property is an immediate consequence of a result of Kreck \cite{Kr99} (see also \cite[Corollary 6.4]{GRW18}), stated as Theorem \ref{thm:cancellation} above.

We commence by briefly recalling the definition and properties of a  {\em Dehn twist} along a sphere; see \cite{Kr79,BBP20}. Let $M$ be a connected, orientable smooth manifold of dimension $n+1\ge 3$, $S$ a smoothly embedded $n$-sphere in $M$ and $U \cong S \times [0,1]$ a tubular neighborhood of $S$. Let $\gamma$ be a generator of $\pi_1(SO(n+1)) \cong \mathbb Z_2$. The {\em Dehn twist} $T_S \in \Map(M)$ is the isotopy class of the diffeomorphism
which is the identity on $M\setminus U$ and is given by $(x,t)\to (\gamma(t)x,t)$ on $U$.
The mapping class $T_S$ depends only on the isotopy class of $S$; moreover, $T_S$ has order at most 2. 

Recall that an $n$-dimensional manifold $M$ is {\em stably parallelizable} if $\tau_M \oplus \lambda \cong M \times \mathbb R^{n+1}$, where $\tau_M$ is the tangent bundle of $M$ and $\lambda$ is the trivial line bundle. We have: 

\begin{lemma}\label{lem-dehntw-nontr-high}
Let $n\geq 2$ and $M =  N \# (\mathbb{S}^1\times \mathbb{S}^n)$ be a smooth closed stably parallelizable manifold. Let $T_S$ be the Dehn twist along the $n$-sphere $S=0\times \bS^n \subseteq N\#(\bS^1\times \bS^n)$. Then $T_S$ is a nontrivial element of order $2$ in $\Map(M)$.
\end{lemma}
\begin{proof}
When $n=2$, this is due to Laudenbach \cite{Lau73}. The proof we give here follows closely to \cite[Sections 4 \&  5]{BBP20}. As discussed above, $T_S$ has  order at most 2, so it suffices to show $T_S$ is nontrivial.

Any diffeomorphism $f\in \Diff(M)$  induces a bundle map $D_f: TM\to TM$ by taking its  differential. Fix a parallelization of the stable tangent bundle  $TM \oplus \epsilon^k$, where $\epsilon$ is the trivial line bundle.
The differential $D_f$ then induces a map $\bar{f}: M\to SO(n+k+1)$. If $f$ is isotopic to the identity map, we would have that $\bar{f}$ is homotopic to the constant map. In particular, $\bar{f}\mid_{S^1\times 0}$ is homotopic to a constant path in $SO(n+k+1)$, contrary to the construction of the twist $T_S$. Indeed, since $T_S$ is defined via a nontrivial element in $\pi_1(SO(n+1))$, and the stabilization map  $SO(n+1) \hookrightarrow SO(n+k+1)$ induces an isomorphism at the level of fundamental groups, we deduce that  $\bar{f}\mid_{S^1\times 0}$ must be nontrivial in $\pi_1(SO(n+k+1))$.
\end{proof}

Now following the same proof of \cite[Lemma 2.2]{Lan21}, we have the following.
\begin{corollary}\label{cor-dtw-nt-gen}
Let $N$ be a manifold with at least two sphere boundaries $S_1$ and $S_2$, and suppose that the manifold $M$ obtained by gluing together $S_1$ and $S_2$ is stably parallelizable. Then $T_{S_i}$ is nontrivial in $\Map(N)$.
\end{corollary}
\begin{proof}
First, note that $M$ is the connected sum of some manifold with $\mathbb{S}^1\times \mathbb{S}^n$. Now, there is an induced map $\iota:\Map(N)\to \Map(M)$. By Lemma \ref{lem-dehntw-nontr-high},  $\iota(T_{S_i})$ is non-trivial and hence the same holds for $T_{S_i}$.
\end{proof}

\subsection{Dimension 3} The inclusion property for 3-manifolds appears as \cite[Proposition 2.3]{HW10}, and so we focus on the intersection property.  We start with the following well-known result, which gives information about the subgroup of $\Map(M)$ generated by Dehn twists about boundary spheres. Recall that a manifold has {\em spherical boundary} if every boundary component is diffeomorphic to a sphere. We have: 

\begin{theorem}
Let $M$ be a closed orientable 3-manifold and $M_b = M~\#_b~\mathbb{B}^3$, the connected sum of $M$ with $b$ copies of 3-dimensional balls. Denote the boundary spheres by $S_1, \ldots, S_b$, and write $T_{S_i}$ for the corresponding sphere twist. Let  $\mathbb{T}_b\leq\Map(M_b)$ be the subgroup generated by the twists $T_{S_i}$ for $1\leq i \leq b$. Then: 
\begin{enumerate}
  \item If $\mathbb{T}_1 =1$, then for any $b>1$ one has  $\mathbb T_b= (\Z_2)^{b-2}$, where the only relation is given by $T_{S_1}\cdots T_{S_b}=1$. 
    \item If $\mathbb{T}_1 =\Z_2$, then for any $b>1$ one has $\mathbb T_b= (\Z_2)^{b}$.  
\end{enumerate}
\label{thm:3d-boundarytwist} 
\end{theorem}

\begin{proof}
First of all, by Corollary \ref{cor-dtw-nt-gen}, sphere twists have order $2$ when $b\geq 2$. On the other hand, when $b=1$ the sphere twist could either be trivial or of order $2$, see \cite[Remark 2.4]{HW10} for more details. 

If $\mathbb{T}_1 =0$, we would have $T_{S_1}\cdots T_{S_b}=1$; see the proof of \cite[Lemma 2.1]{Lan21} or Lemma \ref{lemm-prod-dtw-trv} below.  We need to show there is no further relation when $b\geq 3$. Suppose that some partial product $T_{S_{i_1}}\cdots T_{S_{i_k}}$ were equal to the identity in $\Map(M_b)$. Let $S'$ be a sphere which cuts off a ball containing exactly $S_{i_1},\ldots,  S_{i_k}$, and write $M_{b'}$ for the other component of the complement of $S'$. Note that $b'\geq 2$. Hence $T_{S'}$ is nontrivial in $M_{b'}$ by Corollary \ref{cor-dtw-nt-gen}. On the other hand in $\Map(M_b)$, we have $T_{S'} =  T_{S_1}\cdots T_{S_k}$ which equals $1$. But by the inclusion property, this means $T_{S'}=1$ in $\Map(M_{b'})$, which is a contradiction.


Finally, the same argument works when $\mathbb T_1 =\Z_2$, with the only difference that in this case there is no product relation.
\end{proof}

We will now use the ideas in \cite[Section 4,5]{BBP20} in order to deduce that a given mapping class is a product of Dehn twists about boundary spheres. We need some preliminaries first. 

Again, let $M$ be a compact connected orientable 3-manifold with spherical boundary. A diffeomorphism $f\in \Diff(M,\partial M)$ induces a bundle map $D_f: TM\to TM$. Since every 3-manifold is parallelizable \cite{Stiefel}, $D_f$ further induces a map $M \to \SO(3)$. Now, if $C$ is a closed curve on $M$, or an arc connecting two boundary components, then $\mathcal{D}_f(C)\in  \pi_1(\mathrm{SO}(3)) \cong \Z_2$. Moreover, the value of $\mathcal{D}_f(C)$ only depends on the isotopy class (relative to endpoints, if any) of  $C$. The following may be deduced from the description of $\Map(M)$ in terms of automorphisms of $\pi_1(M)$; see \cite[Theorem 1.5]{Mc90} or \cite[Proposition 2.1]{HW10}: 

\begin{theorem}\label{thm-char-diff-tw-gel}
Let $M$ be a $3$-manifold with spherical boundary and at least two boundary components. Let $f\in \Diff(M,\partial M)$ such that:
\begin{enumerate}
    \item For every curve $C$, we have that $\mathcal{D}_f(C)=0$ and $f$ fixes the homotopy class of $C$;
    \item For every arc $A$ connecting two distinct boundary spheres, $f$ fixes the homotopy class (relative to endpoints) of $A$. 
\end{enumerate}
Then $f$ is  the Dehn twist along a separating sphere in $M$.
\end{theorem}

Now let $N$ be a connected orientable $3$-manifold with spherical boundary, and assume that $M$ is obtained from $N$ by attaching $P_1,P_2$ along two of its sphere boundaries, where $P_i =P\#_{d+1}\mathbb{B}^3$, with $P$ a fixed closed prime $3$-manifold or $\mathbb{S}^3$ and $d\geq2$. Let $M_i = N\cup P_i$. By the inclusion property (see e.g. \cite[Proposition 2.3]{HW10}), we regard  $\Map(N),\Map(M_1)$ and $\Map(M_2)$ as subgroups of $\Map(M)$. We finally prove:

\begin{theorem}
$\Map(M_1)\cap \Map(M_2) = \Map(N)$.
\end{theorem}
\begin{proof}
The ``$\supseteq$" part is obvious, so it suffices to show the ``$\subseteq$" part. For $i=1,2$, let  $S_i$ be the boundary sphere of $N$ to which $P_i$ is glued, and write $P_i'= P_i\setminus S_i$.
An element $f\in \Map(M_1)\cap \Map(M_2)$ fixes the isotopy class of $S_i$, and hence the homotopy class of $S_1 \cup S_2$. By Laudenbach's theorem \cite[Theorem III.1.3]{Lau74}, homotopic systems of spheres are isotopic. Thus, up to isotopy, we can assume that $f\mid_{S_i} = \id$ and, as a consequence, $f(P_i)=P_i$.  Hence $f$ decomposes as $f_1f_Nf_2$ where $f_i\in \Map(P_i)$ and $f_N\in \Map(N)$. Since the three maps commute with each other, we have $ff_N^{-1} = f_1f_2 \in\Map(M_1)\cap \Map(M_2)$. Thus it suffices to show that $f_1f_2\in \Map(N)$. 

Note that if two arcs or curves in $P_i'$ are isotopic in $M$, they are already isotopic in $P_i'$. Since $f\in \Map(M_i)$, we have that $\mathcal{D}_f$ values trivially at every curve and arc in $P'_i$. Moreover, as $f$ induces the identity map on $\pi_1(P_i)$, for $i=1,2$, it follows that  $f_i$ is the product of Dehn twists around all boundary spheres of $P_i$. Furthermore, since the Dehn twist around $S_i$ can be pushed to $N$, we may assume $f_i$ is the product of Dehn twists along all boundaries of $P'_i$, denoted ${T}_i$. 

At this point it suffices to show that following:  if ${T}_1 b{T}_2\in \Map(M_1)\cap\Map(M_2)$, then $T_1 T_2 \in \Map(N)$. Observe that if $T_1 T_2\in \Map(M_i)$, we would also have $T_i \in \Map(M_{3-i})$ since $T_{3-i}\in \Map(M_{3-i})$. Thus the theorem finally boils down to the following claim.

\smallskip

\textbf{Claim:} Fix $i=1,2$. If $T_{i}\in \Map(M_{3-i})$, then   $T_{i}= 1$ or $T_1 T_2 \in \Map(N)$.

\smallskip

We prove the claim for $i=2$. Suppose that $T_{2}\neq 1$. As an element in $\Diff(M_1,\partial M_1)$, $T_2$ fixes the homotopy class of every curve on $M_1$ and $\mathcal{D}_{T_1 T_2}$ values trivially at every curve. Moreover $T_2 $ fixes the homotopy class of every arc in $M_1$ connecting boundaries, because  $\pi_1(N) = \pi_1(M_1)\ast\pi_1(M_2)$. Thus by Theorem \ref{thm-char-diff-tw-gel}, $T_2 $ is a Dehn twist about a separating sphere in $M_1$. We can decompose this sphere further as the connected sum of two separating spheres $S_{P_1}$ and $S_N$ with $S_{P_1}\subset P_1$ and $S_N\subset N$. If $S_{P_1}$ is parallel to $S_1$, then we are done. If not, then it is the connected sum of some boundary spheres in $P_1'$. Note that $T_{S_{P_1}} \cdot T_{N} \cdot T_2=1$.

Assume that $T_{S_{P_1}}\neq T_1$. In particular $T_{S_{P_1}}$ is not the product of Dehn twists along all spheres of $P_1'$. Now, $S_N$ cuts off $M$ into two components $K_1$ and $K_2$, where assume that $P_1,P_2\subset K_2$. By choosing an appropriate separating sphere,  $K_2$ may be expressed as the connected sum of a $(2d+1)$-holed sphere $H_{2d+1}$ with another 3-manifold $K_3$, such that the boundary components of $H_{2d+1}$ correspond to $S_N$ and the $2d$ boundaries of $P_1', P_2'$. Now we can decompose the inclusion map $K_1\subset M$ as:
\[ K_1 \subset K_1\bigcup_{S_N}H_{2d+1} \subset (K_1\cup_{S_N}H_{2d+1}) \# K_3. \]

Again, the inclusion property tells us that both inclusion maps induce injective maps between the corresponding mapping class groups. But restricting to $K_1\bigcup_{S_N}H_{2d+1}$, we have $T_{S_{P_1}}  \cdot T_2 \cdot T_{N}\neq 1$  by Theorem \ref{thm:3d-boundarytwist}. Thus  $T_{S_{P_1}} \cdot T_{N} \cdot T_2\neq 1$ in $\Map(M)$. Therefore we must have  $T_{S_{P_1}}= T_1$ which implies $T_1 T_{2} = T_{N} \in \Map(N)$, as desired. 
\end{proof}


\subsection{High dimensions}  Finally, we settle the intersection and inclusion properties in higher-dimensional manifolds. Let $ W_g^b= \#_g~ \mathbb{S}^n \times \mathbb{S}^n ~\#_b ~\mathbb{B}^{2n}$, where $n\ge 3$. When $b=0$, we will simply write $W_g$. Before getting started, we need to quickly review Kreck's calculation of $\Map(W_g)$ \cite{Kr79}; see also \cite[Section 1.3]{Kr20}. Let $H_n(W_g) \cong \Z^{2g}$ be the integral $n$-th homology group of $W_g$, and $G_g$ the group of automorphisms of $H_n(W_g)$ that preserve both the intersection form and {\em Wall's quadratic form}; see \cite[Section 1.2]{Kr20} for an explicit description of $G_g$ as a finite index subgroup of a symplectic or orthogonal group. Kreck \cite{Kr79} showed:

\begin{theorem}\label{thm-ses-high-1}
The natural homomorphism $\mathfrak{h}: \Map(W_g) \to G_g$ is surjective. 
\end{theorem}

The kernel of the above surjective homomorphism is called the {\em Torelli subgroup} $\mathcal{T}_g$ of $\Map(W_g)$, so we have a short exact sequence
\[1 \to \mathcal{T}_g \to \Map(W_g)\xrightarrow[]{\mathfrak{h}} G_g\to 1.\] 
The Torelli subgroup $\mathcal{T}_g$ is further described by Kreck \cite{Kr79} through another short exact sequence, which we now recall. Let $\Sigma:\pi_n(\SO(n)) \to \pi_n(\SO(n+1))$ be the map associated to the natural inclusion homomorphism $\SO(n) \hookrightarrow \SO(n+1)$. Building up on results of Haefliger \cite{Haef}, Kreck \cite{Kr79} proved that there is a short exact sequence
\[ 1 \to \Theta_{2n+1} \to \mathcal{T}_g \to \Hom (\mathbb{Z}^{2g},  \Sigma(\pi_n(SO(n))))\to 1,\]
where $\Theta_{2n+1}$ is the group of homotopy $(2n+1)$-spheres up to $h$-cobordism, which is a finite abelian group after the work of Kervaire and Milnor \cite{KM}.

Recall that the capping homomorphism $\kappa_b: \Map(W_g^b) \to \Map(W_g)$ is the surjective homomorphism induced by attaching a disk to every boundary component of $W_g^b$. The kernel of $\kappa_b$ is generated by the twists along the boundary spheres of $W_g^b$, by a result of Kreck (see \cite[Lemma 1.1]{Kr20} for a proof. In the special case when $b=1$, a result of Kreck (see also \cite[Lemma 1.1]{Kr20}) $\kappa_1$ is an isomorphism. By making repeated use of the capping homomorphism, we deduce:  

\begin{lemma}[Inclusion property, I]\label{lem-inc-prop}
Let $W_{g}^{b+b'}$ be obtained from $W_{g}^{b}$ by attaching $W_0^{b'+2}$ along one of its boundary components, where  $b, b'\geq 1$. Then, the associated homomorphism  $\Map(W_{g}^{b})\to \Map(W_{g}^{b+b'})$ is injective.
\end{lemma}
\proof[Sketch of proof] 
It suffices to prove the case $b'=2$ as the general case follows by induction. In this case, the lemma follows from the observation that the composition of inclusion map $W_{g}^{b}\to W_{g}^{b+1}$ and the capping map (i.e. attaching a disk to a newly created boundary) $W_{g}^{b+1}\to W_{g}^{b}$  induces identity map on mapping class groups. 
\qed

After the above description, we have the following high-dimensional analog of Theorem \ref{thm:3d-boundarytwist}:

\begin{theorem}\label{lemm-prod-dtw-trv}
Let $S_1,\cdots,S_b \subset W_g^b$ be embedded spheres parallel to the $b\ge 1$ boundary components of $W_g^b$. Then the product $T_{S_1}\cdots T_{S_b}$ is trivial in $\Map(W_g^b)$, and the subgroup generated by the twists along boundary spheres is isomorphic to $(\Z_2)^{b-1}$. 
\end{theorem}

\begin{proof}
First, assume $b=1$. If $g=0$, the result is trivial. If $g\ge 1 $, it holds because of Kreck's theorem that $\kappa_1$ is an isomorphism. 

Hence we assume $b\ge 2$, so that Corollary \ref{cor-dtw-nt-gen} gives that $T_{S_i}$ is a non-trivial element of order 2 in $\Map(W_g^b)$ since $W_g^b$ is stably parallelizable for any $g$ and $b$. In fact $W_g^0$ bounds the boundary sum of $g$ copies of $S^{n} \times D^{n+1}$ which is parallelizable, hence $W_g^0$ and $W_g^b$  are stably parallelizable.   We claim that $T_{S_1}\cdots T_{S_b}= 1$. If $g=0$, then the claim follows as in \cite[p. 214--215]{HW10}. If $g\ge 1$, let $S$ be a sphere cutting $W_g^b$ into $W_0^{b+1}$ and $W_g^{1}$. Since $T_S = T_{S_1} \cdots T_{S_b}$ in $\Map(W_0^{b+1})$, then the same relation holds in $\Map(W_g^b)$. But as $T_S$ vanishes in $\Map(W_g^1)$, it vanishes in $\Map(W_g^b)$ too. This establishes the claim, and the also result for $b\le 2$. 

Finally, we need to show there is no further relation when $b\geq 3$. Suppose that some partial product $T_{S_{i_1}}\cdots T_{S_{i_k}}$ were equal to the identity in $\Map(W_g^b)$. Let $S'$ be a sphere which cuts off a disk containing exactly $S_{i_1},\ldots,  S_{i_k}$, and write $W_g^{b'}$ for the other component of the complement of $S'$. Note that $b'\geq 2$. Hence $T_{S'}$ is nontrivial in $W_g^{b'}$ by Corollary \ref{cor-dtw-nt-gen}. On the other hand in $\Map(W_g^b)$, we have $T_{S'} =  T_{S_1}\cdots T_{S_k}$ which equals $1$. But by Lemma \ref{lem-inc-prop}, this means $T_{S'}=1$ in $\Map(W_g^{b'})$, which is a contradiction.
\end{proof}

As a consequence, we obtain:  

\begin{corollary}\label{prop-cap-ker-high}
There is a split short exact sequence:
\[ 1 \to (\Z_2)^{b-1} \to \Map(W_g^b) \xrightarrow{\kappa_b} \Map(W_g) \to 1\]
where $(\Z_2)^{b-1}$ is generated by the twists along $b-1$ boundary spheres. In particular, we have  $\Map(W_g^b) \cong  \Map(W_g) \times (\Z_2)^{b-1} $.
\end{corollary}
\begin{proof}
After Theorem \ref{lemm-prod-dtw-trv}, we only need to show that $\kappa_b$ splits. To this end, first recall that $\Map(W_g)\cong \Map(W_g^1)$. Now, realize $W_g^1$ inside $W_g^b$ as a connected component of the complement of a separating sphere that cuts off the $b$ boundary spheres of $W_g^b$. In this way, we obtain a map $\Map(W_g^1) \to \Map(W_g^b)$ which provides the desired splitting.

\end{proof}

The general case of the inclusion property now follows immediately from Lemma \ref{lem-inc-prop}, Corollary \ref{prop-cap-ker-high} and Kreck's calculation of $\Map(W_g)$:

\begin{theorem}[Inclusion property, II]\label{thm-inc-prop-h}
Let $W_{g_2}^{b_2}$ be obtained from $W_{g_1}^{b_1}$ by attaching a copy of $W_h^n$ ($n\geq 2$) along a boundary component. Then the associated homomorphism  $\Map(W_{g_1}^{b_1})\to \Map(W_{g_2}^{b_2})$ is injective. 
\end{theorem}

Finally, we settle the intersection property for the class of manifolds $W_g^b$. For convenience, let $N= W_{g}^b$ with $b\geq 2$, $h\in\{0,1\}$ and $M$ the result of gluing two copies of $W_{h}^{d+1}$ along two distinct boundary spheres of $N$.  Write $M_1$ and $M_2$ for the submanifolds of $M$ resulting from attaching to $N$ one copy of $W_{1}^{d+1}$. Again, we use the inclusion property to regard $\Map(N), \Map(M_1),\Map(M_2)$ as subgroups of $\Map(M)$. We have: 

\begin{theorem} \label{thm-map-sur-high}
 $\Map(N)= \Map (M_1)\cap \Map(M_2)$.
\end{theorem}

\begin{proof}
As in the previous subsections,  it suffices to prove $\iota$ is surjective. We have the following composition of maps:
\[ \Map (N) \xrightarrow[]{\iota}  \Map(M_1)\cap \Map(M_2) \xrightarrow[]{\kappa_b} \kappa(\Map(M_1)) \cap \kappa(\Map(M_2))\]
We know that $\kappa_b$ is surjective by Proposition \ref{prop-cap-ker-high}; similarly,  $\kappa_b\circ \iota$ is surjective by Kreck's description of the mapping class group, as described above. At this point, a direct calculation based on Theorem \ref{lemm-prod-dtw-trv} and Corollary \ref{prop-cap-ker-high} shows that the kernel of $\kappa_b$ lies in the image of $\iota$. Hence we are done. 
\end{proof}

\section{The handle complex in high dimensions\\by Oscar Randal-Williams} 
\label{sec:piececomplex_highdim}
The purpose of this section is to study the handle complex in high dimensions, and in particular prove Theorem \ref{thm:RW} in Section \ref{sec:highdim}. 
Following the notation therein, $O$ is a compact, simply-connected orientable manifold of even dimension $2n\ge 6$. Write $W_g$ for the connected sum of $O$ and $g$ copies of $\mathbb S^n \times
\mathbb S^n$, where $n\ge 3$, and $W_g^b$ for the result of removing
$b\ge 1$ disjoint $2n$-dimensional open balls from
$W_g$.

Let $\mH(W_g^b)$ be the handle complex. Recall that a handle is a manifold diffeomorphic to $H:=\bS^n\times \bS^n \setminus \mathbb{B} ^{2n}$; its boundary $\partial H$ is then identified with $S^{2n-1}$. The simplicial complex $\mH(W_g^b)$ has vertices given by isotopy classes of smoothly embedded separating $(2n-1)$-spheres in $W_g^b$ which bound, to one side, a manifold diffeomorphic to $H$. A collection of $k+1$ distinct vertices span a simplex if they can be realized disjointly. The following lemma shows that it is equivalent to ask for the spheres to be realized disjointly, or the handles which they bound to be realized disjointly.

\begin{lemma}\label{lem:TwoDefsAreEqual}
If a collection of disjoint $(2n-1)$-spheres in $W_g^b$ each bound a handle, and no two are isotopic, then they bound disjoint handles.
\end{lemma}
\begin{proof}
Let $\{e_i : H \hookrightarrow W_g^b\}_{i=0}^k$ be a collection of embeddings such that the submanifolds $\{e_i(\partial H)\}_{i=0}^k$ are disjoint and distinct up to isotopy. 
We claim that then the $e_i$ have mutually disjoint images. 

If the images of $e_i$ and $e_j$ intersect then as their boundaries are disjoint the image of one is contained in the image of the other: suppose $e_i(H) \subset e_j(H)$, so that $e_j^{-1} \circ e_i : H \hookrightarrow H$ is an embedding. Then by Lemma \ref{lem:SelfEmbIsDiffeo} below this self-embedding is isotopic to a diffeomorphism, and so $e_i$ is isotopic to an embedding with the same image as $e_j$, and hence the spheres $e_i(\partial H)$ and $e_j(\partial H)$ are isotopic. But we supposed that such spheres were not isotopic, giving a contradiction.
\end{proof}

\begin{lemma}\label{lem:SelfEmbIsDiffeo}
Any embedding $e : H \hookrightarrow H$ is isotopic to a diffeomorphism.
\end{lemma}
\begin{proof}
First note that the map $e_* : H_n(H;\mathbb Z) \to H_n(H;\mathbb Z)$ induced by such an embedding is an isomorphism, as it respects the intersection form and the intersection form of $H$ is non-degenerate.

Given such an embedding, by pushing inwards along a collar we can suppose that it has image in the interior of $H$. Then the difference $K := H \setminus \mathrm{int}(e(H))$ is a cobordism from $e(\partial H)$ to $\partial H$. It is easy to see using the Seifert--van Kampen theorem that $K$ is simply-connected, and by excision we have $H_*(K, e(\partial H);\mathbb Z) \cong H_*(H, e(H);\mathbb Z)$ which vanishes by the discussion above. Thus $K$ is a simply-connected $h$-cobordism, so by the $h$-cobordism theorem \cite{MilCobnor} there is a diffeomorphism $K \cong [0,1] \times e(\partial H)$ relative to $e(\partial H)$. Using this product structure, and a collar of $H$, we can find an isotopy from $e$ to a diffeomorphism.
\end{proof}


In view of this discussion, we adopt the view that simplices of $\mH(W_g^b)$ are tuples of vertices which can be realised to bound disjoint handles. The purpose of this Appendix is to prove: 

\begin{theorem}\label{thm:pCx}
$\mH(W_g^b)$ is $\left\lfloor\frac{g-4}{2}\right\rfloor$-connected.
\end{theorem}

Let $\mH'(W_g^b)$ be simplicial complex whose vertices are isotopy classes of embeddings $e : H \to W_g^b$. Its $k$-simplices are sets of $k+1$ vertices which can be represented disjointly in $W_g^b$. There is a map of simplicial complexes
$$\psi : \mH'(W_g^b) \lra \mH(W_g^b)$$
given on vertices by sending an isotopy class of embeddings $[e : H \to W_g^b]$ to the isotopy class of submanifolds $[e(\partial H) \subset W_g^b]$; this clearly sends simplices to simplices.

\begin{lemma}\label{lem:PsiHasSection}
The map $\psi$ has a section.
\end{lemma}
\begin{proof}
For each vertex $v \in \mH(W_g^b)$ we choose a representative $v=[S \subset W_g^b]$. As $S$ bounds a manifold diffeomorphic to $H$, choosing such a diffeomorphism gives an embedding $e : H \to W_g^b$ with $[e(\partial H) \subset W_g^b]=[S \subset W_g^b] = v$. We make such a choice for each vertex $v$, and attempt to define a simplicial map by $\bar{\psi}(v) := [e : H \to W_g^b]$; if this does define a simplicial map then it will be a section as required. 

To see that $\bar{\psi}$ defines a simplicial map, it suffices to show that if $\{e_i\}_{i=0}^k$ is a collection of embeddings $H \hookrightarrow W_g^b$ such that the submanifolds $\{e_i(\partial H)\}_{i=0}^k$ are distinct up to isotopy and may be isotoped to be mutually disjoint, then the $e_i$ may be isotoped to be mutually disjoint. By applying isotopy extension to the isotopies which make the $e_i(\partial H)$ mutually disjoint, we may suppose that the maps $\partial e_i$ have mutually disjoint images. But then by Lemma \ref{lem:TwoDefsAreEqual} the $e_i$ have mutually disjoint images. 
\end{proof}

\begin{lemma}\label{lem:flag}
$\mH(W_g^b)$ and $\mH'(W_g^b)$ are flag complexes. 
\end{lemma}
\begin{proof}
We first show that $\mH'(W_g^b)$ is a flag complex. Let $\{[e_i]\}_{i=0}^k$ be a set of vertices of $\mH'(W_g^b)$ which pairwise span 1-simplices, i.e.\ each pair $e_i$ and $e_j$ may be made disjoint up to isotopy. It suffices to show that if $e_0, \ldots, e_{i-1}$ are disjoint embeddings and $e_i$ is disjoint up to isotopy from each of them, then it may be changed by an isotopy to be simultaneously disjoint from all of them. This follows by an application of the Whitney trick. 

In more detail, let us construct $H = \bS^n\times \bS^n \setminus \mathbb{B} ^{2n}$ by taking the ball $\mathbb{B} ^{2n}$ to lie inside the product $\mathbb{D}^n_+ \times \mathbb{D}^n_+$ of the two upper hemispheres. We call the submanifolds $\bS^n \times \mathbb{D}^n_-$ and $\mathbb{D}^n_- \times \bS^n$ of $H$ the \emph{thickened cores} of $H$, and the submanifolds $\bS^n \times \{*\}$ and $\{*\} \times \bS^n$ given by the centre of $\mathbb{D}^n_-$ the \emph{cores} of $H$. 
As $e_i$ is disjoint up to isotopy from each $e_0, \ldots, e_{i-1}$, the algebraic intersection number of the cores of $e_i(H)$ with those of each $e_j(H)$ are zero. As $W_g^b$ is simply-connected we may therefore use the Whitney trick \cite[Theorem 6.6]{MilCobnor} to isotope $e_i$ so that the cores of $e_i(H)$ are disjoint from those of each $e_j(H)$. Now $H$ may be isotoped into an arbitrarily small neighborhood of its cores, so we may isotope $e_i$ so that its image is disjoint from the cores of each $e_j(H)$. Finally, using an isotopy of each $e_j(H)$ to a small neighborhood of its cores, and using isotopy extension, we may find an isotopy of $e_i$ to an embeddings which misses the $e_j(H)$ as required.

To see that $\mH(W_g^b)$ is also a flag complex, we use Lemma \ref{lem:PsiHasSection}. Let $v_0, \ldots, v_k \in \mH(W_g^b)$ be a set of distinct vertices such that each pair spans a 1-simplex. Using the section $\bar{\psi}$ from Lemma \ref{lem:PsiHasSection} it follows that $\bar{\psi}(v_0), \ldots, \bar{\psi}(v_k) \in \mH'(W_g^b)$ is a set of distinct vertices in which each pair spans a 1-simplex, and as we have seen above $\mH'(W_g^b)$ is a flag complex so this collection of vertices spans a $k$-simplex. Applying $\psi$ shows that $v_0, \ldots, v_k$ spans a $k$-simplex in $\mH(W_g^b)$ as required.
\end{proof}

\begin{proof}[Proof of Theorem \ref{thm:pCx}]
As $\psi$ has a section, it suffices to prove that $\mH'(W_g^b)$ is $\left\lfloor\frac{g-4}{2}\right\rfloor$-connected, and to show this we follow \cite[Section 5]{GRW18}. We first construct an algebraic avatar of the simplicial complex $\mH'(W_g^b)$ as follows. Let $\mathcal{I}_n^\text{fr}(W_g^b)$ denote the set of regular homotopy classes of immersions $i : \bS^n \times \mathbb D^n \looparrowright W_g^b$. Choose once and for all a framing of $\bS^n \times \mathbb D^n$. Assigning to such an immersion $i$ the image of this framing under $Di$ at each point of $\mathbb{S}^n \times \{*\}$ gives a function
$$\mathcal{I}_n^\text{fr}(W_g^b) \lra [\bS^n, \mathrm{Fr}(W_g^b)]$$
to the set of homotopy classes of maps from $\bS^n$ to $\mathrm{Fr}(W_g^b)$. The Hirsch--Smale theory of immersions \cite[Section 5]{Hirsch} says that this function is a bijection. 
As $W_g^b$ is simply-connected, the frame bundle $\mathrm{Fr}(W_g^b)$ is simple,
and so based and unbased homotopy classes of maps to this space agree: thus $\mathcal{I}_n^\text{fr}(W_g^b)$ agrees with \cite[Definition 5.2]{GRW18} and in particular has the structure of an abelian group. This abelian group is equipped with a $(-1)^n$-symmetric bilinear form $\lambda: \mathcal{I}_n^\text{fr}(W_g^b) \otimes \mathcal{I}_n^\text{fr}(W_g^b) \to \mathbb{Z}$ given by the algebraic intersection number, and with a function 
$$\mu : \mathcal{I}_n^\text{fr}(W_g^b) \lra \begin{cases}
\mathbb Z & n \text{ even}\\
\mathbb Z_2 & n \text{ odd}
\end{cases} $$
given by counting self-intersections; $\mu$ is a quadratic refinement of $\lambda$, and this data forms a quadratic module in the sense of \cite[Section 3]{GRW18}.

Let $\mathsf{H}$ denote the hyperbolic quadratic module, i.e.\ $\mathbb{Z}\{e,f\}$ with $\lambda$ given in this basis by $\left(\begin{smallmatrix} 0 & 1 \\ (-1)^n & 0\end{smallmatrix}\right)$ and $\mu$ determined by $\mu(e)=\mu(f)=0$ and the quadratic property. For a quadratic module $\mathsf{M}$ let $Q(\mathsf{M})$ denote the simplicial complex having vertices the morphisms of quadratic modules $\mathsf{H} \to \mathsf{M}$ (which are automatically injective, as $\mathsf{H}$ is non-degenerate), and where a collection of such morphisms spans a simplex if their images are mutually orthogonal. As orthogonality of submodules can be tested pairwise, this is a flag complex.

There is a map of simplicial complexes
$$\phi: \mH'(W_g^b) \lra Q(\mathcal{I}_n^\text{fr}(W_g^b), \lambda, \mu)$$
given as follows. If $[s : H \hookrightarrow W_g^b]$ is an isotopy class of embedding, then restricting this embedding to the two thickened cores $e, f : \bS^n \times \mathbb D^n \hookrightarrow H$ determines a hyperbolic pair $[s \circ e], [s \circ f] \in \mathcal{I}_n^\text{fr}(W_g^b)$ and hence a map $s_\text{alg} : \mathsf{H} \to \mathcal{I}_n^\text{fr}(W_g^b)$, and we declare $\phi([s : H \hookrightarrow W_g^b])$ to be $s_\text{alg}$. If we change $s$ by an isotopy then the regular homotopy classes $[s \circ e]$ and $[s \circ f]$ do not change, so $\phi$ is well-defined on vertices. If a pair of embeddings $s$ and $s'$ have disjoint images up to isotopy then the submodules $s_\text{alg}(\mathsf{H})$ and $s'_\text{alg}(\mathsf{H})$ of $\mathcal{I}_n^\text{fr}(W_g^b)$ are clearly orthogonal: as both are flag complexes (using Lemma \ref{lem:flag}), it follows that $\phi$ is a simplicial map.

Now $\mathcal{I}_n^\text{fr}(W_g^b)$ contains $\mathsf{H}^{\oplus g}$ as a quadratic submodule, using the $g$ disjoint copies of $H$, so by \cite[Theorem 3.2]{GRW18} the complex $Q(\mathcal{I}_n^\text{fr}(W_g^b), \lambda, \mu)$ is $\left\lfloor\frac{g-4}{2}\right\rfloor$-connected, and is locally $wCM$ of dimension $\geq \lfloor \tfrac{g-1}{2} \rfloor$. One can then repeat the argument of \cite[Lemma 5.5]{GRW18} on the map $\phi$. We note that there is a mild difference in that
\begin{enumerate}
\item we are considering isotopy classes of embeddings, rather than actual embeddings, and

\item our embeddings do not come with a tether to the boundary;
\end{enumerate}
however these two differences simplify rather than complicate the argument.
\end{proof}

\bibliographystyle{amsplain}

\end{document}